\documentclass{amsart}
\usepackage{style}

\begin{document}
\setlength{\perspective}{2pt}
\title[Measures on Cameron's treelike classes and tensor categories]{Measures on Cameron's treelike classes and applications to tensor categories}
\author{Thanh Can}
\address{Jackson-Reed High School, Washington, D.C., USA}
\email{chithanh08dgs@gmail.com}
\author{Thomas R\"ud}
\address{Massachusetts Institute of Technology\\Cambridge, MA, USA}
\email{rud@mit.edu}
\date{\today}
\setcounter{tocdepth}{1}
\begin{abstract}
   Measures on Fra\"iss\'e classes are a key input in the Harman--Snowden (2022) construction of tensor categories. Treelike Fra\"iss\'e classes provide a particularly tractable source of examples. In this paper, we complete the classification of measures on Cameron's elementary treelike classes. In particular, for the class $\partial \mathfrak{T}_3(n)$ of node-colored rooted binary tree structures with $n$ colors, we classify measures by an explicit bijection with directed rooted trees edge-labeled by $\{1, \dots, n\}$ with a distinguished vertex, yielding $(2n+2)^n$ distinct $\ZZ\left[\frac 12\right]$-valued measures. For each $n \geq 1$, we use a family of measures $\mu_n^I$ and their supports $\partial \mathfrak{T}_3(n)^{\mathrm{ord}}_I$ (where $I \subseteq \{1, \dots n\}$) to construct the Karoubi envelopes $\mathbf{Rep}(\partial \mathfrak{T}_3(n)^{\mathrm{ord}}_I;\mu^I_n)$, producing infinite families of semisimple tensor categories with superexponential growth that cannot be obtained via Deligne's interpolation of representation categories. We also prove the nonexistence of measures on the $n$-colored tree class $C_n\mathfrak{T}$ for $n \geq 2$ and the labeled tree class $L \mathfrak{T}$, extending Snowden's results for uncolored trees.
\end{abstract}
\maketitle
\tableofcontents
\section{Introduction}
The primary objective of this paper is to construct a large number of new tensor categories through the computation of measures on Fra\"iss\'e classes (or, equivalently, oligomorphic groups). Though the general machinery has been relatively well-developed \cite{harman2024oligomorphicgroupstensorcategories}, computing measures is a notoriously difficult combinatorial problem. In this introduction, we first provide some background on the subject and recall previous work. We then provide an overview of our results and an outline of the paper.
\subsection{Background and motivation} 
For an affine (super)group scheme $G$, consider the category of representations $\mathbf{Rep}(G)$ over an algebraically closed field of characteristic zero. One can check that $\mathbf{Rep}(G)$ is a tensor category (Definition \ref{tenscat}), and Deligne \cite{Del02} shows, in fact, that every tensor category with subexponential growth (Definition \ref{subexp}) is equivalent to $\mathbf{Rep}(G,z)$ where $z$ is a suitable element of order $2$ acting by parity. We are interested in tensor categories with superexponential growth excluded from Deligne's characterization. 

Currently, it is difficult to understand tensor categories in a general setting due to the lack of concrete examples. Earlier work by Deligne \cite{Del02} and Knop \cite{Knop_2006, Knop_2007} show that novel examples of tensor categories with superexponential growth can be constructed by interpolating a sequence of representation categories of finite groups; for instance, Deligne obtains $\mathbf{Rep}(S_t)$ for $t \in \CC$ by interpolating the sequence of categories of representations of finite symmetric groups $\{\mathbf{Rep}(S_n)\}$ for positive integers $n$. 

In Harman and Snowden's recent work \cite{harman2024oligomorphicgroupstensorcategories}, a new way to build tensor categories based on oligomorphic theory emerged. A group $G$ acting on a set $\Omega$ is an \emph{oligomorphic group} if $G$ has finitely many orbits on $\Omega^n$ for all $n \geq 0$. The authors introduced the notion of a measure $\mu$ on oligomorphic groups valued in a field such that if $\mu$ satisfies certain properties, then we can construct the tensor category $\mathbf{Rep}(G;\mu)$ that is even semisimple in some cases. Early results \cite{harman2024oligomorphicgroupstensorcategories} include the tensor category $\mathbf{Rep}(S_{\infty}; \mu_t)$, which remarkably recovers Deligne's interpolation category $\mathbf{Rep}(S_t)$, and the Delannoy category associated with the oligomorphic group $\mathrm{Aut}(\RR,<)$, which is a semisimple tensor category.

Notably, in the same paper, the authors noted the intimate connection between model theory and oligomorphic groups to lay the groundwork for the construction of tensor categories through classes of finite combinatorial structures known as \emph{Fraïssé classes} (Definition \ref{fraisse}). The model-theoretic perspective allows for the reduction of the classification of measures to a purely combinatorial problem, where there is an analogous notion of complex-valued measures (Definition \ref{fraissemeasure}) on embeddings in Fraïssé classes. In this paper, we will focus on this perspective. 

To discuss our results later, we briefly describe the Harman--Snowden construction of tensor categories. Given a Fra\"iss\'e class $\mathfrak{F}$ with a measure $\mu$, we define the category $\mathbf{Perm}(\mathfrak{F}; \mu)$ whose objects are formal direct sums of structures, where $\mathrm{Vec}_X$ corresponds to $X \in \mathfrak{F}$. Morphisms $\mathrm{Vec}_X \to \mathrm{Vec}_Y$ are $\CC$-linear amalgamations (Definition \ref{amalg}) of $X$ and $Y$, and composition involves summing over amalgamations scaled by $\mu$. The category carries a tensor product $\otimes$ where $\mathrm{Vec}_X \otimes \mathrm{Vec}_Y = \bigoplus_Z \mathrm{Vec}_Z$ as $Z$ spans over all amalgamations of $X$ and $Y$. We then let $\mathbf{Rep}(\mathfrak{F}; \mu)$ be the Karoubi envelope of $\mathbf{Perm}(\mathfrak{F}; \mu)$. See Section \ref{category} for the specific details.

Through the above construction, measures on Fra\"iss\'e provide a way to produce entirely new examples of tensor categories: indeed, as long as a Fra\"iss\'e class contains structures of arbitrary (finite) size and admits a regular measure (Definition \ref{reg}), then the tensor category arising from it will have superexponential growth \cite[Remark 8.5]{harman2024oligomorphicgroupstensorcategories}. It is also frequently the case that the tensor categories arising from the Harman--Snowden construction could not be obtained by interpolation (with $\mathbf{Rep}(S_{\infty}; \mu_t)$ as an exception), particularly when the associated oligomorphic group does not resemble finite permutation groups; the groups we work with in this paper indeed do not. Nonetheless, measures remain relatively obscure objects, and only a small number of cases have been resolved \cite{harman2024arborealtensorcategories, snowden2024thirtysevenmeasurespermutations, snowden2023measurescoloredcircle, snowden2023fastgrowingtensorcategories, snowden2024representationtheorysymmetrygroup, kriz_quantum_delannoy_categories, Nekrasov2024TensorialMeasures}. 
\subsection{Overview of results}
To state the main results, we first note that there is the formal notion of a \emph{measure ring} (Definition \ref{thetag}), which is a commutative ring $\Theta(\mathfrak{F})$ such that $\CC$-valued measures correspond to ring homomorphisms $\mu: \Theta(\mathfrak{F}) \to \CC$. Thus, it suffices to compute $\Theta(\mathfrak{F})$ to classify measures on $\mathfrak{F}$. Moreover, to analyze measures on subclasses, we later extend the domain of measures to also include structures, where $\mu(X) = \mu(\varnothing \to X)$ for a structure $X$.

In Section \ref{s2}, we provide necessary background on Fra\"iss\'e classes, oligomorphic groups, measures, and prove some preliminary theoretical results. We define Cameron's treelike classes: the $n$-colored tree structures $C_n \mathfrak{T}$ and their degree-$\leq k$ variants $C_n \mathfrak{T}_k$ where $n \geq 2$ and $k \geq 3$, the labeled tree structures $L \mathfrak{T}$, and the $n$-colored rooted binary tree structures $\partial \mathfrak{T}_3(n)$. These treelike classes generalize the uncolored tree structures studied by Harman, Snowden, and Nekrasov, where the corresponding measures are classified \cite{harman2024arborealtensorcategories}.

In Section \ref{sec3}, we first show that $C_n \mathfrak{T}$ and $C_n \mathfrak{T}_k$ do not admit any measures, illustrating the complicated nature of measures and equations arising from them.
\begin{theorem}[Theorem \ref{firstzero}]
We have
$\Theta(C_n \mathfrak{T}) \cong \Theta(C_n \mathfrak{T}_k) \cong \mathbf{0}.$
\end{theorem}
The same applies to $L \mathfrak{T}$:
\begin{theorem}[Theorem \ref{secondzero}]
    We have $\Theta(L\mathfrak{T}) \cong \mathbf{0}$.
\end{theorem}

Section \ref{ms} presents our main result on measures, in which we classify measures on $\partial \mathfrak{T}_3(n)$ by reducing the computation to finite families of equations.
\begin{theorem}[Theorem \ref{maintheorem}]
    For each integer $n \geq 1$, there is an explicit bijection between measures on $\partial \mathfrak{T}_3(n)$ and directed rooted trees edge-labeled by $\{1, \dots, n\}$ with a distinguished vertex, yielding $\Theta(\partial \mathfrak{T}_3(n)) \cong \ZZ\left[\frac 12 \right]$.
\end{theorem} See the proof of Theorem \ref{maintheorem} for the description of the bijection, along with additional details, and \ref{examp} for a concrete example. 

In Section \ref{sec5}, we further analyze these measures and their supports, classifying subclasses of $\partial \mathfrak{T}_3(n)$ and producing regular measures on these subclasses. We show that subclasses correspond to sets of three-leaf trees (Corollary \ref{class}), enumerate these subclasses (Proposition \ref{enumerate}), and provide an explicit formula for their induced regular measures (Proposition \ref{formula}).  For instance, taking $I = \{1,2,3\}$, we have \begin{center}
\begin{tikzpicture}
  \node (d5)  at (0.5,0.5)   [rectangle,draw, inner sep=1.5pt] {$3$};
  \node (d6)  at (1,0)       [rectangle, draw, inner sep=1.5pt] {$2$};
  \node (d7)  at (1.5,0.5)   [rectangle, draw, inner sep=1.5pt] {$3$};
  \node (d8)  at (1.25,1)    [circle, fill, inner sep=1.5pt] {};
  \node (d9)  at (1.75,1)    [circle, fill, inner sep=1.5pt] {};
  \node (1)   at (0.75,1)    [circle, fill, inner sep = 1.5pt] {};
  \node (2)   at (0.25,1)    [circle, fill, inner sep = 1.5pt] {};

  \draw (d5)--(d6) (d7)--(d6) (d7)--(d8) (d7)--(d9) (1)--(d5) (2)--(d5);

  \path (0.2,-0.25) coordinate (bbSW);
  \path (1.8, 1.75) coordinate (bbNE);

  \node[anchor=east] at (bbSW |-0,0.5)
    {$\left(\vphantom{\rule{0pt}{0.8cm}}\right.$};
  \node[anchor=west] at (bbNE |- 0,0.5)
    {$\left.\vphantom{\rule{0pt}{0.8cm}}\right)$};

  \node[anchor=east] at (-0.2,0.5) {$\mu_3^I$};

  \node[anchor=west] at (2.3,0.5)
    {$=-\dfrac{1}{2^7}$,};

  \begin{scope}[yshift=-0.25cm]
    \node (d5)  at (5,0.5)   [circle,fill, inner sep=1.5pt] {};
    \node (d6)  at (5.5,0)   [rectangle, draw, inner sep=1.5pt] {$1$};
    \node (d7)  at (6,0.5)   [rectangle, draw, inner sep=1.5pt] {$2$};
    \node (d8)  at (5.5,1)   [circle, fill, inner sep=1.5pt] {};
    \node (d9)  at (6.5,1)   [rectangle, draw, inner sep=1.5pt] {$3$};
    \node (d10) at (6,1.5)   [circle, fill, inner sep=1.5pt] {};
    \node (d11) at (7,1.5)   [circle, fill, inner sep=1.5pt] {};

    \draw (d5)--(d6) (d7)--(d6) (d7)--(d8) (d7)--(d9)
          (d9)--(d10) (d9)--(d11);

    \path (5,-0.25) coordinate (bbSW);
    \path (7, 1.75) coordinate (bbNE);

\coordinate (midB) at (bbSW |- 5,0.75);

\node[anchor=east] at ([xshift=-1.4em]midB) {$\mu_3^I$};

\node[anchor=east] at (midB)
  {$\left(\vphantom{\rule{0pt}{1cm}}\right.$};
\node[anchor=west] at (bbNE |- 5,0.75)
  {$\left.\vphantom{\rule{0pt}{1cm}}\right)$};

\node[anchor=west] at ([xshift=1.2em]bbNE |- 5,0.75)
  {$=-\dfrac{1}{2^9}$.};
  \end{scope}
\end{tikzpicture}
\end{center}

Finally, in Section \ref{category}, we give a semisimplicity criterion and construct the families of tensor categories associated with the above regular measures on certain subclasses of $\partial \mathfrak{T}_3(n)$.
\begin{theorem}[Theorem \ref{main}]
    For each integer $n \geq 1$ and subset $I \subseteq [n]$, the category $\mathbf{Rep}(\partial \mathfrak{T}_3(n)^{\mathrm{ord}}_I;\mu^I_n)$ is a semisimple tensor category.
\end{theorem}
 Here, $\partial \mathfrak{T}_3(n)^{\mathrm{ord}}_I$ is the subclass of $\partial \mathfrak{T}_3(n)$ where colors are in increasing order along every root-leaf path and colors in the index set $I$ are allowed to repeat along any path.
\section*{Acknowledgements}
The authors would like to thank Andrew Snowden for kindly suggesting this project and several helpful discussions, in particular, for providing an outline of the semisimplicity argument. We also thank Arun Kannan for providing feedback on a draft version of this paper. Finally, we thank the organizers of MIT PRIMES-USA, under which this project was realized.
\section{Measures on Fra\"iss\'e classes} \label{s2}
In this section, we define Fra\"iss\'e classes, measures on them, and Cameron's treelike classes, showing that these classes are Fra\"iss\'e. We connect Fra\"iss\'e classes to oligomorphic groups to prove some theoretical results about measures. Additionally, we define measure rings and provide insights on how measures can be computed.
\subsection{Fra\"iss\'e classes and oligomorphic groups}\label{fraise} We begin by establishing key model-theoretic definitions and relating Fraïssé classes to oligomorphic groups.
\begin{defn}
    A \emph{structure} is a set $X$ equipped with relations $\{R_i\}_{i \in I}$ where $R_i$ is defined on $X^{n_i}$ for a positive integer $n_i$. A \emph{substructure} $Y$ of $X$ is a subset of $X$ where the relations are restricted to elements of $Y$.
\end{defn}
Two structures are isomorphic if there is a bijection between them that preserves relations. 
\begin{defn} \label{amalg}
    Let $i : Z \to X$ and $j : Z \to Y$ be embeddings of $Z$ into $X$ and $Y$, i.e., $X$ and $Y$ share a common substructure (isomorphic to) $Z$. An \emph{amalgamation} of $X$ and $Y$ over $Z$ is a structure $Z'$ and jointly surjective embeddings $i': Y \to Z'$ and $j': X \to Z'$ such that $i' \circ j = j' \circ i$.
\end{defn}
More naturally, $Z'$ is obtained by identifying a subset of $X$ with a subset of $Y$, both of which contain $Z$, while the relations between the remaining elements of $X$ and $Y$ can be freely determined. We include examples in \ref{cameron}.
\begin{defn} \label{fraisse}
    A class of finite structures $\mathfrak{F}$ is called a \emph{Fraïssé class} if it satisfies the following properties:
    \begin{enumerate}[label=(\alph*),ref=\ref{fraisse}\alph*]
\item If $X$ is in $\mathfrak{F}$, then any structure isomorphic to $X$ is in $\mathfrak{F}$. \label{fa}
\item If $X$ is in $\mathfrak{F}$, then any substructure of $X$ is in $\mathfrak{F}$. \label{fb}
\item For each $n \geq 0$, $\mathfrak{F}$ has only finitely many $n$-element isomorphism classes. \label{fc}
\item For each $X,Y \in \mathfrak{F}$, there is an amalgamation $Z'$ of $X$ and $Y$ over each common substructure $Z \in \mathfrak{F}$ such that $Z' \in \mathfrak{F}$. \label{fd}
    \end{enumerate}
Condition \ref{fd} is known as the amalgamation property of Fraïssé classes.
\end{defn}
A remarkable theorem by Fraïssé \cite{Fraisse1953} guarantees that any Fraïssé class $\mathfrak{F}$ admits a countable structure $\Omega$ called the \emph{Fraïssé limit} of $\mathfrak{F}$ (that is unique up to isomorphism) such that every structure in $\mathfrak{F}$ is a finite substructure of $\Omega$. Additionally, $\Omega$ is \emph{homogeneous} i.e. for any isomorphism of substructures $i : X \to Y$ in $\Omega$, there is an automorphism $\sigma$ of $\Omega$ extending $i$, meaning $\sigma(x) = i(x)$ for all $x \in X$. As a result, two finite substructures of $\Omega$ are isomorphic if and only if they belong to the same orbit under the natural action of $G=\mathrm{Aut}(\Omega)$ on $\Omega$. It follows that the orbits of the $n$-element subsets of $\Omega$ under the $G$-action are in bijection with isomorphism classes of $n$-element structures in $\mathfrak{F}$, which are finite by \ref{fc} and have the same finitary condition as $\Omega^n$. Hence $G$ has finitely many orbits on $\Omega^n$ for all $n \geq 0$. We say $G$ is an \emph{oligomorphic group} with its respective action on $\Omega$, and a $G$-set $X$ is \emph{finitary} if the corresponding $G$-action has finitely many orbits. 

Following \cite{harman2024oligomorphicgroupstensorcategories}, we endow $G$ with a natural topology as follows: for each $A \subseteq \Omega$, let $G(A)$ be the subgroup of $G$ fixing each element of $A$. The set of $G(A)$ as $A$ ranges over all finite subsets of $\Omega$ then forms a neighborhood basis of $1_G$. Let $\mathscr{E}$ be the set of open subgroups of the form $G(A)$ for some finite subset $A$ of $\Omega$.

We say that an action of $G$ on $\Omega$ is \emph{$\mathscr{E}$-smooth} if every stabilizer belongs to $\mathscr{E}$ (and is thus open), and a $G$-set $X$ is $\mathscr{E}$-smooth if every pointwise stabilizer of $X$ belongs to $\mathscr{E}$. Let $(G, \mathscr{E})$ be the class of transitive $\mathscr{E}$-smooth $G$-sets.
\subsection{Relative Fra\"iss\'e classes}
Let $\mathfrak{F}$ be a Fra\"iss\'e class. To state some results appearing later in this section, we define relative Fra\"iss\'e classes, which depend on a base structure. 
\begin{defn}
    For a fixed structure $A \in \mathfrak{F}$, a \emph{relative Fra\"iss\'e class} is a class $\mathscr{F}_A$ of ordered pairs $(X,i)$ where $i: A \to X$ is an embedding in $\mathfrak{F}$. 
\end{defn}
Note that $\mathfrak{F}_{\varnothing} = \mathfrak{F}$. We define embeddings in $\mathscr{F}_A$ as follows. Given pairs $(X,i)$ and $(Y,j)$, an embedding $k: (X,i) \to (Y,j)$ in $\mathscr{F}_A$ is an embedding $k: X \to Y$ in $\mathfrak{F}$ such that $j = k \circ i$. 
\begin{prop} \label{small}
    For every structure $A \in \mathfrak{F}$, the smallest Fra\"iss\'e class containing $\mathscr{F}_A$ is $\mathfrak{F}$.
\end{prop}
\begin{proof}
    Let $\mathfrak{F}'$ be a Fra\"iss\'e class containing $\mathscr{F}_A$. For any $X \in \mathfrak{F}$, the amalgamation of $X$ and $A$ over any substructure in $\mathfrak{F}$ belongs to $\mathscr{F}_A$ by \ref{fd}. It follows from \ref{fb} that $X \in \mathfrak{F}'$. This means that $\mathfrak{F} \subseteq \mathfrak{F}'$, so $\mathfrak{F}$ is the smallest desired class.
\end{proof}

\begin{remark}
    Note that the notion of relative Fra\"iss\'e class is somewhat unconventional and is not needed. Indeed, one can check that any such class $\mathscr{F}_A$ is equivalent to the Fra\"iss\'e class $\mathfrak{F}$ obtained by mapping every element $X\in\mathscr{F}_A$ to $X\setminus i(A)$ and adding $(|A|+1)^n-|A|^n-1$ relations for each $n$-ary relation $R$ to encode lost information involving elements of $i(A)$ and $X \setminus i(A)$. However, this notion is useful for clarity when discussing open subgroups of some oligomorphic group and the corresponding subclasses. 
\end{remark}
\subsection{Measures} We now introduce measures on Fra\"iss\'e classes, which are the central objects of this paper and an essential ingredient of the Harman--Snowden construction of tensor categories.
\begin{defn} \label{fraissemeasure}
A \emph{measure} on a Fraïssé class $\mathfrak{F}$ is a $\CC$-valued function $\mu$ on embeddings $i$ in $\mathfrak{F}$ such that the following conditions hold:
\begin{enumerate}[label=(\alph*), ref=\ref{fraissemeasure}\alph*]
    \item If $i$ is an isomorphism then $\mu(i) = 1$. \label{ma}
    \item If $i : X \to Y$ and $j : Y \to Z$ are embeddings, then $\mu(j \circ i) = \mu(j) \mu(i)$. \label{mb}
    \item Let $i : Z \to X$ and $j: Z \to Y$ be embeddings. If $i'_{k} : Y \to Z'_{k}$ are embeddings into all amalgamations $Z'_{k}$ of $X$ and $Y$ over $Z$ for $k = 1, \dots, n$, then $\mu(i) = \sum_{k=1}^n \mu(i'_{k})$. \label{mc}
\end{enumerate}
\end{defn}
Every embedding in a relative Fra\"iss\'e class $\mathscr{F}_A$ is an embedding in a Fra\"iss\'e class $\mathfrak{F}$, so by Proposition \ref{small}, we define measures on a relative Fra\"iss\'e class to be restrictions of measures on the smallest Fra\"iss\'e class containing it. 
 
Let $G$ be the automorphism group of the Fra\"iss\'e limit $\Omega$ of $\mathfrak{F}$. Given a structure $X \in \mathfrak{F}$, let $\Omega^X$ denote the set of embeddings $X \to \Omega$. Every embedding of structures $i: X \to Y$ induces a corresponding $G$-morphism $i^*: \Omega^Y \to \Omega^X$ where the $G$-action is given by $g \cdot i^* = g \circ i$, so $i^*$ is also a finitary $\mathscr{E}$-smooth $G$-set.  

There is an analogous measure $\mu^*$ on $(G, \mathscr{E})$ assigning each map $i^*: \Omega^Y \to \Omega^X$ of finitary $\mathscr{E}$-smooth $G$-sets, where $\Omega^Y$ is transitive, a complex value $\mu^*(i^*)$ satisfying a list of similar axioms \cite[Definition 3.1, p.20]{harman2024oligomorphicgroupstensorcategories} such that $\mu^*$ uniquely satisfies $\mu(i) = \mu^*(i^*)$ \cite[Theorem 6.9, p.45]{harman2024oligomorphicgroupstensorcategories}. We say that $\mu^*$ is the \emph{corresponding measure} of $\mu$. 
\begin{defn} \label{reg}
  Let $\mu$ be a measure on a Fra\"iss\'e class $\mathfrak{F}$. We say that $\mu$ is \emph{regular} if $\mu(i)$ is nonzero for each embedding $i$ in $\mathfrak{F}$.    
\end{defn}

Note that for a more general ring-valued measure, we may instead require measures to be invertible. If $\mu$ is regular, then its corresponding measure $\mu^*$ is also regular on $(G, \mathscr{E})$, i.e., $\mu^*(X)$ is nonzero for all $X \in (G, \mathscr{E})$.
\begin{defn} \label{quasidef}
  Let $\mu$ be a measure on a Fra\"iss\'e class $\mathfrak{F}$ with Fra\"iss\'e limit $\Omega$, and $G = \mathrm{Aut}(\Omega)$. We say that $\mu$ is \emph{quasi-regular} if there is an open subgroup $U$ of $G$ such that the corresponding measure $\mu^*$ is regular on $(U, \mathscr{E})$.
\end{defn}
Every quasi-regular measure is regular by taking $U$ to be $G$.
\begin{prop} \label{stab}
    A measure $\mu$ on $\mathfrak{F}$ (with corresponding measure $\mu^*$) is quasi-regular if and only if there exists a finite structure $A \subset \Omega$ such that $\mu^*|_{\Stab_G(A)}$ is regular where $\Stab_G(A)$ is the subgroup of $G$ fixing each element of $A$.  
\end{prop}
\begin{proof}
    Stabilizers of finite structures in $\mathfrak{F}$ form the neighborhood basis of the identity in $G$, so every open subgroup $U$ of $G$ contains $\Stab_G(A)$ for some finite structure $A$. Since $(\Stab_G(A), \mathscr{E}) \subseteq (U, \mathscr{E})$, if $\mu^*$ is regular on $(U, \mathscr{E})$, then it is also regular on $(\Stab_G(A), \mathscr{E})$.

    On the other hand, since the action of $G$ on $\Omega$ is $\mathscr{E}$-smooth, $\Stab_G(A)$ is an open subgroup of $G$. Consequently, if $\mu^*|_{\Stab_G(A)}$ is regular for some $A$, then $\mu$ is quasi-regular.
\end{proof}
\begin{prop} \label{correspond}
Suppose $A$ is a finite structure in $\mathfrak{F}$, and $\mu$ is a measure on $\mathfrak{F}$ with corresponding measure $\mu^*$. Then $\mu^*|_{\Stab_G(A)}$ is regular if and only if $\mu$ is regular on the relative Fra\"iss\'e class $\mathscr{F}_A$.   
\end{prop}
\begin{proof}
    Let $U = \Stab_G(A)$ and consider $X \in \mathscr{F}_A$. Let $\Omega_A^X$ denote the set of all embeddings $X \to \Omega$ over $A$ i.e. if $i \in \Omega_A^X$, then $i(a)=a$ for all $a \in A$. Since $\Omega$ is homogeneous, the action of $U$ on $\Omega_A^X$ given by $u \cdot i = u(i)$ is transitive. Further, every stabilizer of $\Omega_A^X$ is the stabilizer of $i(X) \subset \Omega$, so $\Omega_A^X$ is $\mathscr{E}$-smooth. As a result, for any finite structure $B$, we have $$\Orb(B) = \{i(B) : i \in \Omega_A^{B \cup A}\} \cong \Omega_A^{B \cup A},$$ so we obtain the decomposition $$\Omega^{[n]} \cong \bigsqcup_{\substack{|X| = n \\ A \subseteq X}} \Omega_A^X$$ where the right hand side is the disjoint union of transitive $\mathscr{E}$-smooth $U$-sets. Since any transitive $\mathscr{E}$-smooth $U$-set is an orbit on $\Omega^n$ for some $n$, all such sets are of the form $\Omega_A^X$ where $|X| \leq n$.
    
    Now, every map of transitive $\mathscr{E}$-smooth $U$-sets $f:\Omega_A^Y \to \Omega_A^X$ corresponds to an embedding $j: X \to Y$ such that the diagram \[\begin{tikzcd}
	A \\
	X & Y
	\arrow["i"', from=1-1, to=2-1]
	\arrow["{i^*}", from=1-1, to=2-2]
	\arrow["j", from=2-1, to=2-2]
\end{tikzcd}\] commutes. This means that $\mu^*(f)$ is nonzero if and only if $\mu(j)$ is nonzero, but all such embeddings $j$ in $\mathfrak{F}$ are exactly all the embeddings in the corresponding relative class $\mathscr{F}_A$, so $\mu^*|_U$ is regular if and only if $\mu|_{\mathscr{F}_A}$ is regular, as desired.
\end{proof}
The above propositions allow us to translate the quasi-regularity condition to the language of finite combinatorial structures.
\begin{cor} \label{quasicor} A measure $\mu$ on $\mathfrak{F}$ is quasi-regular if and only if there exists $A \in \mathfrak{F}$ such that $\mu(i)$ is nonzero for all possible embeddings $i: A \to X$ in $\mathfrak{F}$. 
\end{cor}
\begin{proof}
    Combining Propositions \ref{stab} and \ref{correspond}, $\mu$ is quasi-regular if and only if $\mu|_{\mathscr{F}_A}$ is regular for some $A \in \mathfrak{F}$. The corollary now follows from the definition of $\mathscr{F}_A$.
\end{proof}
Given a structure $Y \in \mathfrak{F}$ with substructure $X$, there may be multiple possible embeddings $X \to Y$. We show that (quasi-)regular measures behave nicely and ignore how $X$ embeds into $Y$. 
\begin{prop}
    Let $\mu$ be a regular measure on a relative Fra\"iss\'e class $\mathscr{F}_A$ and $K$ be the set of all embeddings $(X,i) \to (Y,j)$ for some $(X,i), (Y,j) \in \mathscr{F}_A$. The value of $\mu(k)$ is constant across all embeddings $k \in K$. 
\end{prop}
\begin{proof}
    For any embedding $k : (X,i) \to (Y,j)$, we have $j = k \circ i$, so $\mu(k) = \mu(j)\mu(i)^{-1}$, which is well-defined and independent of the embedding choice. 
\end{proof}
Due to the complicated nature of measures, we introduce a formal construction that captures all and only the combinatorial constraints of measures.
\begin{defn} \label{thetag}
    For a Fraïssé class $\mathfrak{F}$, we define its \emph{measure ring} to be a commutative ring $\Theta(\mathfrak{F})$ where each embedding $i$ is represented by a formal symbol $[i]$ subject to the following conditions:
\begin{enumerate}[label=(\alph*), ref=\ref{thetag}\alph*]
    \item If $i$ is an isomorphism then $[i] = 1$. \label{thetaa}
    \item If $i,j$ are composable embeddings then $[i \circ j]=[i] \cdot [j]$. \label{thetab}
    \item $[i] = \sum_{k=1}^n [i'_{k}]$ with the same notation as in Definition \ref{fraissemeasure}. \label{thetac}
\end{enumerate}
\end{defn}
In particular, realizing a measure in $\CC$ amounts to a ring homomorphism $\mu:\Theta(\mathfrak{F}) \to \CC$, so computing measures on $\mathfrak{F}$ is equivalent to computing its associated measure ring.

To further simplify matters, observe that every embedding $i: X \to Y$ is composed of $|Y|-|X|$ one-point extensions. Hence $\Theta(\mathfrak{F})$ is generated by one-point extensions. Adopting the notation of \cite{harman2024arborealtensorcategories}, we introduce the following object.
\begin{defn}
    Let $X$ be a structure in $\mathfrak{F}$. For $a \in X$, the \emph{marked structure} $(X,a)$ is the one-point extension $i: X \setminus a \to X$.
\end{defn} 
For each marked structure $(X,a)$, there is a corresponding class $[X,a] \in \Theta(\mathfrak{F})$. It follows that isomorphism classes of marked structures generate $\Theta(\mathfrak{F})$, so it suffices to consider marked structures and relations between them to compute $\Theta(\mathfrak{F})$. A large number of measure rings can, in fact, be reduced to a finite computation with the following key idea of separated elements.
\begin{defn}
    Two elements $a,b \in X$ are said to be \emph{separated} if $X$ is the unique amalgamation of $X \setminus a$ and $X \setminus b$ over $X \setminus \{a,b\}$. 
\end{defn}
The following proposition allows us to remove separated elements from the marked point until we obtain a relatively simple structure. We include the proof to illustrate the nature of amalgamation diagrams.
\begin{prop}[{\cite[Proposition 2.9, p.8]{harman2024arborealtensorcategories}}] \label{sep}
    Let $(X,a)$ be a marked structure. If $b \in X$ is separated from $a$, then $[X,a] \cong [X \setminus b, a]$ in $\Theta(\mathfrak{F})$.  
\end{prop}
\begin{proof}
    Consider the unique amalgamation diagram: \[\begin{tikzcd}
	{X \setminus b} & X \\
	{X \setminus \{a,b\}} & {X \setminus a}
	\arrow[from=1-1, to=1-2]
	\arrow["i", from=2-1, to=1-1]
	\arrow[from=2-1, to=2-2]
	\arrow["j"', from=2-2, to=1-2]
\end{tikzcd}\]
By Definition \ref{thetac}, we have $[i]=[j]$, which is the desired result.
\end{proof}
\subsection{Cameron's treelike structures} \label{cameron} 
We now give an overview of the classes of treelike structures originally constructed by Cameron \cite{treelikeobjects}. Our investigation of measures on these classes extends previous work by Harman, Snowden, and Nekrasov  \cite{harman2024arborealtensorcategories,Nekrasov2024TensorialMeasures} and exhausts Cameron's list of elementary treelike structures.

Recall that a \emph{tree} is a connected simple graph with no cycles. For a tree $T$, let $L(T)$ denote its set of leaves. Cameron \cite{treelikeobjects} assigned a quartic relation $R$ on $L(T)$ where $R(a,b;c,d)$ is true if and only if the four leaves are distinct and the simple path from $a$ to $b$ does not intersect the simple path from $c$ to $d$. Indeed, the subtree spanned by any four leaves $a,b,c,d$ is one of the following configurations:
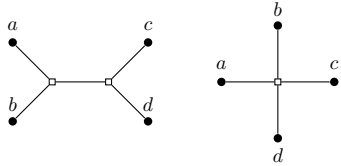
\begin{center}
\begin{tikzpicture}[scale = 0.75, transform shape]

\node (A1) at (0,0) [rectangle, draw, inner sep=1.5pt] {};
\node (B1) at (1,0) [rectangle, draw, inner sep=1.5pt] {};
\node (L1) at (-0.7,0.7) [circle, fill, inner sep=1.5pt] {};
\node (L2) at (-0.7,-0.7) [circle, fill, inner sep=1.5pt] {};
\node (L3) at (1.7,0.7) [circle, fill, inner sep=1.5pt] {};
\node (L4) at (1.7,-0.7) [circle, fill, inner sep=1.5pt] {};
\draw (A1) -- (B1);
\draw (A1) -- (L1);
\draw (A1) -- (L2);
\draw (B1) -- (L3);
\draw (B1) -- (L4);
\node at (-0.7,1) {$a$};
\node at (-0.7,-0.4) {$b$};
\node at (1.7,1) {$c$};
\node at (1.7,-0.4) {$d$};
\node (1) at (3,0) [circle, fill, inner sep=1.5pt] {};
        \node (2) at (4,1) [circle, fill, inner sep=1.5pt] {};
        \node (3) at (5,0) [circle, fill, inner sep=1.5pt] {};
        \node (4) at (4,-1) [circle, fill, inner sep=1.5pt] {};
        \node (5) at (4,0) [rectangle, draw, inner sep=1.5pt] {};

        \draw (1) -- (5);
        \draw (5) -- (2);
        \draw (5) -- (3);
        \draw (5) -- (4);
        \node at (3,0.3) {$a$};
        \node at (4,1.3) {$b$};
        \node at (5,0.3) {$c$};
        \node at (4,-1.3) {$d$};
\end{tikzpicture}
\end{center}
In the first configuration, $R(a,b;c,d)$ is true while $R(a,c;b,d)$ and $R(a,d;b,c)$ are false. In the second configuration, $R$ is always false. We say that a tree $T$ is \emph{reduced} if $T$ does not have any vertices of degree two. Cameron showed that if $T$ is the reduction of $T'$, then $L(T)$ and $L(T')$ (with relation $R$) are isomorphic structures \cite[Proposition 3.1]{treelikeobjects}. This property still holds when we impose more relations on $L(T)$, so we assume throughout this paper that all trees are reduced.  We now introduce the treelike classes we work with in this paper.
\begin{defn}
Given a tree $T$, an \emph{$n$-colored tree structure} is the set of leaves $L(T)$ equipped with the quartic relation $R$ and an $n$-coloring $\sigma: L(T) \to \{1, \dots,n\}$. 
\end{defn}
\begin{example}
The following are examples of $2$-colored trees, whose leaf sets are $2$-colored tree structures:
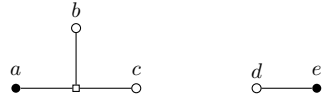
\begin{center}
\begin{tikzpicture}[scale = 0.8, transform shape]
        \node (1) at (0,0) [circle, fill, inner sep=1.5pt] {};
        \node (2) at (1,1) [circle, draw, inner sep=1.5pt] {};
        \node (3) at (2,0) [circle, draw, inner sep=1.5pt] {};
        \node (4) at (1,0) [rectangle, draw, inner sep=1.5pt] {};
        \draw (1) -- (4);
        \draw (4) -- (3);
        \draw (4) -- (2);
        \node at (0,0.3) {$a$};
        \node at (1,1.3) {$b$};
        \node at (2,0.3) {$c$};
        
        \node (4) at (4,0) [circle, draw, inner sep=1.5pt] {};
        \node (5) at (5,0) [circle, fill, inner sep=1.5pt] {};
        \draw (4) -- (5);
        \node at (4,0.3) {$d$};
        \node at (5,0.3) {$e$};
    \end{tikzpicture}
\end{center} For convenience, $\sigma^{-1}(1)$ is represented by black leaves, while $\sigma^{-1}(2)$ is represented by white leaves. In general, we amalgamate (variants of) tree structures by amalgamating their corresponding trees. In particular, if $T$ is an amalgamation of the two above trees over the empty tree, we can only identify leaves with labels in the pairs $(a,e)$, $(b,d)$, and $(c,d)$. There are two possible amalgamations $T$ where $|L(T)|=3$: 
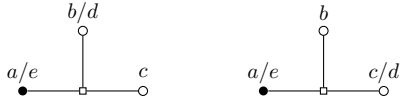
\begin{center}
\begin{tikzpicture}[scale = 0.8, transform shape]
        \node (1) at (0,0) [circle, fill, inner sep=1.5pt] {};
        \node (2) at (1,1) [circle, draw, inner sep=1.5pt] {};
        \node (3) at (2,0) [circle, draw, inner sep=1.5pt] {};
        \node (4) at (1,0) [rectangle, draw, inner sep=1.5pt] {};
        \draw (1) -- (4);
        \draw (4) -- (3);
        \draw (4) -- (2);
        \node at (0,0.3) {$a/e$};
        \node at (1,1.3) {$b/d$};
        \node at (2,0.3) {$c$};

        \node (1) at (4,0) [circle, fill, inner sep=1.5pt] {};
        \node (2) at (5,1) [circle, draw, inner sep=1.5pt] {};
        \node (3) at (6,0) [circle, draw, inner sep=1.5pt] {};
        \node (4) at (5,0) [rectangle, draw, inner sep=1.5pt] {};
        \draw (1) -- (4);
        \draw (4) -- (3);
        \draw (4) -- (2);
        \node at (4,0.3) {$a/e$};
        \node at (5,1.3) {$b$};
        \node at (6,0.3) {$c/d$};
    \end{tikzpicture}
\end{center}
\end{example}
We let $C_n\mathfrak{T}$ and $C_n\mathfrak{T}_k$ be the classes of $n$-colored tree structures and $n$-colored tree structures with vertices of degree at most $k$ (where $k \geq 3$), respectively. 
\begin{defn}
    Given a tree $T$, a \emph{labeled tree structure} is the set of leaves $L(T)$ equipped with the quartic relation $R$ and a total order $<$. 
\end{defn}
The total order on $L(T)$ is a binary relation and induces an order-preserving bijection $\ell: L(T) \to \{1, \dots, |L(T)|\}$, so we assign each leaf a different positive integer in $\{1, \dots, |L(T)|\}$ to represent its order.
\begin{example}
    The following are examples of labeled trees, whose leaf sets are labeled tree structures: 
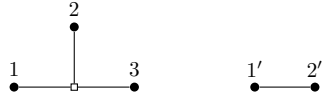
\begin{center}
\begin{tikzpicture}[scale = 0.8, transform shape]
        \node (1) at (0,0) [circle, fill, inner sep=1.5pt] {};
        \node (2) at (1,1) [circle, fill, inner sep=1.5pt] {};
        \node (3) at (2,0) [circle, fill, inner sep=1.5pt] {};
        \node (4) at (1,0) [rectangle, draw, inner sep=1.5pt] {};
        \draw (1) -- (4);
        \draw (4) -- (3);
        \draw (4) -- (2);
        \node at (0,0.3) {1};
        \node at (1,1.3) {2};
        \node at (2,0.3) {3};
        
        \node (4) at (4,0) [circle, fill, inner sep=1.5pt] {};
        \node (5) at (5,0) [circle, fill, inner sep=1.5pt] {};
        \draw (4) -- (5);
        \node at (4,0.3) {$1'$};
        \node at (5,0.3) {$2'$};
    \end{tikzpicture}
\end{center} To distinguish their leaves, we made the relabelings $1 \mapsto 1'$ and $2 \mapsto 2'$ in the right tree. Some valid amalgamations (where the new order is shown by relabelings) include
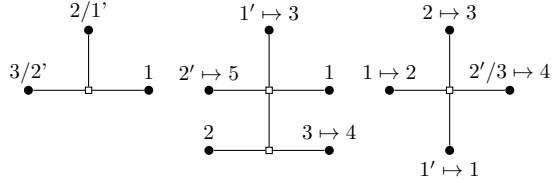
\begin{center}
\begin{tikzpicture}[scale = 0.8, transform shape]
        \node (1) at (0,0) [circle, fill, inner sep=1.5pt] {};
        \node (2) at (1,1) [circle, fill, inner sep=1.5pt] {};
        \node (3) at (2,0) [circle, fill, inner sep=1.5pt] {};
        \node (4) at (1,0) [rectangle, draw, inner sep=1.5pt] {};
        \draw (1) -- (4);
        \draw (4) -- (3);
        \draw (4) -- (2);
        \node at (0,0.3) {3/2'};
        \node at (1,1.3) {2/1'};
        \node at (2,0.3) {1};
        
\node (1) at (3,0) [circle, fill, inner sep=1.5pt] {};
        \node (2) at (4,1) [circle, fill, inner sep=1.5pt] {};
        \node (3) at (5,0) [circle, fill, inner sep=1.5pt] {};
        \node (4) at (4,0) [rectangle, draw, inner sep=1.5pt] {};
        \node (5) at (4,-1) [rectangle, draw, inner sep=1.5pt] {};
        \node (6) at (3,-1) [circle, fill, inner sep=1.5pt] {};
        \node (7) at (5,-1) [circle, fill, inner sep=1.5pt] {};
        \draw (1) -- (4);
        \draw (4) -- (3);
        \draw (4) -- (2);
        \draw (4) -- (5);
        \draw (5) -- (6);
        \draw (5) -- (7);
        \node at (3,0.3) {$2' \mapsto 5$};
        \node at (4,1.3) {$1' \mapsto 3$};
        \node at (5,0.3) {$1$};
        \node at (3,-0.7) {$2$};
        \node at (5,-0.7) {$3 \mapsto 4$};

        \node (1) at (6,0) [circle, fill, inner sep=1.5pt] {};
        \node (2) at (7,1) [circle, fill, inner sep=1.5pt] {};
        \node (3) at (8,0) [circle, fill, inner sep=1.5pt] {};
        \node (4) at (7,0) [rectangle, draw, inner sep=1.5pt] {};
        \node (5) at (7,-1) [circle, fill, inner sep=1.5pt] {};
        \draw (1) -- (4);
        \draw (4) -- (3);
        \draw (4) -- (2);
        \draw (4) -- (5);
        \node at (6,0.3) {$1 \mapsto 2$};
        \node at (7,1.3) {$2 \mapsto 3$};
        \node at (8,0.3) {$2'/3 \mapsto 4$};
        \node at (7,-1.3) {$1' \mapsto 1$};
    \end{tikzpicture}
\end{center}
while the following are non-examples since the original orders are not preserved:
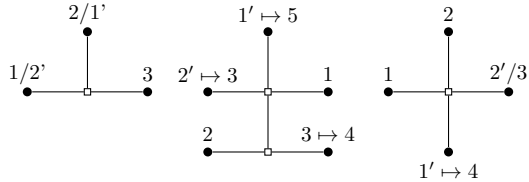
\begin{center}
\begin{tikzpicture}[scale = 0.8, transform shape]
        \node (1) at (0,0) [circle, fill, inner sep=1.5pt] {};
        \node (2) at (1,1) [circle, fill, inner sep=1.5pt] {};
        \node (3) at (2,0) [circle, fill, inner sep=1.5pt] {};
        \node (4) at (1,0) [rectangle, draw, inner sep=1.5pt] {};
        \draw (1) -- (4);
        \draw (4) -- (3);
        \draw (4) -- (2);
        \node at (0,0.3) {1/2'};
        \node at (1,1.3) {2/1'};
        \node at (2,0.3) {3};
        
\node (1) at (3,0) [circle, fill, inner sep=1.5pt] {};
        \node (2) at (4,1) [circle, fill, inner sep=1.5pt] {};
        \node (3) at (5,0) [circle, fill, inner sep=1.5pt] {};
        \node (4) at (4,0) [rectangle, draw, inner sep=1.5pt] {};
        \node (5) at (4,-1) [rectangle, draw, inner sep=1.5pt] {};
        \node (6) at (3,-1) [circle, fill, inner sep=1.5pt] {};
        \node (7) at (5,-1) [circle, fill, inner sep=1.5pt] {};
        \draw (1) -- (4);
        \draw (4) -- (3);
        \draw (4) -- (2);
        \draw (4) -- (5);
        \draw (5) -- (6);
        \draw (5) -- (7);
        \node at (3,0.3) {$2' \mapsto 3$};
        \node at (4,1.3) {$1' \mapsto 5$};
        \node at (5,0.3) {$1$};
        \node at (3,-0.7) {$2$};
        \node at (5,-0.7) {$3 \mapsto 4$};

        \node (1) at (6,0) [circle, fill, inner sep=1.5pt] {};
        \node (2) at (7,1) [circle, fill, inner sep=1.5pt] {};
        \node (3) at (8,0) [circle, fill, inner sep=1.5pt] {};
        \node (4) at (7,0) [rectangle, draw, inner sep=1.5pt] {};
        \node (5) at (7,-1) [circle, fill, inner sep=1.5pt] {};
        \draw (1) -- (4);
        \draw (4) -- (3);
        \draw (4) -- (2);
        \draw (4) -- (5);
        \node at (6,0.3) {$1$};
        \node at (7,1.3) {$2$};
        \node at (8,0.3) {$2'/3$};
        \node at (7,-1.3) {$1' \mapsto 4$};
    \end{tikzpicture}
\end{center}
\end{example}
Let $L\mathfrak{T}$ be the class of labeled tree structures. Note that in practice, $R$ and $<$ can be treated as independent relations on $L(T)$.

To introduce the final variant, let a \emph{node} in a tree be a non-leaf vertex. A \emph{node-colored rooted binary tree} is a rooted tree where each node has two children with no left-right distinction and assigned one of $n$ colors from $\{1, \dots, n \}$. If a tree has at least two leaves, then its root is a node. Notice that, unlike trees in $C_n\mathfrak{T}$, the leaves are not colored. Given a node-colored rooted binary tree $T$ with root $r$, there is a ternary relation $S$ on $L(T)$ where $S(a,b,c)$ is true if and only if $a,b,c$ are distinct and the simple path from the root to $a$ does not intersect the simple path from $b$ to $c$. 
\begin{defn}
   Given a node-colored rooted binary tree $T$, a \emph{node-colored rooted binary tree structure} is the set of leaves $L(T)$ equipped with the ternary relation $S$.
\end{defn}
\begin{example}
    The following are examples of node-colored rooted binary trees, whose leaf sets are node-colored rooted binary tree structures:     
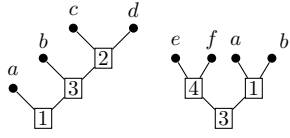
\begin{center}
\begin{tikzpicture}[scale = 0.8, transform shape]
        \node (1) at (-0.5,0.5) [circle, fill, inner sep=1.5pt] {};
        \node (2) at (0,1) [circle, fill, inner sep=1.5pt] {};
        \node (3) at (0.5,1.5) [circle, fill, inner sep=1.5pt] {};
        \node (4) at (0,0) [rectangle,draw, inner sep=1.5pt] {1};
        \node (5) at (0.5,0.5) [rectangle, draw, inner sep=1.5pt] {3};
        \node (6) at (1,1) [rectangle, draw, inner sep=1.5pt] {2};
        \node (7) at (1.5,1.5) [circle, fill, inner sep=1.5pt] {};
        \draw (1) -- (4);
        \draw (4) -- (5);
        \draw (5) -- (2);
        \draw (5) -- (6);
        \draw (6) -- (7);
        \draw (6) -- (3);
        \node at (-0.5,0.8) {$a$};
        \node at (0,1.3) {$b$};
        \node at (0.5,1.8) {$c$};
        \node at (1.5,1.8) {$d$};
    
            \node (B1) at (3,0)     [rectangle, draw, inner sep=1.5pt] {$3$};
    \node (B4) at (2.5,0.5) [rectangle, draw, inner sep=1.5pt] {$4$};
    \node (B5) at (3.5,0.5) [rectangle, draw, inner sep=1.5pt] {$1$};

    \node (B7) at (2.2,1.0) [circle, fill, inner sep=1.5pt] {}; 
    \node (B8) at (2.8,1.0) [circle, fill, inner sep=1.5pt] {}; 
    \node (B2) at (3.2,1.0) [circle, fill, inner sep=1.5pt] {}; 
    \node (B6) at (3.8,1.0) [circle, fill, inner sep=1.5pt] {}; 

    \draw (B1) -- (B4);
    \draw (B1) -- (B5);
    \draw (B4) -- (B7);
    \draw (B4) -- (B8);
    \draw (B5) -- (B2);
    \draw (B5) -- (B6);

    \node at (2.2,1.3) {$e$};
    \node at (2.8,1.3) {$f$};
    \node at (3.2,1.3) {$a$};
    \node at (4,1.3) {$b$};

    \end{tikzpicture}
\end{center} In the left tree, the root has color $1$, and $S(a,b,c)$ and $S(b,c,d)$ are true while $S(c,a,d)$ is false. In the right tree, the root has color $3$, and $S(e,a,b)$ and $S(a,e,f)$ are true while $S(a,b,f)$ is false. An amalgamation of the two trees over the subtree on $\{a,b\}$ is 
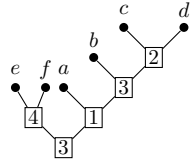
\begin{center}
\begin{tikzpicture}[scale = 0.8, transform shape]
    \node (1)  at (-0.5,0.5)  [circle, fill, inner sep=1.5pt] {};
    \node (2)  at (0,1)       [circle, fill, inner sep=1.5pt] {};
    \node (3)  at (0.5,1.5)   [circle, fill, inner sep=1.5pt] {};
    \node (4)  at (0,0)       [rectangle, draw, inner sep=1.5pt] {1};
    \node (5)  at (0.5,0.5)   [rectangle, draw, inner sep=1.5pt] {3};
    \node (6)  at (1,1)       [rectangle, draw, inner sep=1.5pt] {2};
    \node (7)  at (1.5,1.5)   [circle, fill, inner sep=1.5pt] {};

    \node (8)  at (-0.5,-0.5) [rectangle, draw, inner sep=1.5pt] {3};
    \node (9)  at (-1,0)      [rectangle, draw, inner sep=1.5pt] {4};
    \node (10) at (-1.3,0.5)    [circle, fill, inner sep=1.5pt] {};
    \node (11) at (-0.8,0.5)    [circle, fill, inner sep=1.5pt] {};

    \draw (1) -- (4);
    \draw (4) -- (5);
    \draw (5) -- (2);
    \draw (5) -- (6);
    \draw (6) -- (7);
    \draw (6) -- (3);
    \draw (4) -- (8);
    \draw (8) -- (9);
    \draw (9) -- (10);
    \draw (9) -- (11);

    \node at (-0.5,0.8) {$a$};
    \node at (0,1.3) {$b$};
    \node at (0.5,1.8) {$c$};
    \node at (1.5,1.8) {$d$};
    \node at (-1.3,0.8) {$e$};
    \node at (-0.8,0.8) {$f$};

\end{tikzpicture}
\end{center} 
\end{example}
We let $\partial \mathfrak{T}_3(n)$ denote the class of $n$(-node)-colored rooted binary tree structures. 

To assign measures on Cameron's treelike classes, we need to verify that they are Fra\"iss\'e, which is the following proposition.
\begin{prop}
    The classes $C_n\mathfrak{T}$, $C_n\mathfrak{T}_k$, $L\mathfrak{T}$, and $\partial \mathfrak{T}_3(n)$ are all Fra\"iss\'e classes.
\end{prop}
\begin{proof}
\ref{fa}, \ref{fb}, and \ref{fc} are clear for all four classes, so it suffices to check the amalgamation property. We adapt the proof of \cite[Proposition 3.2, p.158]{treelikeobjects}, which dealt with the uncolored and unlabeled case. Suppose we have trees $T_1$ and $T_2$ with a common subtree $T_0$. Let $\mathrm{deg}(v)$ denote the degree of a vertex $v$. To amalgamate $T_1$ and $T_2$ over $T_0$, we start with $T_0$ and branch out to build the rest of $T_1$ and $T_2$. If the sets of vertices in which $T_1$ and $T_2$ each branch out at are pairwise disjoint, then simply attach the remaining parts of $T_1$ and $T_2$ at these vertices. If there is $v \in T_0$ such that $T_1$ and $T_2$ both branch out at $v$ (possibly leading to $\mathrm{deg}(v) > k$), then attach $v$ to new nodes $u$ and $w$, and branch out $T_1$ at $u$ and $T_2$ at $w$. Notice that $\mathrm{deg}(v) = 3$ (or $2$ if $v$ is the root), and $L(T_1) \cup L(T_2)$ is the same after adding $u$ and $w$, so we obtain a valid amalgamation.

In $L\mathfrak{T}$, we amalgamate the total orders separately from the trees. Given total orders $L_0$, $L_1$, and $L_2$ on $T_0$, $T_1$, and $T_2$, respectively, where $L_0 \subseteq L_1, L_2$, again begin with $L_0$ and add on the remaining elements of $L_1$ and $L_2$ while choosing any ordering among elements of $L_1$ and $L_2$ (such that $L_0$ is preserved). This process is independent of the amalgamation of the underlying unlabeled tree and completes an amalgamation of labeled trees. 
\end{proof}
\section{Vanishing measure rings} \label{sec3}
In this section, we compute the measure rings $\Theta(C_n \mathfrak{T})$, $\Theta(C_n \mathfrak{T}_k)$, and $\Theta(L\mathfrak{T})$, which correspond to the Fra\"iss\'e classes of $n$-colored tree structures, $n$-colored tree structures with maximum degree $k$, and labeled tree structures, respectively. In particular, we show that all three measure rings are isomorphic to the zero ring, implying that none of the classes admits a measure. 
\subsection{\texorpdfstring{$n$-colored trees}{n-colored trees}} \label{n-colored}
For a tree $T$ with leaf $a$, recall that the marked structure $(T,a)$ is the embedding $T \setminus a \to T$, and $[T,a]$ is the corresponding variable for the isomorphism class of $(T,a)$ in the measure ring. Let $N(a)$ denote the distinct node neighboring $a$, if this node exists. 
\begin{prop} \label{31}
    Let $(T,a)$ be a marked structure in $C_n \mathfrak{T}$. If there is $b \in L(T)$ such that $d(N(a),N(b)) \geq 2$ or $\mathrm{deg(N(a))} \geq 4$ or $\mathrm{deg}(N(b)) \geq 4$ (where $d(u,v)$ denotes the geodesic distance between $u$ and $v$), then $a$ and $b$ are separated. In other words, $[T,a] = [T \setminus b, a]$.  
\end{prop}
\begin{proof}
    Given a marked structure $(T,a)$, it suffices to show the proposition for the uncolored version of $T$. If $d(N(a),N(b)) \geq 2$, consider the induced subtree $X$ of $T$ on $a,b$, and all $c \in L(T) \setminus \{a,b\}$ such that the simple path from $a$ to $b$ intersects the simple path from $c$ to either $a$ or $b$. In particular, the picture of $X$ is  
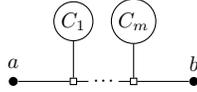
\begin{center}
        \begin{tikzpicture}[scale = 0.8, transform shape]
\node (1) at (0,0) [circle,fill, inner sep=1.5pt] {};
  \node (2) at (1,0) [rectangle, draw, inner sep=1.5pt] {};
  \node (3) at (1,1) [circle, draw, inner sep=1.5pt] {$C_1$};
  \draw (1) -- (2);
  \draw (2) -- (3);
  \draw (2) -- (5/4, 0);
  \node (4) at (11/8,0) {$\cdot$};
  \node (6) at (12/8,0) {$\cdot$};
  \node (7) at (13/8,0) {$\cdot$};
  \node (5) at (2,0) [rectangle, draw, inner sep =1.5pt] {};
  \draw (5) -- (7/4, 0);
  \node (8) at (2,1) [circle, draw, inner sep=1.5pt] {$C_m$};
  \draw (5) -- (8);
  \node (9) at (3,0) [circle, fill, inner sep =1.5pt] {};
  \draw (5)--(9);

  \node at (0,0.3) {$a$};
  \node at (3, 0.3) {$b$};
  \end{tikzpicture}
  \end{center}
where $m>2$ and $c \in C_i$ for some $i \in \{1, \dots, m\}$. Observe that $X$ is the unique amalgamation of $X \setminus a$ and $X \setminus b$ over $X \setminus \{a,b\}$ i.e., $a$ and $b$ are separated in $X$. Let $T'$ be an amalgamation of $T \setminus a$ and $T \setminus b$ over $T \setminus \{a,b\}$; the following diagram commutes:
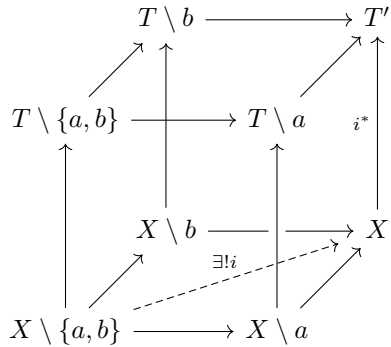
\[
\begin{tikzcd}[row sep={40,between origins}, column sep={40,between origins}]
      &[-\perspective] T \setminus b \ar{rr} &[\perspective] &[-\perspective] T' \\[-\perspective]
    T \setminus \{a,b\} \ar[crossing over]{rr} \ar[ur] & & T \setminus a \ar[ur] \\[\perspective]
      & X \setminus b  \ar{rr} \ar[uu] &&  X \ar[uu, "i^*"] \\[-\perspective]
    X \setminus \{a,b\} \ar{rr} \ar[uu] \ar[ur] && X \setminus a \ar[uu,crossing over] \ar[ur]
    \arrow[dashed, from=4-1, to=3-4, "\exists! i", shorten >=6pt]
\end{tikzcd}
\]
Notice that there is a unique embedding $i: X\setminus \{a,b\} \to X$ for the lower square to commute, so $X$ (along with embeddings $X \setminus a \to X$ and $X \setminus b \to X$) is the pushout of embeddings $X \setminus \{a,b\} \to X \setminus a$ and $X \setminus \{a,b\} \to X \setminus b$. By the universal property, there is a unique embedding $i^*: X \to T'$, which means that $T'$ and $T$ also agree on $X$. Consequently, for any leaves $c,d$ of $T$ (which are also leaves of $T'$ as $L(T)=L(T')$), we have $R(a,b;c,d)=R'(a,b;c,d)$ where $R'$ is the quartic relation on $T'$. It easily follows that $R(a,c;b,d) = R'(a,c;b,d)$ as well. Finally, we are given that $R$ and $R'$ agree on all $4$-tuples not witnessing both $a$ and $b$ since $T$ and $T'$ both preserve $T \setminus a$, $T \setminus b$, and $T \setminus \{a,b\}$, so we conclude that $R = R'$. This means that $T = T'$, so $T$ is the unique amalgamation, implying that $a$ and $b$ are separated in $T$. The result now follows from Proposition \ref{sep}.

The cases $\mathrm{deg}(N(a)) \ge 4$ or $\mathrm{deg}(N(b)) \ge 4$ follow from a similar argument where $a$ and $b$ are separated in $X$ and thus also separated in $T$.
\end{proof}
At this point, we state the main result.
\begin{theorem} \label{firstzero}
    For all $n \geq 2$ and $k \geq 3$, we have $\Theta(C_n\mathfrak{T}) \cong \Theta(C_n\mathfrak{T}_k) \cong \mathbf{0}$. As a result, there are no measures on the classes $C_n\mathfrak{T}$ and $C_n\mathfrak{T}_k$.
\end{theorem}
To show Theorem \ref{firstzero}, we derive equations from the properties of measure rings given by Definition \ref{thetag} and show that these equations eventually contradict one another. We first introduce the following marked structures, where the marked leaf is enlarged, and only two colors (white and black) are featured:
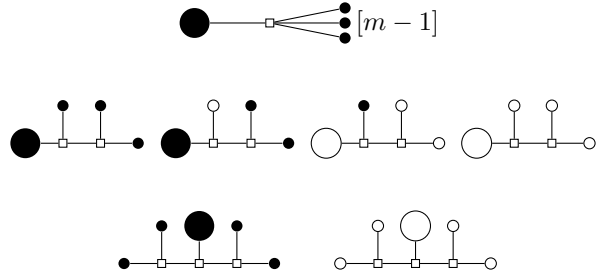
\begin{center}
  \begin{tikzpicture}[baseline=(current bounding box.center)][scale = 0.8, transform shape]

  \begin{scope}[shift={(1.5,0)}]
    \node (0) at (0,0) [circle,fill, inner sep=4pt] {};
    \node (1) at (1,0) [rectangle, draw, inner sep=1.5pt] {};
    \node (2) at (2,0) [circle, fill, inner sep=1.5pt] {};
    \node (3) at (2,0.2) [circle, fill, inner sep=1.5pt] {};
    \node (4) at (2,-0.2) [circle, fill, inner sep=1.5pt] {};
    \draw (1) -- (3);
    \draw (1) -- (4);
    \draw (1) -- (2);
    \draw (1) -- (0);

    \node at (2.7,0) {\([m-1]\)};
  \end{scope}

  \begin{scope}[shift={(1.25,-0.6)}]
    \node (0) at (0,-2/2) [circle,fill, inner sep=4pt] {};
    \node (1) at (1/2,-2/2) [rectangle, draw, inner sep=1.5pt] {};
    \node (2) at (1/2,-1/2) [circle, draw, inner sep = 1.5pt] {};
    \node (3) at (2/2,-2/2) [rectangle, draw, inner sep = 1.5pt] {};
    \node (4) at (2/2,-1/2) [circle, fill, inner sep=1.5pt] {};
    \node (5) at (3/2,-2/2) [circle, fill, inner sep = 1.5pt] {};

    \draw (1) -- (2);
    \draw (1) -- (0);
    \draw (3) -- (4);
    \draw (3) -- (5);
    \draw (1) -- (3);
  \end{scope}

  \begin{scope}[shift={(1.25,-0.6)}]
    \node (0) at (4/2,-2/2) [circle,draw, inner sep=4pt] {};
    \node (1) at (5/2,-2/2) [rectangle, draw, inner sep=1.5pt] {};
    \node (2) at (5/2,-1/2) [circle, fill, inner sep = 1.5pt] {};
    \node (3) at (6/2,-2/2) [rectangle, draw, inner sep = 1.5pt] {};
    \node (4) at (6/2,-1/2) [circle, draw, inner sep=1.5pt] {};
    \node (5) at (7/2,-2/2) [circle, draw, inner sep = 1.5pt] {};

    \draw (1) -- (2);
    \draw (1) -- (0);
    \draw (3) -- (4);
    \draw (3) -- (5);
    \draw (1) -- (3);
  \end{scope}

  \begin{scope}[shift={(1.25,-0.6)}]
    \node (0) at (-4/2,-2/2) [circle,fill, inner sep=4pt] {};
    \node (1) at (-3/2,-2/2) [rectangle, draw, inner sep=1.5pt] {};
    \node (2) at (-3/2,-1/2) [circle, fill, inner sep = 1.5pt] {};
    \node (3) at (-2/2,-2/2) [rectangle, draw, inner sep = 1.5pt] {};
    \node (4) at (-2/2,-1/2) [circle, fill, inner sep=1.5pt] {};
    \node (5) at (-1/2,-2/2) [circle, fill, inner sep = 1.5pt] {};

    \draw (1) -- (2);
    \draw (1) -- (0);
    \draw (3) -- (4);
    \draw (3) -- (5);
    \draw (1) -- (3);
  \end{scope}

  \begin{scope}[shift={(1.25,-0.6)}]
    \node (0) at (8/2,-2/2) [circle,draw, inner sep=4pt] {};
    \node (1) at (9/2,-2/2) [rectangle, draw, inner sep=1.5pt] {};
    \node (2) at (9/2,-1/2) [circle, draw, inner sep = 1.5pt] {};
    \node (3) at (10/2,-2/2) [rectangle, draw, inner sep = 1.5pt] {};
    \node (4) at (10/2,-1/2) [circle, draw, inner sep=1.5pt] {};
    \node (5) at (11/2,-2/2) [circle, draw, inner sep = 1.5pt] {};

    \draw (1) -- (2);
    \draw (1) -- (0);
    \draw (3) -- (4);
    \draw (3) -- (5);
    \draw (1) -- (3);
  \end{scope}

  \begin{scope}[shift={(1.0625,-0.2)}]
    \node (1) at (-1/2,-6/2) [circle,fill, inner sep=1.5pt] {};
    \node (2) at (0,-6/2) [rectangle, draw, inner sep=1.5pt] {};
    \node (3) at (0,-5/2) [circle,fill, inner sep=1.5pt] {};
    \node (4) at (1/2,-6/2) [rectangle, draw, inner sep=1.5pt] {};
    \node (5) at (1/2,-5/2) [circle,fill, inner sep=4pt] {};
    \node (6) at (2/2,-6/2) [rectangle, draw, inner sep = 1.5pt] {};
    \node (7) at (2/2,-5/2) [circle,fill, inner sep=1.5pt] {};
    \node (8) at (3/2,-6/2) [circle,fill, inner sep=1.5pt] {};

    \draw (1) -- (2);
    \draw (2) -- (3);
    \draw (2) -- (4);
    \draw (4) -- (5);
    \draw (4) -- (6);
    \draw (6) -- (7);
    \draw (6) -- (8);
  \end{scope}

  \begin{scope}[shift={(1.4375,-0.2)}]
    \node (1) at (4/2,-6/2) [circle,draw, inner sep=1.5pt] {};
    \node (2) at (5/2,-6/2) [rectangle, draw, inner sep=1.5pt] {};
    \node (3) at (5/2,-5/2) [circle,draw, inner sep=1.5pt] {};
    \node (4) at (6/2,-6/2) [rectangle, draw, inner sep=1.5pt] {};
    \node (5) at (6/2,-5/2) [circle,draw, inner sep=4pt] {};
    \node (6) at (7/2,-6/2) [rectangle, draw, inner sep = 1.5pt] {};
    \node (7) at (7/2,-5/2) [circle,draw, inner sep=1.5pt] {};
    \node (8) at (8/2,-6/2) [circle,draw, inner sep=1.5pt] {};

    \draw (1) -- (2);
    \draw (2) -- (3);
    \draw (2) -- (4);
    \draw (4) -- (5);
    \draw (4) -- (6);
    \draw (6) -- (7);
    \draw (6) -- (8);
  \end{scope}

  \end{tikzpicture}
\end{center}

    Let $x_m$ denote the isomorphism classes of marked structures in the first row, where $m \leq k$ in $C_n \mathfrak{T}_k$. From left to right, let $y_1,y_2,y_3,y_4$ denote the classes of the marked structures in the second row, and $z_1,z_2$ denote the classes in the third row.
\begin{lemma} \label{33}
    The following (class of) marked structure is equal to $x_{4}$ in $\Theta(C_n \mathfrak{T})$ and $\Theta(C_n \mathfrak{T}_k)$ for $k \geq 4$:
    
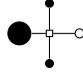
\begin{center}
        \begin{tikzpicture}[scale = 0.8, transform shape]
            \node (0) at (0,0) [circle, fill, inner sep = 4pt] {};
            \node (1) at (1/2,0) [rectangle, draw, inner sep = 1.5pt] {};
            \node (2) at (2/2,0) [circle, draw, inner sep = 1.5pt] {};
            \node (3) at (1/2,1/2) [circle, fill, inner sep = 1.5pt] {};
            \node (4) at (1/2,-1/2) [circle, fill, inner sep = 1.5pt] {};
            \draw (0)--(1);
            \draw (1)--(2); 
            \draw (1)--(3); 
            \draw (1)--(4);
        \end{tikzpicture}
    \end{center}
\end{lemma}
\begin{proof}
    Using Proposition \ref{sep}, we can add then remove separated leaves to obtain the chain of equalities $$\left[\newx\right] = \left[\newex\right]=\left[\newxx\right]=x_{4},$$ as desired.
\end{proof}
For a marked structure $(T,a)$, let $t_{T,a}$ denote the restriction of $T$ to $N(a)$ and its leaf neighbors. For example, $t_{y_1}$ is the class corresponding to the tree structure with a node and two black leaves (which is isomorphic to the tree structure with just two black leaves) and $t_{z_1}$ is isomorphic to the tree structure with one black leaf.
\begin{lemma} \label{34}
    Let $(T,a)$ be a marked structure with at least two nodes and suppose that $\mathrm{deg}(N(a))=3$. If $t_{T,a} \cong t_{y_i}$ for some $i \in \{1,2,3,4\}$, then $[T,a] = y_i$. If $t_{T,a} \cong t_{z_j}$ for some $j \in \{1,2\}$, then $[T,a] = z_j$.
\end{lemma}
\begin{proof}
 We show the result for $i=1$ in each situation; the remaining cases are similar. If $t_{T,a} \cong t_{y_1}$, then $N(a)$ is adjacent to two black leaves (one of which is $a$) and one node, which we denote as $n$. By Proposition \ref{sep}, we can remove leaves from nodes with geodesic distance at least two from $N(a)$ as well as from $n$ until $n$ is adjacent to exactly two leaves of some colors $c_1,c_2$, which is the following picture:  
 \begin{center}
     \begin{tikzpicture}[scale = 0.8, transform shape]
         \node (0) at (0,0) [circle, fill, inner sep = 4pt] {};
         \node (1) at (1,0) [rectangle, draw, inner sep = 1.5pt] {};
         \node (2) at (1,1) [circle, fill, inner sep = 1.5pt] {};
         \node (3) at (2,0) [rectangle, draw, inner sep = 1.5pt] {$n$};
         \node (4) at (2,1) [circle, draw, inner sep = 1.5pt] {$c_1$};
         \node (5) at (3,0) [circle, draw, inner sep = 1.5pt] {$c_2$};
     \draw (0) -- (1);
     \draw (1)--(2);
     \draw (1)--(3);
     \draw (3)--(4);
     \draw (3)--(5);
     \node at (1,1.3) {$a$};
\end{tikzpicture}
 \end{center}
 Now, consider the chain of equalities obtained by repeatedly adding and deleting separated vertices \begin{align*} \left[\ntree\right]&=\left[\ntreetwo\right]=\left[\ntreethree\right] \\ &= \left[\ntreefour\right]=\left[\ntreefive\right] \\&= y_1, \end{align*} which is what we need.

 If $t_{T,a} \cong t_{z_1}$, then $N(a)$ is adjacent to one leaf and two nodes. We can again remove leaves on all other nodes of distance at least two from $N(a)$ as well as leaves on the adjacent nodes until every node of $T$ has degree three with leaves of colors $c_1,c_2,c_3,c_4$ as follows:
 \begin{center}
     \begin{tikzpicture}[scale = 0.8, transform shape]
         \node (0) at (0,0) [circle, draw, inner sep = 1.5pt] {$c_1$};
         \node (1) at (1,0) [rectangle, draw, inner sep = 1.5pt] {};
         \node (2) at (1,1) [circle, draw, inner sep = 1.5pt] {$c_2$};
         \node (3) at (2,0) [rectangle, draw, inner sep = 1.5pt] {};
         \node (4) at (2,1) [circle, fill, inner sep = 4pt] {};
         \node (5) at (3,0) [rectangle, draw, inner sep =1.5pt] {};
         \node (6) at (3,1) [circle, draw, inner sep = 1.5pt] {$c_3$};
         \node (7) at (4,0) [circle, draw, inner sep = 1.5pt] {$c_4$};

         \draw (0)--(1);
         \draw (1)--(2);
         \draw (1)--(3);
         \draw (3)--(4);
         \draw (3)--(5);
         \draw (5)--(6);
         \draw (5)--(7);
     \end{tikzpicture}
 \end{center}
 We have the following chain of equalities: \begin{align*}
     \left[\mtree\right] &= \left[\mtreeone\right] = \left[\mtreetwo\right] \\ &= \left[\mtreethree\right] = \left[\mtreefour\right],
 \end{align*} which allows us to change $c_1$ and $c_2$ to black. By a similar procedure, we can change $c_3$ and $c_4$ to black to obtain $z_1$ as desired.
\end{proof}
The two above lemmas lead to the following relationships between the $y_i$, which we believe are central in explaining the vanishing of the measure ring.
\begin{lemma} \label{odd}
    We have $y_1+1=y_2$ and $y_4+1=y_3$ in $\Theta(C_n\mathfrak{T})$ and $\Theta(C_n\mathfrak{T}_k)$ for all $n \geq 2$ and $k \geq 3$.
\end{lemma}
\begin{proof}
    It suffices to show the former equation since the latter follows from switching colors. Consider the following diagram consisting of two side-by-side amalgamation squares:
    
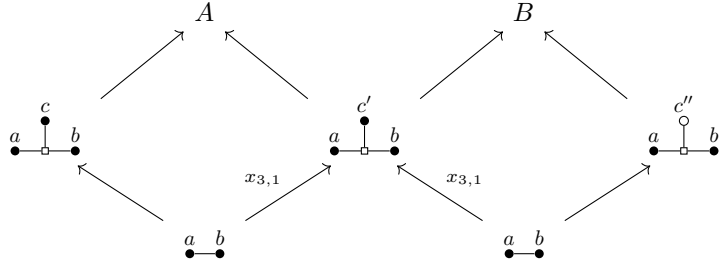
\[\begin{tikzcd}
	& A && B \\
	\threenodes && \treenodes && \teenodes \\
	& \twonodes && \twonodes
	\arrow[from=2-1, to=1-2]
	\arrow[from=2-3, to=1-2]
	\arrow[from=2-3, to=1-4]
	\arrow[from=2-5, to=1-4]
	\arrow[from=3-2, to=2-1]
	\arrow["{x_{3,1}}", from=3-2, to=2-3]
	\arrow["{x_{3,1}}"', from=3-4, to=2-3]
	\arrow[from=3-4, to=2-5]
\end{tikzcd}\]
We let $A$, $B$ denote the sets of amalgamations in the left and right squares, respectively. 

Suppose $k \geq 4$. In the left square, we can choose to either identify $c$ and $c'$ or keep them distinct. We end up with the following possible amalgamations in $A$:
\begin{center}
\begin{tikzpicture}
    \node (0) at (0,0) [circle, fill, inner sep = 1.5pt] {};
    \node (1) at (0.5,0) [rectangle, draw, inner sep = 1.5pt] {};
    \node (2) at (1,0) [circle, fill, inner sep = 1.5pt] {};
    \node (3) at (0.5,0.5) [circle, fill, inner sep = 1.5pt] {};
    \draw (0) -- (1);
    \draw (1) -- (2);
    \draw (1) -- (3);
    \node at (0,0.3) {$a$};
    \node at (1,0.3) {$b$};
    \node at (0.5,0.8) {$c/c'$};

    \node (0) at (2,0) [circle, fill, inner sep = 1.5pt] {};
    \node (1) at (2.5,0) [rectangle, draw, inner sep = 1.5pt] {};
    \node (2) at (3,0) [circle, fill, inner sep = 1.5pt] {};
    \node (3) at (2.5,0.5) [circle, fill, inner sep = 1.5pt] {};
    \node (4) at (2.5,-0.5) [circle, fill, inner sep = 1.5pt] {};
    \draw (0) -- (1);
    \draw (1) -- (2);
    \draw (1) -- (3);
    \draw (1) -- (4);
    \node at (2,0.3) {$a$};
    \node at (3,0.3) {$b$};
    \node at (2.5,0.8) {$c$};
    \node at (2.5,-0.8) {$c'$};

    \node (0) at (4,0) [circle, fill, inner sep = 1.5pt] {};
    \node (1) at (4.5,0) [rectangle, draw, inner sep = 1.5pt] {};
    \node (2) at (5,0) [rectangle, draw, inner sep = 1.5pt] {};
    \node (3) at (4.5,0.5) [circle, fill, inner sep = 1.5pt] {};
    \node (4) at (5,0.5) [circle, fill, inner sep = 1.5pt] {};
    \node (5) at (5.5,0) [circle, fill, inner sep = 1.5pt] {};
    \draw (0) -- (1);
    \draw (1) -- (2);
    \draw (1) -- (3);
    \draw (2) -- (4);
    \draw (2) -- (5);
    \node at (4,0.3) {$a$};
    \node at (5.5,0.3) {$c'$};
    \node at (4.5,0.8) {$b$};
    \node at (5,0.75) {$c$};

        \node (0) at (6.5,0) [circle, fill, inner sep = 1.5pt] {};
    \node (1) at (7,0) [rectangle, draw, inner sep = 1.5pt] {};
    \node (2) at (7.5,0) [rectangle, draw, inner sep = 1.5pt] {};
    \node (3) at (7,0.5) [circle, fill, inner sep = 1.5pt] {};
    \node (4) at (7.5,0.5) [circle, fill, inner sep = 1.5pt] {};
    \node (5) at (8,0) [circle, fill, inner sep = 1.5pt] {};
    \draw (0) -- (1);
    \draw (1) -- (2);
    \draw (1) -- (3);
    \draw (2) -- (4);
    \draw (2) -- (5);
    \node at (6.5,0.3) {$a$};
    \node at (8,0.3) {$c'$};
    \node at (7,0.75) {$c$};
    \node at (7.5,0.78) {$b$};

    \node (0) at (9,0) [circle, fill, inner sep = 1.5pt] {};
    \node (1) at (9.5,0) [rectangle, draw, inner sep = 1.5pt] {};
    \node (2) at (10,0) [rectangle, draw, inner sep = 1.5pt] {};
    \node (3) at (9.5,0.5) [circle, fill, inner sep = 1.5pt] {};
    \node (4) at (10,0.5) [circle, fill, inner sep = 1.5pt] {};
    \node (5) at (10.5,0) [circle, fill, inner sep = 1.5pt] {};
    \draw (0) -- (1);
    \draw (1) -- (2);
    \draw (1) -- (3);
    \draw (2) -- (4);
    \draw (2) -- (5);
    \node at (9,0.3) {$a$};
    \node at (10.5,0.27) {$c$};
    \node at (9.5,0.78) {$c'$};
    \node at (10,0.75) {$b$};
    \end{tikzpicture}
\end{center}
It follows from \ref{thetaa} and \ref{thetac} that \begin{align} \label{eone} x_{3,1} = 1 + x_{4,1} + 3y_1. \end{align}
In the right square, $c'$ and $c''$ could not be identified since they have different colors, so the possible amalgamations in $B$ are obtained by changing $c$ to $c''$ and their corresponding colors in all but the first amalgamation in the list above. Combined with Lemmas \ref{33} and \ref{34}, we have the amalgamation equation \begin{align} \label{etwo} x_{3,1} = x_{4,1} + y_2+2y_1 \end{align} by \ref{thetac}. For the $k=3$ case, simply note that $x_{4,1}$ does not exist in both \ref{eone} and \ref{etwo}. Combining the two equations now yields the desired conclusion.
\end{proof}
\begin{proof}[Proof of Theorem \ref{firstzero}]
We will arrive at a contradiction by combining Lemma \ref{odd} and some quadratic equations resulting from Definition \ref{thetab}. Notice that if $a,b$ are distinct leaves of a tree $T$, then $$[T,a][T \setminus a, b] = [T,b][T \setminus b,a], $$ so each tree with two marked leaves $a,b$ gives rise to a quadratic equation in $\Theta(C_n\mathfrak{T})$ and $\Theta(C_n\mathfrak{T}_k)$. Indeed, consider the four following double-marked trees: 
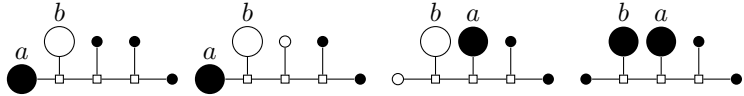
\begin{center}
        \begin{tikzpicture}
             \node (1) at (0,0) [circle,fill, inner sep=4pt] {};
  \node (2) at (1/2,0) [rectangle, draw, inner sep=1.5pt] {};
  \node (3) at (1/2,1/2) [circle,draw, inner sep=4pt] {};
  \node (4) at (2/2,0) [rectangle, draw, inner sep=1.5pt] {};
  \node (5) at (2/2,1/2) [circle,fill, inner sep=1.5pt] {};
  \node (6) at (3/2,0) [rectangle, draw, inner sep = 1.5pt] {};
  \node (7) at (3/2,1/2) [circle,fill, inner sep=1.5pt] {};
  \node (8) at (4/2,0) [circle,fill, inner sep=1.5pt] {};

    \draw (1) -- (2);
  \draw (2) -- (3);
  \draw (2) -- (4);
  \draw (4) -- (5);
  \draw (4) -- (6);
  \draw (6) -- (7);
  \draw (6) -- (8);

  \node at (0,0.3/2+0.2) {$a$};
  \node at (1/2,1.3/2+0.25) {$b$};

  \node (1) at (0+5/2,-2/2+1) [circle,fill, inner sep=4pt] {};
  \node (2) at (1/2+5/2,-2/2+1) [rectangle, draw, inner sep=1.5pt] {};
  \node (3) at (1/2+5/2,-1/2+1) [circle,draw, inner sep=4pt] {};
  \node (4) at (2/2+5/2,-2/2+1) [rectangle, draw, inner sep=1.5pt] {};
  \node (5) at (2/2+5/2,-1/2+1) [circle,draw, inner sep=1.5pt] {};
  \node (6) at (3/2+5/2,-2/2+1) [rectangle, draw, inner sep = 1.5pt] {};
  \node (7) at (3/2+5/2,-1/2+1) [circle,fill, inner sep=1.5pt] {};
  \node (8) at (4/2+5/2,-2/2+1) [circle,fill, inner sep=1.5pt] {};

    \draw (1) -- (2);
  \draw (2) -- (3);
  \draw (2) -- (4);
  \draw (4) -- (5);
  \draw (4) -- (6);
  \draw (6) -- (7);
  \draw (6) -- (8);

  \node at (0+5/2,-1.7/2+1.2) {$a$};
  \node at (1/2+5/2,-0.7/2+1.25) {$b$};

  \node (1) at (0+5,-4/2+2) [circle,draw, inner sep=1.5pt] {};
  \node (2) at (1/2+5,-4/2+2) [rectangle, draw, inner sep=1.5pt] {};
  \node (3) at (1/2+5,-3/2+2) [circle,draw, inner sep=4pt] {};
  \node (4) at (2/2+5,-4/2+2) [rectangle, draw, inner sep=1.5pt] {};
  \node (5) at (2/2+5,-3/2+2) [circle, fill, inner sep=4pt] {};
  \node (6) at (3/2+5,-4/2+2) [rectangle, draw, inner sep = 1.5pt] {};
  \node (7) at (3/2+5,-3/2+2) [circle,fill, inner sep=1.5pt] {};
  \node (8) at (4/2+5,-4/2+2) [circle,fill, inner sep=1.5pt] {};

    \draw (1) -- (2);
  \draw (2) -- (3);
  \draw (2) -- (4);
  \draw (4) -- (5);
  \draw (4) -- (6);
  \draw (6) -- (7);
  \draw (6) -- (8);

  \node at (2/2+5,-2.7/2+2.2) {$a$};
  \node at (1/2+5,-2.7/2+2.25) {$b$};

  \node (1) at (0+15/2,-6/2+3) [circle, fill, inner sep=1.5pt] {};
  \node (2) at (1/2+15/2,-6/2+3) [rectangle, draw, inner sep=1.5pt] {};
  \node (3) at (1/2+15/2,-5/2+3) [circle, fill, inner sep=4pt] {};
  \node (4) at (2/2+15/2,-6/2+3) [rectangle, draw, inner sep=1.5pt] {};
  \node (5) at (2/2+15/2,-5/2+3) [circle, fill, inner sep=4pt] {};
  \node (6) at (3/2+15/2,-6/2+3) [rectangle, draw, inner sep = 1.5pt] {};
  \node (7) at (3/2+15/2,-5/2+3) [circle,fill, inner sep=1.5pt] {};
  \node (8) at (4/2+15/2,-6/2+3) [circle,fill, inner sep=1.5pt] {};

    \draw (1) -- (2);
  \draw (2) -- (3);
  \draw (2) -- (4);
  \draw (4) -- (5);
  \draw (4) -- (6);
  \draw (6) -- (7);
  \draw (6) -- (8);

  \node at (2/2+15/2,-4.7/2+3.2) {$a$};
  \node at (1/2+15/2,-4.7/2+3.25) {$b$};
        \end{tikzpicture}
    \end{center}
From left to right, let the marked trees be $T_1$, $T_2$, $T_3$, $T_4$. From $T_1$ and Proposition \ref{sep}, we obtain $$y_2y_3 = y_3y_1 = y_3(y_2-1),$$ so $y_3=0$. It follows that $y_4=-1$. Using Lemma \ref{34}, the equation corresponding to $T_2$ is $$y_2y_4=y_3y_2,$$ implying that $(y_1,y_2)=(-1,0)$. Again by Lemma \ref{34}, $T_3$ gives $$z_1y_4=y_4y_2,$$ which leads to $z_1=0$. Finally, the equation corresponding to $T_4$ is $$z_1y_1=y_1^2,$$ or $0=1$. Consequently, $\Theta(C_n\mathfrak{T})$ and $\Theta(C_n\mathfrak{T}_k)$ vanish. 
\end{proof}
\subsection{Labeled trees} \label{labeled}
We again begin with a sufficient condition for separated leaves. Since Proposition \ref{31} also applies to uncolored trees, for a labeled (uncolored) tree $T$, we say that two leaves $a$ and $b$ are \emph{$R$-separated} if $d(N(a),N(b)) \geq 2$ or $\mathrm{deg}(N(a)) \geq 4$ or $\mathrm{deg}(N(b)) \geq 4$ and \emph{$<$-separated} if there exists another leaf $c \in T$ such that either $a<c<b$ or $b<c<a$. In particular, if $a$ and $b$ are $<$-separated, then we cannot modify their order in any amalgamation of $T \setminus a$ and $T \setminus b$ over $T \setminus \{a,b\}$, so any two amalgamations carry the same labeling. \begin{prop} \label{36}
    Let $T$ be a labeled tree with distinct leaves $a$ and $b$. If $a$ and $b$ are both $R$-separated and $<$-separated, then $a$ and $b$ are separated in $T$, i.e., $[T,a] = [T \setminus b,a]$.
\end{prop} 
\begin{proof}
    Recall from \ref{cameron} that we treat $R$ and $<$ as independent relations where we perform amalgamations independently. This means that the total number of amalgamations of $T \setminus a$ and $T \setminus b$ over $T \setminus \{a,b\}$ is equal to the product of the number of amalgamations of their unlabeled versions and their total orders, which is $1 \cdot 1 = 1$ when $a$ and $b$ are both $R$-separated and $<$-separated.  
\end{proof}
In fact, we introduce another result to further manipulate the ordering of the leaves in labeled trees. 
\begin{prop} \label{37}
    Let $(T,a)$ be a marked structure, and $b$ be a leaf $<$-separated from $a$ such that $N(a) \neq N(b)$. Suppose $T'$ is obtained from $T$ by relocating $b$ to $b'$ in the induced order of the leaves, so that $b'$ remains $<$-separated from $a$ and the relative order of the remaining leaves is preserved. Then $[T,a]=[T',a]$.
\end{prop}
\begin{proof}
If $a$ and $b$ are also $R$-separated, then they are separated in $T$ by Proposition \ref{36}, so indeed $[T,a]=[T \setminus b, a] = [T' \setminus b', a] = [T',a]$. Hence suppose that $b$ (and $b'$) is not $R$-separated from $a$. It follows that $N(a)$ and $N(b)$ are adjacent and $\mathrm{deg}(N(a)) = \mathrm{deg}(N(b)) = 3$, so $(T,a)$ is of the form 
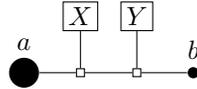
\begin{center}
    \begin{tikzpicture}[scale=3/2]
        \node (0) at (0,0) [circle, fill, inner sep = 4pt] {};
        \node (1) at (1/2,0) [rectangle, draw, inner sep = 1.5pt] {};
        \node (2) at (1/2,1/2) [rectangle, draw, inner sep = 2pt] {$X$};
        \node (3) at (1,0) [rectangle, draw, inner sep = 1.5pt] {};
        \node (4) at (1,1/2) [rectangle, draw, inner sep = 2pt] {$Y$};
        \node (5) at (3/2,0) [circle, fill, inner sep = 1.5pt] {};
        \draw (0)--(1);
        \draw (1)--(2);
        \draw (1)--(3);
        \draw (3)--(4);
        \draw (3)--(5);
        \node at (0,0.25) {$a$};
        \node at (3/2,0.2) {$b$};
    \end{tikzpicture} 
\end{center} By Proposition \ref{36}, we have $$
    \left[\xytree[0.85]\right] = \left[\xytreetwo[0.85]\right]=\left[\xytreethree[0.85]\right]
$$ where the class after the second equal sign is $[T',a]$, as needed.
\end{proof}
The following is the second main result of this section.
\begin{theorem} \label{secondzero}
    We have $\Theta(L\mathfrak{T}) \cong \mathbf{0}$, so there are no measures on the class of labeled trees.
\end{theorem}
\begin{proof} As in the previous case with colored trees, we arrive at the result by invoking a series of linear and quadratic equations that lead to a contradiction. We note the following observation by \cite{harman2024arborealtensorcategories}: given marked structures $(T \setminus b,a)$ and $(T \setminus a, b)$ where $T_1, \dots, T_k$ are all possible amalgamations of $T \setminus a$ and $T \setminus b$ over $T \setminus \{a,b\}$, there is an associated linear equation due to Definition \ref{thetac} given by $$[T \setminus b,a] = \sum_{i=1}^k [T_i,a].$$ Consequently, each ordered pair of marked trees $((T \setminus b, a),(T \setminus a,b))$ leads to a linear equation. Moreover, recall from \ref{n-colored} that a quadratic equation $[T,a][T \setminus a,b]=[T,b][T \setminus b,a]$ arises from a tree $T$ with marked leaves $a,b$.

First, consider the duplicate pair of marked trees $(\xedge[0.75],\xedge[0.75])$. Its associated equation is $$\left[\!\xedge[0.75]\!\right] = 1 + \left[\xthreem[0.75]\right] + \left[\!\xthreel[0.75]\!\right].$$ Note, however, that in the rightmost class, $1$ is separated from $3$, so we end up with $\left[\xthreem[0.5]\right] = -1$.

Next, consider the pair  $\left(\xtwofour[0.5],\xfourm[0.5] \right)$, which gives rise to two amalgamations:
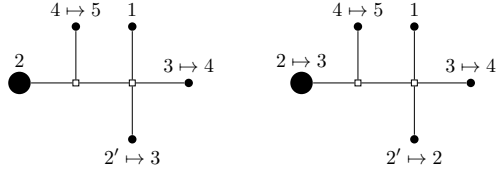
\begin{center}
\scalebox{0.75}{
    \begin{tikzpicture}
        \node (4) at (3,0) [circle, fill, inner sep=4pt] {};
        \node (5) at (4,0) [rectangle, draw, inner sep=1.5pt] {};
        \node (6) at (4,1) [circle, fill, inner sep = 1.5pt] {};
        \node (7) at (5,0) [rectangle, draw, inner sep = 1.5pt] {};
        \node (8) at (5,1) [circle, fill, inner sep = 1.5pt] {};
        \node (9) at (5,-1) [circle, fill, inner sep = 1.5pt] {};
        \node (10) at (6,0) [circle, fill, inner sep = 1.5pt] {};
        \draw (4) -- (5);
        \draw (5) -- (6);
        \draw (5) -- (7);
        \draw (7) -- (8);
        \draw (7) -- (9);
        \draw (7) -- (10);
        \node at (3,0.4) {$2$};
        \node at (5,1.3) {$1$};
        \node at (4,1.3) {$4 \mapsto 5$};
        \node at (6,0.3) {$3 \mapsto 4$};
        \node at (5,-1.3) {$2' \mapsto 3$};

        \node (4) at (8,0) [circle, fill, inner sep=4pt] {};
        \node (5) at (9,0) [rectangle, draw, inner sep=1.5pt] {};
        \node (6) at (9,1) [circle, fill, inner sep = 1.5pt] {};
        \node (7) at (10,0) [rectangle, draw, inner sep = 1.5pt] {};
        \node (8) at (10,1) [circle, fill, inner sep = 1.5pt] {};
        \node (9) at (10,-1) [circle, fill, inner sep = 1.5pt] {};
        \node (10) at (11,0) [circle, fill, inner sep = 1.5pt] {};
        \draw (4) -- (5);
        \draw (5) -- (6);
        \draw (5) -- (7);
        \draw (7) -- (8);
        \draw (7) -- (9);
        \draw (7) -- (10);
        \node at (8,0.4) {$2 \mapsto 3$};
        \node at (10,1.3) {$1$};
        \node at (9,1.3) {$4 \mapsto 5$};
        \node at (11,0.3) {$3 \mapsto 4$};
        \node at (10,-1.3) {$2' \mapsto 2$};
    \end{tikzpicture}}
\end{center}
By Proposition \ref{36}, $2$ and $1$ are separated in the left amalgamation and $3$ and $1$ are separated in the right amalgamation, so the resulting equation is $$\left[\xtwofour[0.75] \right] = \left[\xtwofour[0.75]\right] + \left[\xtwofour[0.75] \right],$$ which yields $\left[\xtwofour[0.5]\right]=0$.

Now, consider the following double-marked tree: \begin{center}
\scalebox{0.75}
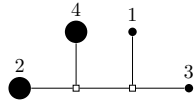
{
    \begin{tikzpicture}
        \node (4) at (3,0) [circle, fill, inner sep=4pt] {};
        \node (5) at (4,0) [rectangle, draw, inner sep=1.5pt] {};
        \node (6) at (4,1) [circle, fill, inner sep = 4pt] {};
        \node (7) at (5,0) [rectangle, draw, inner sep = 1.5pt] {};
        \node (8) at (5,1) [circle, fill, inner sep = 1.5pt] {};
        \node (10) at (6,0) [circle, fill, inner sep = 1.5pt] {};
        \draw (4) -- (5);
        \draw (5) -- (6);
        \draw (5) -- (7);
        \draw (7) -- (8);
        \draw (7) -- (10);
        \node at (3,0.4) {$2$};
        \node at (5,1.3) {$1$};
        \node at (4,1.4) {$4$};
        \node at (6,0.3) {$3$};
        \end{tikzpicture}}
        \end{center}
We previously found that $\left[\xtwofour[0.5]\right] = 0$, so the associated quadratic equation reduces to $$\left[\xfourtwo \right] \cdot \left[\xthreem\right] = -\left[\xfourtwo \right] = -\left[\xfourone \right] = 0,$$ which means that $\left[ \xfourone[0.5] \right] = 0$.

The remaining relevant diagrams (either pairs of marked trees or double-marked trees), as well as their corresponding equations, are given in the table below. They are consequences of Propositions \ref{36} and \ref{37} combined with the above equations. 
\begin{center}
\setlength{\extrarowheight}{2pt}      
\renewcommand{\arraystretch}{1.15}    

\begin{tabular}{c| r@{\;$=$\;}l}
  Diagram & \multicolumn{2}{c}{Equation} \\
  \hline
  \noalign{\vskip 3pt} 
  $\left(\xfourm[0.5], \xtwofour[0.5] \right)$
    & $\left[\xfourm[0.5]\right]$
    & $0$ \\[1.0ex]

  $\left(\xfourthree[0.5], \xfourg[0.5]\right)$
    & $\left[\xfourthree[0.5]\right]$
    & $\left[\xtwoone[0.5]\right]$ \\[1.0ex]

  $\xthreetwoq[0.5]$
    & $\left[\xtwoone[0.5]\right]\cdot\left[\xthreeg[0.5]\right]$
    & 0 \\[1.0ex]

  $\double[0.5]$
    & $\left[\xtwothree[0.5]\right]$
    & $\left[\xtwoone[0.5]\right]$ \\[1.0ex]

  $\left(\xthreem[0.5], \xthreeg[0.5]\right)$
    & $\left[\xtwoone[0.5]\right]$
    & $-\frac{1}{2}$ \\[1.0ex]

  $\xtwooneq[0.5]$
    & $\left[\xfourthree[0.5]\right]$
    & $0$
\end{tabular}
\end{center}
From the second and last two equations, we have $\left[\xtwoone[0.5]\right]=0=-\frac{1}{2}$, so $\Theta(L \mathfrak{T})$ vanishes. \end{proof}
\section{Measures on node-colored rooted binary trees} \label{ms}
In this section, we fully compute the measure ring $\Theta(\partial \mathfrak{T}_3(n))$, thus classifying all measures on the class of $n$-colored rooted binary trees. Afterwards, we classify the Fraïssé subclasses of $\partial \mathfrak{T}_3(n)$ to examine measures that are regular on a subclass. These measures give rise to semisimple tensor categories as we discuss in Section \ref{category}.
\subsection{Computation of measure ring} We note that our approach to computing the measure ring is similar to that used in \cite{harman2024arborealtensorcategories}, which starts with reducing the generating set of $\Theta(\partial \mathfrak{T}_3(n))$ to finitely many marked structures through the following proposition.
\begin{prop} \label{easy}
    Two distinct leaves $a$ and $b$ of $T \in \partial \mathfrak{T}_3(n)$ are separated if and only if $N(a)$ and $N(b)$ are neither the same nor adjacent. 
\end{prop}
\begin{proof}
    If $N(a)=N(b)$, then we obtain another amalgamation of $T \setminus a$ and $T \setminus b$ over $T \setminus \{a,b\}$ (apart from $T$) by identifying $a$ and $b$. If $N(a)$ and $N(b)$ are adjacent, which is the following picture, 
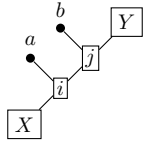
\begin{center}
    \begin{tikzpicture}[scale = 0.8, transform shape]
        \node (2) at (0,1) [circle, fill, inner sep=1.5pt] {};
        \node (3) at (0.5,1.5) [circle, fill, inner sep=1.5pt] {};
        \node[draw, rectangle] (X) at (-0.1,-0.1) {$X$};
        \node (5) at (0.5,0.5) [rectangle, draw, inner sep=1.5pt] {$i$};
        \node (6) at (1,1) [rectangle, draw, inner sep=1.5pt] {$j$};
        \node[draw, rectangle] (Y) at (1.6,1.6) {$Y$};
        \draw (5) -- (2);
        \draw (5) -- (X);
        \draw (6) -- (Y);
        \draw (6) -- (3);
        \draw (5) -- (6);
        \node at (0,1.3) {$a$};
        \node at (0.5,1.8) {$b$};
    \end{tikzpicture}
\end{center}
then swapping $a$ and $b$ (and $N(a)$ and $N(b)$) is another amalgamation:
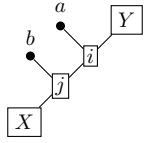
\begin{center}
    \begin{tikzpicture}[scale = 0.8, transform shape]
        \node (2) at (0,1) [circle, fill, inner sep=1.5pt] {};
        \node (3) at (0.5,1.5) [circle, fill, inner sep=1.5pt] {};
        \node[draw, rectangle] (X) at (-0.1,-0.1) {$X$};
        \node (5) at (0.5,0.5) [rectangle, draw, inner sep=1.5pt] {$j$};
        \node (6) at (1,1) [rectangle, draw, inner sep=1.5pt] {$i$};
        \node[draw, rectangle] (Y) at (1.6,1.6) {$Y$};
        \draw (5) -- (2);
        \draw (5) -- (X);
        \draw (6) -- (Y);
        \draw (6) -- (3);
        \draw (5) -- (6);
        \node at (0,1.3) {$b$};
        \node at (0.5,1.8) {$a$};
    \end{tikzpicture}
\end{center} Now suppose that $N(a)$ and $N(b)$ are neither the same nor adjacent. If the path from $a$ to $b$ passes through the root, then in any amalgamation $T'$, both $a$ and $b$ remain on different sides of the root, so they can be neither identified nor swapped. As a result, $T' = T$. If $a$ and $b$ are on the same side of the root, then there is a node $M$ such that $M$ is closer to the root $r$ than exactly one of $N(a)$ and $N(b)$ (say $N(a)$). This means that in any amalgamation $T'$, we have $d(a,r) < d(b,r)$, so both $a$ and $b$ are again fixed, implying that there are no amalgamations distinct from $T$.
\end{proof}
Let $[n] = \{1, \dots, n\}$. For $i,j,k \in [n]$, let $A$, $B(i)$, $C(i,j)$, $D(i,j)$, $E(i,j,k)$ denote the isomorphism classes of the following marked trees as shown from left to right, where the arguments indicate colors of nodes as ordered by increasing distance from the root. 
    
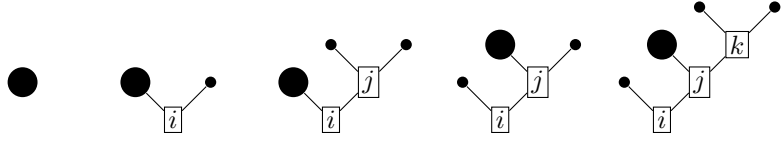
\begin{center}
\begin{tikzpicture}[baseline=(current bounding box.center)][scale = 0.8, transform shape]
\begin{scope}[xshift=0cm]
  \node (a0) at (0,0.5) [circle,fill, inner sep=4pt] {};
  \node (a2) at (2,0)   [rectangle, draw, inner sep=1.5pt] {$i$};
  \node (a3) at (1.5,0.5) [circle, fill, inner sep=4pt] {};
  \node (a4) at (2.5,0.5) [circle, fill, inner sep=1.5pt] {};
  \draw (a2)--(a3) (a2)--(a4);
\end{scope}
\begin{scope}[xshift=3.60cm]
  \node (b5)  at (0,0.5)   [circle,fill, inner sep=4pt] {};
  \node (b6)  at (0.5,0)   [rectangle, draw, inner sep=1.5pt] {$i$};
  \node (b7)  at (1,0.5)   [rectangle, draw, inner sep=1.5pt] {$j$};
  \node (b8)  at (0.5,1)   [circle, fill, inner sep=1.5pt] {};
  \node (b9)  at (1.5,1)   [circle, fill, inner sep=1.5pt] {};
  \draw (b5)--(b6) (b7)--(b6) (b7)--(b8) (b7)--(b9);
\end{scope}
\begin{scope}[xshift=5.85cm]
  \node (c5)  at (0,0.5)   [circle,fill, inner sep=1.5pt] {};
  \node (c6)  at (0.5,0)   [rectangle, draw, inner sep=1.5pt] {$i$};
  \node (c7)  at (1,0.5)   [rectangle, draw, inner sep=1.5pt] {$j$};
  \node (c8)  at (0.5,1)   [circle, fill, inner sep=4pt] {};
  \node (c9)  at (1.5,1)   [circle, fill, inner sep=1.5pt] {};
  \draw (c5)--(c6) (c7)--(c6) (c7)--(c8) (c7)--(c9);
\end{scope}
\begin{scope}[xshift=8.00cm]
  \node (d5)  at (0,0.5)   [circle,fill, inner sep=1.5pt] {};
  \node (d6)  at (0.5,0)   [rectangle, draw, inner sep=1.5pt] {$i$};
  \node (d7)  at (1,0.5)   [rectangle, draw, inner sep=1.5pt] {$j$};
  \node (d8)  at (0.5,1)   [circle, fill, inner sep=4pt] {};
  \node (d9)  at (1.5,1)   [rectangle, draw, inner sep=1.5pt] {$k$};
  \node (d10) at (1,1.5)   [circle, fill, inner sep=1.5pt] {};
  \node (d11) at (2,1.5)   [circle, fill, inner sep=1.5pt] {};
  \draw (d5)--(d6) (d7)--(d6) (d7)--(d8) (d7)--(d9)
        (d9)--(d10) (d9)--(d11);
\end{scope}
\end{tikzpicture}

  \end{center}
Notice that $A$ is the only class of one-leaf trees up to isomorphism.
\begin{prop} \label{irreducible}
    The ring $\Theta(\partial \mathfrak{T}_3(n))$ is generated by  $A$, $B(i)$, $C(i,j)$, $D(i,j)$, $E(i,j,k)$ where $i,j,k$ range over $[n]$.
\end{prop}
\begin{proof}
    We show that any class of marked structures $[T,a]$ is equal to one of the above classes in $\partial \mathfrak{T}_3(n)$ by induction on $|L(T)|$. The result holds trivially if $|L(T)| \leq 3$, so suppose $|L(T)| \geq 4$. If $[T,a]$ is not $E(i,j,k)$, then there exists $b \in L(T)$ such that $d(N(a),N(b)) > 1$. By Proposition \ref{easy}, we have $[T \setminus b, a] = [T,a]$ where $|L(T) \setminus b| = |L(T)|-1$, so $[T,a]$ is equal to one of the above classes by the inductive hypothesis. This completes the induction.
\end{proof}
In particular, we observe that the marked leaf of the structures in Proposition \ref{irreducible} is not separated from any of the remaining leaves. We say that such marked structures are \emph{irreducible}. 
To address linear relations in $\partial \mathfrak{T}_3(n)$, recall that given marked structures $(T \setminus b,a)$ and $(T \setminus a, b)$ where $T_1, \dots, T_k$ are all the possible amalgamations of $T \setminus a$ and $T \setminus b$ over $T \setminus \{a,b\}$, there is an associated linear equation $L_{T,a,b}$ given by $$[T \setminus b,a] = \sum_{i=1}^k [T_i,a].$$ We say that $L_{T,a,b}$ is a \emph{minimal linear equation} if $[T \setminus b,a]$ and $[T \setminus a, b]$ are both irreducible and isomorphic up to the colors of $N(a)$ and $N(b)$.
\begin{prop}
    Any nontrivial linear equation $L_{T,a,b}$ (where $k \geq 2$) is equal to a minimal linear equation. 
\end{prop}
\begin{proof}
    If $[T \setminus b,a]$ is not irreducible, then there exists $c \in L(T)$ that is separated from $a$, so we have $[T \setminus b, a] = [T \setminus \{b,c\}, a]$. Moreover, $c$ remains separated from $a$ in any amalgamation $T_i$, so $[T_i,a] = [T_i \setminus c, a]$ for each $i$, which means that removing leaves separated from $a$ does not change the equation. Thus, suppose that $[T \setminus b,a]$ is irreducible. If $a$ and $b$ are separated in $T$, then $T$ is the unique desired amalgamation by Proposition \ref{easy}, so $L_{T,a,b}$ is just the trivial equality. It follows that $a$ and $b$ are not separated, so $N(a)$ and $N(b)$ are either the same or adjacent, implying that $[T \setminus a,b]$ only potentially differs from $[T \setminus b, a]$ at the colors of $N(a)$ and $N(b)$. Consequently, $[T \setminus a,b]$ is also irreducible, reducing $L_{T,a,b}$ to a minimal linear equation, as desired.
\end{proof}
By the above proposition, it suffices to consider linear relations among $A$, $B(i)$, $C(i,j)$, $D(i,j)$, and $E(i,j,k)$, particularly those arising from structures only potentially differing at the colors of the nodes adjacent to their respective marked points. Since there are finitely many families of such relations, we obtain the comprehensive list below upon simplification.
\begin{cor} \label{linear}
    For pairwise distinct $i,j,k$ ranging over $[n]$, the linear relations in $\Theta(\partial \mathfrak{T}_3(n))$ are generated by the following families: $$
    \begin{array}{ll}
        A = 1 + \sum_{i \in [n]} B(i), & B(i) = 1 + C(i,i) + D(i,i) + \sum_{p\in[n]} D(i,p), \\[4pt]
 B(i) = C(i,j) + D(j,i), & 0 = 1 + \sum_{p\in[n]} D(i,p)+E(i,i,i), \\[4pt]
C(i,i) = C(i,j) + E(j,i,i), &  D(i,i) = D(j,i) + E(i,i,j), \\[4pt] E(i,j,i) = 0, & C(i,j) = C(i,k) + E(k,i,j).
\end{array} $$
\end{cor}
\begin{proof}
    We show how the equation in the second column of the second row is obtained; the rest can be produced through a similar process. Consider the two copies of $C(i,i)$ with differentiated marked points: 
    
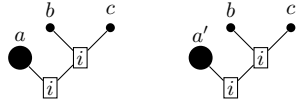
\begin{center}
    \begin{tikzpicture}[scale = 0.8, transform shape]
    \node (5) at (5.5,0.5) [circle,fill, inner sep=4pt] {};
  \node (6) at (6,0) [rectangle, draw, inner sep=1.5pt] {$i$};
  \node (7) at (6.5, 0.5) [rectangle, draw, inner sep = 1.5pt] {$i$};
  \node (8) at (6,1) [circle, fill, inner sep=1.5pt] {};
  \node (9) at (7,1) [circle, fill, inner sep = 1.5pt] {};

  \draw (5) -- (6);
  \draw (7) -- (6);
  \draw (7) -- (8);
  \draw (7) -- (9);
\node at (5.5,0.85) {$a$};
\node at (6, 1.3) {$b$};
\node at (7,1.3) {$c$};
  \node (5) at (8.5,0.5) [circle,fill, inner sep=4pt] {};
  \node (6) at (9,0) [rectangle, draw, inner sep=1.5pt] {$i$};
  \node (7) at (9.5, 0.5) [rectangle, draw, inner sep = 1.5pt] {$i$};
  \node (8) at (9,1) [circle, fill, inner sep=1.5pt] {};
  \node (9) at (10,1) [circle, fill, inner sep = 1.5pt] {};

  \draw (5) -- (6);
  \draw (7) -- (6);
  \draw (7) -- (8);
  \draw (7) -- (9);
  \node at (8.5,0.9) {$a'$};
\node at (9, 1.3) {$b$};
\node at (10,1.3) {$c$};
  \end{tikzpicture}
  \end{center}
Self-amalgamations of $C(i,i)$ with marked leaf $a$ are labeled with $a,a',b,c$ such that the induced marked subtrees on $\{a,b,c\}$ and $\{a',b,c\}$ (with marked points $a$ and $a'$, respectively) are isomorphic to $C(i,i)$. Up to coloring and isomorphism, there are only two unordered rooted binary trees with four leaves, so all the self-amalgamations are as follows
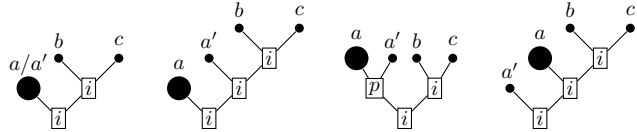
\begin{center}
    \begin{tikzpicture}[scale = 0.8, transform shape]
        \node (0) at (0,0) [rectangle, draw, inner sep = 1.5pt] {$i$};
        \node (1) at (-0.5,0.5) [circle, fill, inner sep = 4pt] {};
        \node (2) at (0.5,0.5) [rectangle, draw, inner sep = 1.5pt] {$i$};
        \node (3) at (0,1) [circle, fill, inner sep = 1.5pt] {};
        \node (4) at (1,1) [circle, fill, inner sep = 1.5pt] {};
    \draw (0) -- (1);
    \draw (0) -- (2);
    \draw (2) -- (3);
    \draw (2) -- (4);

    \node at (-0.5,0.9) {$a/a'$};
    \node at (0,1.3) {$b$};
    \node at (1,1.3) {$c$};

\node (1) at (2,0.5) [circle, fill, inner sep=4pt] {};
        \node (2) at (2.5,1) [circle, fill, inner sep=1.5pt] {};
        \node (3) at (3,1.5) [circle, fill, inner sep=1.5pt] {};
        \node (4) at (2.5,0) [rectangle, draw, inner sep=1.5pt] {$i$};
        \node (5) at (3,0.5) [rectangle, draw, inner sep=1.5pt] {$i$};
        \node (6) at (3.5,1) [rectangle, draw, inner sep=1.5pt] {$i$};
        \node (7) at (4,1.5) [circle, fill, inner sep=1.5pt] {};
        \draw (1) -- (4);
        \draw (4) -- (5);
        \draw (5) -- (2);
        \draw (5) -- (6);
        \draw (6) -- (7);
        \draw (6) -- (3);
        \node at (2,0.9) {$a$};
        \node at (2.5,1.3) {$a'$};
        \node at (3,1.8) {$b$};
        \node at (4,1.8) {$c$};

        \node (1) at (7.5,0.5) [circle, fill, inner sep=1.5pt] {};
        \node (2) at (8,1) [circle, fill, inner sep=4pt] {};
        \node (3) at (8.5,1.5) [circle, fill, inner sep=1.5pt] {};
        \node (4) at (8,0) [rectangle, draw, inner sep=1.5pt] {$i$};
        \node (5) at (8.5,0.5) [rectangle, draw, inner sep=1.5pt] {$i$};
        \node (6) at (9,1) [rectangle, draw, inner sep=1.5pt] {$i$};
        \node (7) at (9.5,1.5) [circle, fill, inner sep=1.5pt] {};
        \draw (1) -- (4);
        \draw (4) -- (5);
        \draw (5) -- (2);
        \draw (5) -- (6);
        \draw (6) -- (7);
        \draw (6) -- (3);
        \node at (7.5,0.8) {$a'$};
        \node at (8,1.4) {$a$};
        \node at (8.5,1.8) {$b$};
        \node at (9.5,1.8) {$c$};
    
            \node (B1) at (5.75,0)     [rectangle, draw, inner sep=1.5pt] {$i$};
    \node (B4) at (5.25,0.5) [rectangle, draw, inner sep=1.5pt] {$p$};
    \node (B5) at (6.25,0.5) [rectangle, draw, inner sep=1.5pt] {$i$};

    \node (B7) at (4.95,1.0) [circle, fill, inner sep=4pt] {}; 
    \node (B8) at (5.55,1.0) [circle, fill, inner sep=1.5pt] {}; 
    \node (B2) at (5.95,1.0) [circle, fill, inner sep=1.5pt] {}; 
    \node (B6) at (6.55,1.0) [circle, fill, inner sep=1.5pt] {}; 

    \draw (B1) -- (B4);
    \draw (B1) -- (B5);
    \draw (B4) -- (B7);
    \draw (B4) -- (B8);
    \draw (B5) -- (B2);
    \draw (B5) -- (B6);

    \node at (4.95,1.4) {$a$};
    \node at (5.55,1.3) {$a'$};
    \node at (5.95,1.3) {$b$};
    \node at (6.55,1.3) {$c$};
    \end{tikzpicture}
\end{center}
where $p$ ranges over all colors in $[n]$. From left to right, let the marked trees be $I_1$, $I_2$, $I_{3p}$, and $I_4$. Since $I_1$ is obtained by identifying $a$ and $a'$, we have $[I_1]=1$. In $I_2$, $b$ is separated from $a$, so we have $I_2 \cong I_2 \setminus b \cong C(i,i)$, implying that $[I_2] = C(i,i)$. In $I_{3p}$, $b$ is also separated from $a$, giving $I_{3p} \cong I_{3p} \setminus b \cong D(i,p)$. Thus, $[I_{3p}] \cong D(i,p)$ for all $p \in [n]$. Finally, $I_4$ is already irreducible and isomorphic to the isomorphism class of $E(i,i,i)$, so we have $C(i,i) = 1 + C(i,i) + \sum_{p \in [n]} D(i,p) + E(i,i,i)$, which implies the result.
\end{proof}
Recall that a measure $\mu$ on a Fraïssé class $\mathfrak{F}$ is regular if $\mu(i)$ is nonzero for all embeddings $i$ of structures in $\mathfrak{F}$. 
\begin{cor}
    For each $n \geq 2$, there are no regular measures on $\partial \mathfrak{T}_3(n)$.
\end{cor}
\begin{proof}
    From Corollary \ref{linear}, we have $E(i,j,i) = 0$ for any pair of distinct indices $(i,j)$, so at least one embedding always vanishes whenever at least two colors are present $\partial \mathfrak{T}_3(n)$.
\end{proof}
Now, we consider quadratic relations among marked structures in $\partial \mathfrak{T}_3(n)$. Due to Definition \ref{thetab}, they are generated by equations $Q_{T,a,b}$ of the form $$[T,a][T\setminus a, b]=[T,b][T \setminus b,a]$$ for $a,b \in L(T)$ since amalgamations of $T \setminus a$ and $T \setminus b$ must agree on $T \setminus \{a,b\}$. We say that $Q_{T,a,b}$ is a \emph{minimal quadratic equation} if $T$ is an $n$-coloring of one of the following trees and $a,b$ are the marked leaves in either order (since $Q$ is symmetric):
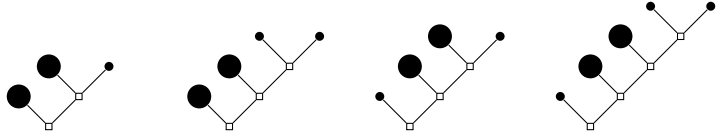
\begin{center}
\begin{tikzpicture}[scale = 0.8, transform shape]
\begin{scope}[xshift=0cm]
  \node (a5) at (0,0.5) [circle,fill, inner sep=4pt] {};
  \node (a6) at (0.5,0) [rectangle, draw, inner sep=1.5pt] {};
  \node (a7) at (1, 0.5) [rectangle, draw, inner sep=1.5pt] {};
  \node (a8) at (0.5,1) [circle, fill, inner sep=4pt] {};
  \node (a9) at (1.5,1) [circle, fill, inner sep=1.5pt] {};
  \draw (a5)--(a6) (a7)--(a6) (a7)--(a8) (a7)--(a9);
\end{scope}
\begin{scope}[xshift=3cm]
  \node (b5) at (0,0.5) [circle,fill, inner sep=4pt] {};
  \node (b6) at (0.5,0) [rectangle, draw, inner sep=1.5pt] {};
  \node (b7) at (1, 0.5) [rectangle, draw, inner sep=1.5pt] {};
  \node (b8) at (0.5,1) [circle, fill, inner sep=4pt] {};
  \node (b9) at (1.5,1) [rectangle, draw, inner sep=1.5pt] {};
  \node (b10) at (1,1.5) [circle, fill, inner sep=1.5pt] {};
  \node (b11) at (2,1.5) [circle, fill, inner sep=1.5pt] {};
  \draw (b5)--(b6) (b7)--(b6) (b7)--(b8) (b7)--(b9) (b9)--(b10) (b9)--(b11);
\end{scope}
\begin{scope}[xshift=6cm]
  \node (c5) at (0,0.5) [circle,fill, inner sep=1.5pt] {};
  \node (c6) at (0.5,0) [rectangle, draw, inner sep=1.5pt] {};
  \node (c7) at (1, 0.5) [rectangle, draw, inner sep=1.5pt] {};
  \node (c8) at (0.5,1) [circle, fill, inner sep=4pt] {};
  \node (c9) at (1.5,1) [rectangle, draw, inner sep=1.5pt] {};
  \node (c10) at (1,1.5) [circle, fill, inner sep=4pt] {};
  \node (c11) at (2,1.5) [circle, fill, inner sep=1.5pt] {};
  \draw (c5)--(c6) (c7)--(c6) (c7)--(c8) (c7)--(c9) (c9)--(c10) (c9)--(c11);
\end{scope}
\begin{scope}[xshift=9cm]
  \node (d5) at (0,0.5) [circle,fill, inner sep=1.5pt] {};
  \node (d6) at (0.5,0) [rectangle, draw, inner sep=1.5pt] {};
  \node (d7) at (1, 0.5) [rectangle, draw, inner sep=1.5pt] {};
  \node (d8) at (0.5,1) [circle, fill, inner sep=4pt] {};
  \node (d9) at (1.5,1) [rectangle, draw, inner sep=1.5pt] {};
  \node (d10) at (1,1.5) [circle, fill, inner sep=4pt] {};
  \node (d11) at (2,1.5) [rectangle, draw, inner sep=1.5pt] {};
  \node (d12) at (2.5,2.0) [circle,fill, inner sep=1.5pt] {};
  \node (d13) at (1.5,2.0) [circle,fill, inner sep=1.5pt] {};
  \draw (d5)--(d6) (d7)--(d6) (d7)--(d8) (d7)--(d9)
        (d9)--(d10) (d9)--(d11) (d11)--(d12) (d11)--(d13);
\end{scope}
\end{tikzpicture}

\end{center}
\begin{prop}
    Any nontrivial quadratic equation $Q_{T,a,b}$ is equal to a minimal quadratic equation.
\end{prop}
\begin{proof}
    If $a$ and $b$ are separated in $T$ then $([T \setminus a, b], [T \setminus b,a]) = ([T,b],[T,a])$, so $Q_{T,a,b}$ trivially holds. If both $[T,a]$ and $[T,b]$ are not irreducible, then there is $c \in L(T) \setminus \{a,b\}$ that is separated from both $a$ and $b$, so removing $c$ does not change $Q$. Further, if $N(a)=N(b)$, then then $([T,a],[T \setminus a,b])=([T,b], [T \setminus b,a])$, so $Q$ again trivially holds. By doing casework on $|L(T)|$ (which is at most 5), we obtain that the $T$ must be among the $n$-colorings of the four families of marked trees above. 
\end{proof}
\begin{cor} \label{quadratic}
    For (not necessarily distinct) $i,j,k,l$ ranging over $[n]$, the quadratic relations in $\Theta(\partial \mathfrak{T}_3(n))$ are generated by the following families:
    $$\begin{array}{ll}
         C(i,j)B(j) = D(i,j)B(i), & C(i,j)C(j,k) = E(i,j,k)C(i,k), \\
         D(j,k)D(i,j) = E(i,j,k)D(i,k) & E(i,j,k)E(i,k,l) = E(j,k,l)E(i,j,l).
    \end{array}$$
\end{cor}
We now combine the families of linear equations in Corollary \ref{linear} and quadratic equations in Corollary \ref{quadratic} to reduce the number of quadratic generators in $\partial \mathfrak{T}_3(n)$.
\begin{prop} \label{system}
    The ring $\Theta(\partial \mathfrak{T}_3(n))$ is generated by $B(i)$, $C(i,j)$ subject to the following quadratic constraints for all pairwise distinct $i,j,k \in [n]$: \begin{align*} B(i)(B(i)-S(i))&=0 \tag{4.8a}\label{systema} \\ C(i,i)(C(i,i)-S(i))&=0 \tag{4.8b}\label{systemb}\\ C(i,j)C(j,i) \tag{4.8c}\label{systemc} &= 0 \\ C(i,j)(C(i,j)-S(i))&=0 \tag{4.8d}\label{systemd} \\  C(i,j)B(j) &= (B(j) - C(j,i))B(i) \tag{4.8e}\label{systeme} \\    C(i,j)C(j,k)&=(C(j,k)-C(j,i))C(i,k) \tag{4.8f}\label{systemf} \end{align*} where \begin{align*}S(i) &= B(i)+C(i,i)-D(i,i) \\&= B(i) + C(i,i) - \frac{1}{2} \left(B(i)-C(i,i)+\sum_{p \neq i} (C(p,i)-B(p))-1\right). \tag{4.8g}\label{systemg}\end{align*} 
\end{prop}
\begin{proof}
    It is easy to check that the above families of equations are true by expressing the relations in Corollary \ref{quadratic} in terms of $B(i)$ and $C(i,j)$ using the relations in Corollary \ref{linear}. Meanwhile, the proposition essentially claims that the two families in the bottom row of Corollary \ref{quadratic}, which we denote $Q_3$ and $Q_4$ from left to right, are redundant, which can be shown by taking cases on which of $i,j,k,l$ are equal. 
    
    We illustrate the case $i=k$ and $i \neq j$. The right hand side of $Q_3$ vanishes since $E(i,j,i)=0$, while $D(i,j)D(j,i) = (B(j)-C(j,i))(B(i)-C(i,j))$ also vanishes due to \ref{systemc} and \ref{systeme}. The left-hand side of $Q_4$ vanishes for any $l$, and so does the right-hand side for $l=i$ and $l=j$. If $l \not\in \{i,j\}$, then by the bottom right family in Corollary \ref{linear} we still have $E(j,i,l)E(i,j,l)=(C(i,l)-C(i,j))(C(j,l)-C(j,i))$, which is zero by \ref{systemc} and \ref{systemf}.
\end{proof}
Using this simplified presentation, we now classify all measures $\mu: \Theta(\partial \mathfrak{T}_3(n)) \to \CC$. For a marked structure $X$ (along with some marked point), write $\mu(X) = \mathbf{X}$.
\begin{prop} \label{nonzero}
  For any measure $\mu$ and each $n \geq 1$, we have $\mathbf{S}(i) \neq 0$ for all $i \in [n]$.   
\end{prop}
\begin{proof}
    We proceed by induction on $n$. Suppose that $2$ is a unit in $\Theta(\partial \mathfrak{T}_3(n))$. For $n=1$, by \ref{systemg}, we obtain
        $S(1) = \frac{1}{2}(B(1) + 3C(1,1)+1)$, so \ref{systema} and \ref{systemb} imply that \begin{align*}
    B(1)(3C(1,1)+1-B(1)) &= 0, \\ C(1,1)(B(1)+C(1,1)+1) &= 0.
\end{align*} Since $\mu$ is a ring homomorphism and $\CC$ is an integral domain, we split cases as follows. If $\mathbf{B}(1) = 0$, then $(\mathbf{C}(1,1),\mathbf{S}(1)) \in \left\{\left(0,\frac{1}{2}\right),(-1,-1)\right\}$. If $\mathbf{B}(1)=3\mathbf{C}(1,1)+1$, then $(\mathbf{C}(1,1),\mathbf{S}(1)) \in \left\{(0,1), \left(-\frac{1}{2},-\frac{1}{2}\right)\right\}$. We check that all four of these solutions are valid and indeed witness $\mathbf{S}(1) \neq 0$.

Now, for the sake of contradiction, suppose $\mathbf{S}(n)=0$ for some fixed $n>1$. It follows that $\mathbf{B}(n) = \mathbf{C}(n,i) =\mathbf{D}(n,n)=0$ for all $i \in [n]$ and by \ref{systemg}, \begin{align} \label{e1} \sum_{p \neq n} (\mathbf{C}(p,n)-\mathbf{B}(p))=1.\end{align} This means that for all $m \in [n-1]$, \begin{align*}
    \mathbf{S}(m) &= \mathbf{B}(m) + \mathbf{C}(m,m) - \frac{1}{2}\left(\mathbf{B}(m)-\mathbf{C}(m,m)+\sum_{p \neq m}(\mathbf{C}(p,m)-\mathbf{B}(p))-1\right) \\ &= \frac{1}{2}\left(3\mathbf{C}(m,m)+\mathbf{B}(m)-\sum_{p \in [n-1]\setminus m} (\mathbf{C}(p,m)-\mathbf{B}(p))+1\right),
\end{align*} which combined with Equation \ref{e1} gives $$3\mathbf{C}(m,m)+\mathbf{C}(m,n)-2\mathbf{S}(m)+\sum_{p \in [n-1] \setminus m} (\mathbf{C}(p,n)-\mathbf{C}(p,m))=0.$$  We thus obtain the system in Proposition \ref{system} for $i,j,k \in [n-1]$ with the additional constraints of Equation \ref{e1} and the fact that $\mathbf{C}(m,n)(\mathbf{C}(m,n)-\mathbf{S}(m))=0$ as $m$ varies. Hence the inductive hypothesis applies, giving $\mathbf{S}(m) \neq 0$ for all $m \in [n-1]$. Letting $\mathbf{C} = (c_{ij})_{i,j=1}^{n-1} \in \RR^{(n-1) \times (n-1)}$ and $\mathbf{s} = (\mathbf{S}(i))_{i=1}^{n-1} \in \RR^{n-1}$ where \begin{align*}
c_{ij} &=
\begin{cases}
   3\delta_{\mathbf{C}(i,i),\mathbf{S}(i)} + \delta_{\mathbf{C}(i,n),\mathbf{S}(i)} - 2 & \text{if } i=j, \\[6pt]
   \delta_{\mathbf{C}(j,n),\mathbf{S}(j)} - \delta_{\mathbf{C}(j,i),\mathbf{S}(j)}       & \text{if } i\neq j,
\end{cases}
\end{align*}
we have \begin{align} \label{CS} \mathbf{Cs} = \mathbf{0} \end{align} where $\mathbf{0} \in \RR^{n-1}$ and the entries of $\mathbf{C}$ are subject to \ref{systemc}, \ref{systeme}, and \ref{systemf}. Before we proceed with the matrix reformulation of the system, consider the following lemma.  
\begin{lemma}
    Let $G$ be a finite directed acyclic graph (DAG). If $\mathbf{A}$ is its adjacency matrix, then there is a permutation matrix $\mathbf{P}$ such that $\mathbf{PAP^{-1}}$ is upper triangular. \label{perm}
\end{lemma}
\begin{proof}
    Take the topological ordering of the vertex set $V$ given by $v_1 \prec  \dots \prec v_n$. Let $\mathbf{A} = (a_{ij})_{i,j=1}^n$ and $\mathbf{P}$ be the permutation sending the standard order of $V$ to $(v_1, \dots ,v_n)$. For all pairs $(i,j)$ in $\mathbf{PAP^{-1}}$ such that $i < j$, there is either no edge between $v_i$ and $v_j$ or an edge from $v_i$ to $v_j$, both of which result in $a_{ij}=0$. It follows that all entries below the diagonal of $\mathbf{PAP^{-1}}$ are zero as desired. 
\end{proof}
Let $G_{\mathbf{C}}$ be the induced graph from $\mathbf{C}$, i.e., $\mathbf{C}$ is the adjacency matrix of $G_{\mathbf{C}}$. We claim that $G_{\mathbf{C}}$ is a DAG with self-loops adjoined. 

Observe that \begin{align*}
    c_{ij}c_{ji} &= (\delta_{\mathbf{C}(j,n),\mathbf{S}(j)} - \delta_{\mathbf{C}(j,i),\mathbf{S}(j)})(\delta_{\mathbf{C}(i,n),\mathbf{S}(j)} - \delta_{\mathbf{C}(i,j),\mathbf{S}(j)}) \\ &= \delta_{\mathbf{C}(j,n),\mathbf{S}(j)}(\delta_{\mathbf{C}(i,n),\mathbf{S}(i)}-\delta_{\mathbf{C}(i,j),\mathbf{S}(j)})-\delta_{\mathbf{C}(j,i),\mathbf{S}(j)}\delta_{\mathbf{C}(i,n),\mathbf{S}(i)} \\ &= 0, 
\end{align*} where the second equality follows from \ref{systemc} and the last equality from \ref{systemf}. This implies that $G_{\mathbf{C}}$ is a directed graph.

Now suppose for the sake of contradiction that there is a cycle $v_{a_1} \to v_{a_2} \to \cdots \to v_{a_l} \to v_{a_1}$ for $v_{a_1}, \dots, v_{a_l} \in G_{\mathbf{C}}$ where $l \geq 2$. This means that for all $i \in \{1, \dots, l\}$, we have $\delta_{\mathbf{C}(a_{i+1},n),\mathbf{S}(a_{i+1})} - \delta_{\mathbf{C}(a_{i+1},a_i),\mathbf{S}(a_{i+1})} \in \{-1,1\},$ or, equivalently, \begin{align} \label{dag}
    (\mathbf{C}(a_{i+1},n),\mathbf{C}(a_{i+1},a_i)) \in \{(\mathbf{S}(a_{i+1}),0), (0,\mathbf{S}(a_{i+1}))\}
\end{align} where indices are taken mod $l$.

First, suppose that $\mathbf{C}(a_2,n)=\mathbf{S}(a_2)$, so $\mathbf{C}(a_2,a_1)=0$. We repeatedly apply \ref{systemf}. Notice that $$\mathbf{C}(a_2,a_3)\mathbf{C}(a_3,n)=(\mathbf{C}(a_3,n)-\mathbf{C}(a_3,a_2))\mathbf{C}(a_2,n),$$ so if $\mathbf{C}(a_3,n)=0$, then $\mathbf{C}(a_3,a_2)\mathbf{C}(a_2,n)=\mathbf{C}(a_3,a_2)\mathbf{S}(a_2)=0$, which implies that $\mathbf{C}(a_3,a_2)=0$ by the inductive hypothesis. However, this means that $\mathbf{C}(a_3,n)=\mathbf{C}(a_3,a_2)=0$, contradicting \ref{dag}. Hence we must have $\mathbf{C}(a_3,n)=\mathbf{S}(a_3)$ and $\mathbf{C}(a_3,a_2)=0$. By induction, $(\mathbf{C}(a_{i+1},n),\mathbf{C}(a_{i+1},a_i))=(\mathbf{S}(a_{i+1}),0)$ for all $i$, and \begin{align*} 0 &= \mathbf{C}(a_2,a_1)\mathbf{C}(a_1,n) \\ &=(\mathbf{C}(a_1,n)-\mathbf{C}(a_1,a_2)\mathbf{C}(a_2,n)\\ &= (\mathbf{S}(a_1)-\mathbf{C}(a_1,a_2))\mathbf{C}(a_2,n), \end{align*} implying that $\mathbf{C}(a_1,a_2)=\mathbf{S}(a_1)$. However, observe that $$\mathbf{C}(a_1,a_2)\mathbf{C}(a_2,a_3) = (\mathbf{C}(a_2,a_3)-\mathbf{C}(a_2,a_1))\mathbf{C}(a_1,a_3),$$ which gives $\mathbf{C}(a_1,a_3)=\mathbf{C}(a_1,a_2)=\mathbf{S}(a_1)$. Analogously, we obtain $$\mathbf{C}(a_1,a_2)=\mathbf{C}(a_1,a_3) = \dots = \mathbf{C}(a_1,a_{l-1}) = \mathbf{S}(a_1),$$ yet \begin{align*} \mathbf{S}(a_1)\mathbf{S}(a_{l-1}) &= \mathbf{C}(a_1,a_{l-1})\mathbf{C}(a_{l-1},a_l) \\ &= (\mathbf{C}(a_{l-1},a_l)-\mathbf{C}(a_{l-1},a_1))\mathbf{C}(a_1,a_l)\\ &= 0.\end{align*} The left-hand side above is nonzero by the inductive hypothesis, yielding a contradiction.

If instead $\mathbf{C}(a_2,n)=0$ and $\mathbf{C}(a_2,a_1)=\mathbf{S}(a_2)$, then we have $\mathbf{C}(a_1,a_2)=0$, so \begin{align*} 0 &= \mathbf{C}(a_1,a_2)\mathbf{C}(a_2,n) \\ &= (\mathbf{C}(a_2,n)-\mathbf{C}(a_2,a_1))\mathbf{C}(a_1,n) \\ &= -\mathbf{S}(a_2)\mathbf{C}(a_1,n), \end{align*} implying that $\mathbf{C}(a_1,n)=0$. Similarly, $\mathbf{C}(a_{i+1},n)=0$ for all $i$, and \ref{dag} implies that $\mathbf{C}(a_{i+1},a_i)=\mathbf{S}(a_{i+1})$. Further, note that \begin{align*}
    \mathbf{S}(a_3)\mathbf{S}(a_2) &= \mathbf{C}(a_3,a_2)\mathbf{C}(a_2,a_1) \\ &= (\mathbf{C}(a_2,a_1)-\mathbf{C}(a_2,a_3))\mathbf{C}(a_3,a_1) \\ &= \mathbf{S}(a_2)\mathbf{C}(a_3,a_1),
\end{align*} which yields $\mathbf{C}(a_3,a_1)=\mathbf{S}(a_3)$. By an analogous process, we obtain $$\mathbf{C}(a_2,a_1)=\mathbf{C}(a_3,a_1)=\dots=\mathbf{C}(a_l,a_1)=\mathbf{S}(a_l),$$ so $\mathbf{C}(a_l,a_1)\mathbf{C}(a_1,a_l)=\mathbf{S}(a_l)\mathbf{S}(a_1)=0$, which contradicts the inductive hypothesis.

It thus follows that $G_{\mathbf{C}}$, ignoring self-loops, is a DAG. By Lemma \ref{perm}, there is a permutation matrix $\mathbf{P}$ such that $\mathbf{PCP^{-1}}$ is upper triangular. Since performing row-column permutations on $\mathbf{C}$ does not change its set of diagonal elements, we have \begin{align*}
    \det(\mathbf{C}) &= \det(\mathbf{PCP^{-1}}) \\ &= \prod_{i=1}^{n-1} c_{ii} \\ &= \prod_{i=1}^{n-1} (3\delta_{\mathbf{C}(i,i),\mathbf{S}(i)} + \delta_{\mathbf{C}(i,n),\mathbf{S}(i)} - 2) \\ &= \pm 2^k \neq 0,
\end{align*} for a suitable nonnegative integer $k$, implying that there are no non-trivial solutions to Equation $\ref{CS}$. This is the desired contradiction.
\end{proof}

The above proposition implies that the Kronecker deltas $\delta_{\mathbf{B}(i),\mathbf{S}(i)}$, $\delta_{\mathbf{C}(i,i),\mathbf{S}(i)}$, and $\delta_{\mathbf{C}(i,j),\mathbf{S}(i)}$ are zero if and only if $\mathbf{B}(i)$, $\mathbf{C}(i,i)$, and $\mathbf{C}(i,j)$ are zero, which motivates the following object.
\begin{defn} \label{nqr}
    An \emph{$n$-code} is a pair of functions $(\beta,\chi)$ where $\mathcal{\beta}: [n] \to \{0,1\}$ and $\chi : [n] \times [n] \to \{0,1\}$ satisfy the following properties for pairwise distinct $i,j,k \in [n]$:
    \begin{align}
    \tag{4.11a} \label{nqra} \chi(i,j)\chi(j,i)&= 0 \\
    \tag{4.11b} \label{nqrb} \beta(i)\beta(j)&=\beta(i)\chi(j,i)+\beta(j)\chi(i,j) \\
    \tag{4.11c} \label{nqrc} \chi(i,k)\chi(j,k) &= \chi(i,j)\chi(j,k) + \chi(j,i)\chi(i,k).
\end{align}
Note that \ref{nqra}, \ref{nqrb}, and \ref{nqrc} are reformulations of \ref{systemc}, \ref{systeme}, and \ref{systemf}, respectively, and there are no constraints involving $\chi(i,i)$.
\end{defn}
\begin{prop} \label{biject}
    There is a natural bijection between measures on $\partial \mathfrak{T}_3(n)$ and $n$-codes where $\delta_{\mathbf{B}(i),\mathbf{S}(i)}$, $\delta_{\mathbf{C}(i,i),\mathbf{S}(i)}$, and $\delta_{\mathbf{C}(i,j),\mathbf{S}(i)}$ correspond to $\beta(i)$, $\chi(i,i)$, and $\chi(i,j)$, respectively.
\end{prop}
\begin{proof}
    We equivalently show that each choice of $\delta_{B(i),S(i)}$,  $\delta_{\mathbf{C}(i,i),\mathbf{S}(i)}$, and $\delta_{\mathbf{C}(i,j),\mathbf{S}(i)}$ satisfying \ref{systemc}, \ref{systeme}, and \ref{systemf} yields a unique solution to the family of linear equations in \ref{systemg}. Notice that rewriting \ref{systemg} gives \begin{align*}
     2\mathbf{S}(i)-3\mathbf{C}(i,i)-\mathbf{B}(i)+\sum_{p \neq i}(\mathbf{C}(p,i)-\mathbf{B}(p)) = 1.
    \end{align*} Following the approach in Proposition \ref{nonzero}, let $\mathbf{D}=(d_{ij})_{i,j=1}^n \in \RR^{n \times n}$ and $\mathbf{s}=(\mathbf{S}(i))_{i=1}^{n} \in \RR^n$ where \begin{align*}
d_{ij} &=
\begin{cases}
   2-3\delta_{\mathbf{C}(i,i),\mathbf{S}(i)}-\delta_{\mathbf{B}(i),\mathbf{S}(i)} & \text{if } i=j, \\[6pt]
 \delta_{\mathbf{C}(j,i),\mathbf{S}(j)}-\delta_{\mathbf{B}(j),\mathbf{S}(j)}      & \text{if } i\neq j.
\end{cases}
\end{align*}
We have \begin{align} \label{ds}
    \mathbf{Ds} &= \begin{pmatrix}
d_{11} & d_{12} & \dotsc  & d_{1n}\\
d_{21} & d_{22} & \dotsc  & d_{2n}\\
\vdots  & \vdots  & \ddots  & \vdots \\
d_{n1} & d_{n2} & \dotsc  & d_{nn}
\end{pmatrix}\begin{pmatrix}
S( 1)\\
S( 2)\\
\vdots \\
S( n)
\end{pmatrix} =\begin{pmatrix}
1\\
1\\
\vdots \\
1
\end{pmatrix}
\end{align}
where the right-hand side is the $n \times 1$ all-ones matrix. It suffices to show that $\mathbf{D}$ is invertible, or $\det(\mathbf{D}) \neq 0$. In fact, we claim that $G_{\mathbf{D}}$, the induced graph from $\mathbf{D}$, is a DAG with self-loops adjoined and existing at every vertex. 

First, we have \begin{align*} 
    d_{ij}d_{ji} &= (\delta_{\mathbf{C}(j,i),\mathbf{S}(j)}-\delta_{\mathbf{B}(j),\mathbf{S}(j)})(\delta_{\mathbf{C}(i,j),\mathbf{S}(i)}-\delta_{\mathbf{B}(i),\mathbf{S}(i)}) \\ &= \delta_{\mathbf{B}(j),\mathbf{S}(j)}\delta_{\mathbf{B}(i),\mathbf{S}(i)}-\delta_{\mathbf{B}(j),\mathbf{S}(j)}\delta_{\mathbf{C}(i,j),\mathbf{S}(i)}- \delta_{\mathbf{B}(i),\mathbf{S}(i)}\delta_{\mathbf{C}(j,i),\mathbf{S}(j)} \\ &= 0
\end{align*} where the last equality follows from \ref{systeme}. Thus, $G_{\mathbf{D}}$ is a directed graph. 

We now claim that $G_{\mathbf{D}}$ has no directed cycle of length at least two. For the sake of contradiction, suppose there is a cycle $v_{a_1} \to v_{a_2} \to \dots \to v_{a_l} \to v_{a_1}$ for $v_{a_1}, \dots, v_{a_l} \in G_{\mathbf{D}}$ where $l \geq 2$. This means that for each $i \in \{1,\dots,l\}$, we have \begin{align} \label{choose}
(\mathbf{C}(a_{i+1},a_i),\mathbf{B}(a_{i+1})) \in \{(0,\mathbf{S}(a_{i+1})), (\mathbf{S}(a_{i+1}),0)\}
\end{align} where all indices are taken mod $l$. 

If $(\mathbf{C}(a_2,a_1),\mathbf{B}(a_2))=(0,\mathbf{S}(a_2))$, then by \ref{systeme}, we have \begin{align*}
    \mathbf{C}(a_1,a_2)\mathbf{B}(a_2) = (\mathbf{B}(a_2)-\mathbf{C}(a_2,a_1))\mathbf{B}(a_1) = \mathbf{B}(a_2)\mathbf{B}(a_1).  
\end{align*} Since $\mathbf{B}(a_2)=\mathbf{S}(a_2) \neq 0$, it follows that $\mathbf{C}(a_1,a_2)=\mathbf{B}(a_1)=\mathbf{S}(a_1)$. Similarly, $\mathbf{C}(a_i,a_{i+1})=\mathbf{S}(a_i)$ for all $i$. Applying \ref{systemf}, we have $$\mathbf{C}(a_1,a_2)\mathbf{C}(a_2,a_3) = (\mathbf{C}(a_2,a_3)-\mathbf{C}(a_2,a_1))\mathbf{C}(a_1,a_3) = \mathbf{C}(a_2,a_3)\mathbf{C}(a_1,a_3),$$ which implies that $\mathbf{C}(a_1,a_2)=\mathbf{C}(a_1,a_3)=\mathbf{S}(a_1)$. Analogously, $$\mathbf{C}(a_1,a_2) = \dots = \mathbf{C}(a_1,a_{l-1}) = \mathbf{S}(a_1).$$ However, \begin{align*}
    \mathbf{S}(a_1)\mathbf{S}(a_{l-1}) &= \mathbf{C}(a_1,a_{l-1})\mathbf{C}(a_{l-1},a_l) \\ &= (\mathbf{C}(a_{l-1},a_l)-\mathbf{C}(a_{l-1},a_1))\mathbf{C}(a_1,a_l) \\ &= 0, 
\end{align*} which is absurd. 

If instead $(\mathbf{C}(a_2,a_1),\mathbf{B}(a_2))=(\mathbf{S}(a_2),0)$, then $\mathbf{C}(a_1,a_2)=0$, and \ref{systeme} yields \begin{align*}
   \mathbf{S}(a_2)\mathbf{B}(a_1) &=  \mathbf{C}(a_2,a_1)\mathbf{B}(a_1) \\ &= (\mathbf{B}(a_1)-\mathbf{C}(a_1,a_2))\mathbf{B}(a_2) \\ &= 0,  
\end{align*} so $B(a_1)=0$. We thus obtain $\mathbf{B}(a_i)=0$ for all $i$, which by \ref{choose} implies that $\mathbf{C}(a_{i+1},a_i)=\mathbf{S}(a_{i+1})$ for all $i$. Now, observe that \begin{align*}
    \mathbf{S}(a_2)\mathbf{S}(a_1) &= \mathbf{C}(a_2,a_1)\mathbf{C}(a_1,a_l) \\ &= (\mathbf{C}(a_1,a_l)-\mathbf{C}(a_1,a_2))\mathbf{C}(a_2,a_l) \\ &= \mathbf{S}(a_1)\mathbf{C}(a_2,a_l)
\end{align*} so $\mathbf{C}(a_2,a_l)=\mathbf{S}(a_2)$. By induction, $\mathbf{C}(a_{l-2},a_l)=\mathbf{S}(a_{l-2})$, but \begin{align*}
    \mathbf{S}(a_{l-1})\mathbf{S}(a_{l-2}) &= \mathbf{C}(a_{l-1},a_{l-2})\mathbf{C}(a_{l-2},a_l) \\ &= (\mathbf{C}(a_{l-2},a_l)-\mathbf{C}(a_{l-2},a_{l-1}))\mathbf{C}(a_{l-1},a_l) \\ &= 0,
\end{align*} which is the desired contradiction, and the claim follows.

Since $G_{\mathbf{D}}$ is a DAG ignoring self loops, again by Lemma \ref{perm}, there is a permutation matrix $\mathbf{P}$ such that $\mathbf{PDP^{-1}}$ is upper triangular, so \begin{align*}
    \det(\mathbf{D}) &= \det(\mathbf{PDP^{-1}}) \\ &= \prod_{i=1}^n d_{ii} \\ &= \prod_{i=1}^n (2-3\delta_{C(i,i),S(i)} - \delta_{B(i),S(i)}) \\ &= \pm 2^k \neq 0
\end{align*} for some nonnegative integer $k$, giving the proposition.
\end{proof}
\begin{corollary} \label{image}
    All measures on $\partial \mathfrak{T}_3(n)$ are valued in $\ZZ \left[\frac 12 \right]$. 
\end{corollary}
\begin{proof}
Since $\mathbf{D}$ has integer entries and determinant a power of two up to sign, the entries of $\mathbf{D}^{-1}$ lie in $\ZZ\left[\frac 12\right]$. Recall that each class of irreducible marked trees in $\Theta(\partial \mathfrak{T}_3(n))$ has measure either zero or equal to that of $S(i)$. The latter is equal to the sum of the entries in row $i$ of $\mathbf{D}^{-1}$, which is an element of $\ZZ\left[\frac 12\right]$.
\end{proof}
\begin{remark}
    In fact, applying a theoretical result due to Nekrasov and Snowden \cite[Theorem 3.4]{nekrasov2024upperboundsmeasuresdistal}, we can show that all measures on $\partial \mathfrak{T}_3(n)$ have images in $\CC$ that are either zero or a power of two up to a sign. The result states that if $m$ divides $|\mathrm{Aut}(X)|$ for all $X \in \mathfrak{F}$ for some Fra\"iss\'e class $\mathfrak{F}$, then values of measures on $\mathfrak{F}$ are rationals only witnessing prime divisors of $m$. Indeed, we later show that $2$ is the only prime dividing $|\mathrm{Aut}(T)|$ for all $T \in \partial \mathfrak{T}_3(n)$ (Proposition \ref{poweroftwo}).
\end{remark}
\begin{theorem} \label{maintheorem}
Given a directed rooted tree $T$ with edges labeled by $\{1,\dots,n\}$ and a distinguished vertex $v$ (which could be the root), there is a corresponding $\ZZ \left[\frac 12 \right]$-valued measure $\mu_{T,v}$ on $\partial \mathfrak{T}_3(n)$. As a result, we have an isomorphism of rings $\Theta(\partial \mathfrak{T}_3(n)) \cong \ZZ\left[\frac 12 \right]^{(2n+2)^n}$.
\end{theorem}
\begin{proof}
    Let $\mathcal{Q}_n$ denote the set of $n$-codes and $\mathcal{T}_n$ denote the set of $[n]$-labeled directed rooted trees with a distinguished vertex. For a tree $T$, let $r(T)$ be its root and $v(T)$ be its distinguished vertex. We write $i$ for the edge labeled by $i$, and say that $i$ \emph{points towards} a vertex $u$ if the path from the tail of $i$ to $u$ contains $i$. Also, by the path from $u$ to $i$, we mean the path from $u$ to the endpoint further away from $u$ among the two endpoints of $i$. Consider the map $\Psi: \mathcal{T}_n \to \mathcal{Q}_n$ given by $\Psi(T)=(\beta_T,\chi_T)$ such that the following rules hold for all distinct $i,j \in [n]$: \begin{align*}
        \chi_T(i,i) &= \begin{cases} 1 & \text{if $i$ points towards $r(T)$} , \\[6pt]
            0 & \text{otherwise},
            \end{cases} \\
        \beta_T(i) &= \begin{cases}
            1 & \text{if $i$ lies on the simple undirected path from $r(T)$ to $v(T)$} , \\[6pt]
            0 & \text{otherwise},
        \end{cases} \\
        \chi_T(i,j) &= \begin{cases}
            1 & \text{if $i$ lies on the simple undirected path from $r(T)$ to $j$} , \\[6pt]
            0 & \text{otherwise}.
        \end{cases}
    \end{align*}
We first check that for any tree $T \in \mathcal{T}$, its image satisfies the properties in Definition \ref{nqr}. For a vertex $u$ and an edge $i=\{a,b\}$, define $d(u,i) = \min\{d(u,a), d(u,b)\}$. If $\chi_T(i,j)=1$, then $d(u,i) < d(u,j)$, forcing $\chi_T(j,i)=0$. Thus, $\chi_T(i,j)\chi_T(j,i)=0$ for all distinct $i,j$, satisfying \ref{nqra}. 

If $\chi_T(i,j)=\chi_T(j,i)=0$ for some distinct $i,j$, then it is impossible for $i,j$ to both lie on the path from $r(T)$ to $v(T)$, since that would cause one of them to lie on the path from $r(T)$ to the other. As a result, $\beta_T(i)\beta_T(j)=0$, which satisfies \ref{nqrb}. Further, if there is $k \in [n] \setminus \{i,j\}$ such that $\chi_T(i,k)=\chi_T(j,k)=1$, then the path from $r(T)$ to $k$ contains both $i$ and $j$, imposing an order between $i$ and $j$, which is a contradiction. Thus we also have $\chi_T(i,k)\chi_T(j,k)=0$ for all such $k$, satisfying $\ref{nqrc}$. 

If $(\chi_T(i,j),\chi_T(j,i))=(1,0)$, then if $\beta(j)=1$, then $i$ must also lie on the path from $r(T)$ to $v(T)$, so $\beta(i)=1$. On the other hand, $\beta(i)=0$ analogously implies that $\beta(j)=0$. Hence $\beta(j)(\beta(i)-1)=0$, so \ref{nqrb} still holds. Moreover, for any $k \in [n] \setminus \{i,j\}$, if $\chi_T(j,k)=1$, then $i$ must also lie on the path from $r(T)$ to $k$, so $\chi_T(i,k)=1$. If $\chi_T(i,k)=0$, then $j$ could not have lied on the path from $r(T)$ to $k$, so $\chi_T(j,k)=0$. It follows that $\chi_T(j,k)(\chi_T(i,k)-1)$, satisfying $\ref{nqrc}$.

The case $(\chi_T(i,j),\chi_T(j,i))=(0,1)$ is similar, and there are no constraints on $\chi_T(i,i)$ in Definition \ref{nqr}. Hence, we have exhausted all cases, and $\Psi$ is a valid map.

Now we show that $\Psi$ is a surjection. Suppose $(\beta, \chi)$ is an $n$-code. Define $$B = \{b \in [n]: \beta(b)=1\} = \{b_1, \dots, b_m\}.$$ By \ref{nqrb}, we have $\chi(b_i,b_j)+\chi(b_j,b_i)=1$ for all distinct $i,j$, so \ref{nqra} implies that $(\chi(b_i,b_j), \chi(b_j,b_i)) \in \{(0,1),(1,0)\}$. Further, \ref{nqrc} requires transitivity where $\chi(b_i,b_j)=\chi(b_j,b_k)=1$ implies that $\chi(b_i,b_k)=1$, so we suppose without loss of generality that $(\chi(b_{i+1},b_i), \chi(b_i,b_{i+1}))=(0,1)$ for all $i \in \{1, \dots, m-1\}$. This induces a natural ordering $b_1  \prec \dots \prec b_m$, so consider a tree $T \in \mathcal{T}$ where the subtree induced by edges in $B$ is a rooted path in which $d(r(T),b_i) < d(r(T),b_{i+1})$. Also, let the endpoint of $b_m$ that is further away from $r(T)$ be the distinguished vertex in $T$. The picture of $T|_B$  (the induced subgraph on vertices connected by edges in $B$) is the following:
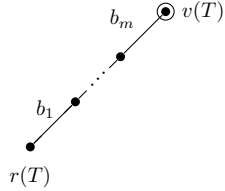
\begin{center}
        \begin{tikzpicture}[scale=0.8, transform shape]
            \node (0) at (0,0) [circle, fill, inner sep=1.5pt] {};
            \node at (0,-0.5) {$r(T)$};
            \node (2) at (3/4,3/4) [circle, fill, inner sep = 1.5pt] {};
            \draw (0)--(2);
            \node at (2/8,5/8) {$b_1$};
            \draw (0)--(3/16,3/16);
            \draw (18/32,18/32) -- (2);
            \node (d4) at (33/32,33/32) {$\cdot$};
            \node (d5) at (18/16,18/16) {$\cdot$};
            \node (d6) at (39/32,39/32) {$\cdot$};
            \node (4) at (3/2,3/2) [circle, fill, inner sep = 1.5pt] {};
            \draw (2) -- (30/32,30/32);
            \draw (42/32,42/32) -- (4);
            \node (5) at (9/4,9/4) [circle, fill, inner sep = 1.5pt] {};
            \draw[thin] (5) circle[radius=4pt];
            \draw (4) -- (5);
            \node at (13/8.5,16/7.5) {$b_m$};
            \node at (23/8,9/4) {$v(T)$};
        \end{tikzpicture}
    \end{center}
We claim that $T|_B$ uniquely extends to a tree $T$ satisfying $\Psi(T)=(\beta,\chi)$. We proceed by induction on $|[n] \setminus B|$. If $|[n]\setminus B|=0$ then $T|_B = T$. Otherwise, consider any $c' \in [n]\setminus B$, and let $k$ be the maximal index such that $\chi(b_i,c')=1$ for all $1 \leq i \leq k$. Let
$$C_k = \{c \in [n]\setminus B: (\chi(b_i,c), \chi(b_{k+1},c))=(0,1) \, \,\forall 1 \leq i \leq k\}.$$
and define the following pairwise disjoint subsets of $C_k$: \begin{align*}
    X &= \{c \in C_k: (\chi(c,c'),\chi(c',c))=(0,0)\} \\
    Y &= \{c \in C_k: (\chi(c,c'),\chi(c',c))=(0,1)\} \\
    Z &= \{c \in C_k: (\chi(c,c'),\chi(c',c))=(1,0)\}.
\end{align*}By the inductive hypothesis, it suffices to (uniquely) extend $T\setminus c'$ to $T$ by inserting the edge $c'$ such that \ref{nqrc} holds for $\Psi(T)$. Indeed, first note that the path from $r(T)$ to $c'$ cannot pass through any edges in $X$; in fact, the path can only contain edges in $Y$. There is a unique position for $c'$ in $T$ such that all edges in $Y$ precede $c'$ on the path and all edges in $Z$ succeed $c'$. This completes the induction showing that the constructed $T$ uniquely satisfies $\Psi(T) = (\beta,\chi)$, implying that $\Psi$ is both surjective and injective, so it is a bijection. 

Finally, note that there are $(n+1)^{n-2}$ trees with edges labeled by $[n]$ by a well-known variant of Cayley's theorem \cite[Proposition 2.1]{Cameron95}. There are each $n+1$ ways to choose a root and a distinguished vertex among the $n+1$ vertices, and $2^n$ ways to orient the edges with no non-trivial automorphisms, giving $$|\mathcal{T}_n| = |\mathcal{Q}_n| = 2^n(n+1)^n = (2n+2)^n,$$ as desired.
\end{proof}
\begin{remark}
    Using the oligomorphic group perspective, \cite{snowden2023measurescoloredcircle} determined that there are also $(2n+2)^n$ measures on the Fra\"iss\'e class $\mathfrak{F}(n)$ of $n$-colored totally ordered sets. We note that there is a natural bijection from the set of measures of $\partial \mathfrak{T}_3(n)$ to that of $\mathfrak{F}(n)$ given by taking the sign of the measure; indeed, measures on $\mathfrak{F}(n)$ are valued in $\{-1,0,1\}$. There is also a natural bijection between their classes of irreducible marked structures, where $B(i)$ corresponds to the one-point marked structure with color $i$; $C(i,j)$, $D(i,j)$ correspond to the two-point structures with the smaller element colored $i$ (and marked for $C(i,j)$) and larger element colored $j$ (and marked for $D(i,j)$); $E(i,j,k)$ corresponds to the three-element structure with the middle element marked and colors $i,j,k$ in increasing order. The relations among them are also identical except for the existence of the sum $\sum_{p \in [n]} D(i,p)$ in the relations inside $\Theta(\partial \mathfrak{T}_3(n))$, which is due to suppressing vertices of degree two and differentiates the relations in the two cases.
\end{remark}
\subsection{An example} \label{examp} We now illustrate an example of how a measure $\mu_{T,v}$ on $\partial \mathfrak{T}_3(n)$ can be produced given an $[n]$-labeled directed rooted tree $T$ with a distinguished vertex $v$. Consider the following tree $T \in \mathcal{T}_8$:
\begin{center}
    \begin{tikzpicture}[scale = 0.8, transform shape]
        \node (0) at (0,0) [circle, fill, inner sep = 1.5pt] {};
        \node (1) at (0,1) [circle, fill, inner sep = 1.5pt] {};
        \node (2) at (-1,1) [circle, fill, inner sep = 1.5pt] {};
        \node (3) at (1,1) [circle, fill, inner sep = 1.5pt] {};
        \node (4) at (-0.75,2) [circle, fill, inner sep = 1.5pt] {};
        \node (5) at (0.75,2) [circle, fill, inner sep = 1.5pt] {};
        \node (6) at (1.25,2) [circle, fill, inner sep = 1.5pt] {};
        \node (7) at (-1.25,2) [circle, fill, inner sep = 1.5pt] {};
        \node (8) at (0,2) [circle, fill, inner sep = 1.5pt] {};
        \draw[thin] (8) circle[radius=4pt];

        \draw[mid arrow] (0) -- (1);
        \draw[mid arrow] (2) -- (0);
        \draw[mid arrow] (0) -- (3);
        \draw[mid arrow] (2) -- (4);
        \draw[mid arrow] (5) -- (3);
        \draw[mid arrow] (3) -- (6);
        \draw[mid arrow] (7) -- (2);
        \draw[mid arrow] (1) -- (8);

        \node at (-1.35,1.5) {\small{$1$}};
        \node at (-0.65,1.5) {\small{$2$}};
        \node at (-0.65,0.3) {\small{$3$}};
        \node at (0.65,0.3) {\small{$6$}};
        \node at (1.35,1.5) {\small{$8$}};
        \node at (0.65,1.5) {\small{$7$}};
        \node at (-0.22,0.5) {\small{$4$}};
        \node at (0.22,1.5) {\small{$5$}};
    \end{tikzpicture}
\end{center}
The circled vertex $v$ is distinguished. We have $\beta(i)=1$ for $i \in \{4,5\}$ and $\beta(i)=0$ for $i \in \{1,2,3,6,7,8\}$. The directed edges of $T$ give $\chi(i,i)=1$ for $i \in \{1,3,7\}$ and $\chi(i,i)=0$ for $i \in \{2,4,5,6,8\}$. Tracing paths from the root to the leaves of $T$, we have the following table, where the value of $\chi(i,j)$ appears in row $i$, column $j$:
\[
\begin{array}{c|cccccccc}
  & 1& 2& 3& 4& 5& 6& 7& 8\\ \hline
1 & 1& 0& 0& 0& 0& 0& 0& 0\\
2 & 0& 0& 0& 0& 0& 0& 0& 0\\
3 & 1& 1& 1& 0& 0& 0& 0& 0\\
4 & 0& 0& 0& 0& 1& 0& 0& 0\\
5 & 0& 0& 0& 0& 0& 0& 0& 0\\
6 & 0& 0& 0& 0& 0& 0& 1& 1\\
7 & 0& 0& 0& 0& 0& 0& 1& 0\\
8 & 0& 0& 0& 0& 0& 0& 0& 0
\end{array}
\]
Under the bijection in Proposition \ref{biject}, we construct the associated matrix $$\mathbf{D}_T = \begin{pmatrix}
-1 & 0 & 1 & -1 & -1 & 0 & 0 & 0\\
0 & 2 & 1 & -1 & -1 & 0 & 0 & 0\\
0 & 0 & -1 & -1 & -1 & 0 & 0 & 0\\
0 & 0 & 0 & 1 & -1 & 0 & 0 & 0\\
0 & 0 & 0 & 0 & 1 & 0 & 0 & 0\\
0 & 0 & 0 & -1 & -1 & 2 & 0 & 0\\
0 & 0 & 0 & -1 & -1 & 1 & -1 & 0\\
0 & 0 & 0 & -1 & -1 & 1 & 0 & 2
\end{pmatrix}$$
which results in $$\mathbf{s} = (\mathbf{S}(i))_{i=1}^8 = \mathbf{D}_T^{-1} \mathbf{1}_8 = 
    (-8,4,-4,2,1,2,-2,1)^{\top}.$$ It also follows that \begin{align*}
    \mathbf{A} = \mu_{T,v}(A) &= 1+\sum_{i=1}^8 \mathbf{B}(i) \\ &= 1+\mathbf{S}(4) + \mathbf{S}(5) \\ &=  4.
\end{align*}We thus find, for instance, that \begin{center}
\begin{tikzpicture}
  \node (d5)  at (0,0.5)   [circle,fill, inner sep=1.5pt] {};
  \node (d6)  at (0.5,0)   [rectangle, draw, inner sep=1.5pt] {2};
  \node (d7)  at (1,0.5)   [rectangle, draw, inner sep=1.5pt] {1};
  \node (d8)  at (0.5,1)   [circle, fill, inner sep=1.5pt] {};
  \node (d9)  at (1.5,1)   [rectangle, draw, inner sep=1.5pt] {3};
  \node (d10) at (1,1.5)   [circle, fill, inner sep=1.5pt] {};
  \node (d11) at (2,1.5)   [circle, fill, inner sep=1.5pt] {};
  \draw (d5)--(d6) (d7)--(d6) (d7)--(d8) (d7)--(d9)
        (d9)--(d10) (d9)--(d11);
  \path (0,-0.25) coordinate (bbSW);
  \path ( 2, 1.75) coordinate (bbNE);
  \node[anchor=east] at (bbSW |- 0,0.75)
    {$\left(\vphantom{\rule{0pt}{1cm}}\right.$};
  \node[anchor=west] at (bbNE |- 0,0.75)
    {$\left.\vphantom{\rule{0pt}{1cm}}\right)$};
  \node[anchor=east] at (-0.4,0.75) {$\mu_{T,v}$};
  \node[anchor=west] at (2.50,0.75)
    {$= \mathbf{AB}(3)\mathbf{C}(1,3)\mathbf{C}(2,1) = 0$};
\end{tikzpicture}
\end{center}
while
\begin{center}
\begin{tikzpicture}
  \node (d5)  at (0.5,0.5)   [rectangle,draw, inner sep=1.5pt] {5};
  \node (d6)  at (1,0)   [rectangle, draw, inner sep=1.5pt] {4};
  \node (d7)  at (1.5,0.5)   [rectangle, draw, inner sep=1.5pt] {5};
  \node (d8)  at (1.25,1)   [circle, fill, inner sep=1.5pt] {};
  \node (d9)  at (1.75,1)   [circle, fill, inner sep=1.5pt] {};
  \node (1) at (0.75,1) [circle, fill, inner sep = 1.5pt] {};
  \node (2) at (0.25,1) [circle, fill, inner sep = 1.5pt] {};

  \draw (d5)--(d6) (d7)--(d6) (d7)--(d8) (d7)--(d9) (1)--(d5) (2)--(d5);

  \path (0.2,-0.25) coordinate (bbSW);
  \path (1.8, 1.75) coordinate (bbNE);

  \node[anchor=east] at (bbSW |-0,0.5)
    {$\left(\vphantom{\rule{0pt}{0.8cm}}\right.$};
  \node[anchor=west] at (bbNE |- 0,0.5)
    {$\left.\vphantom{\rule{0pt}{0.8cm}}\right)$};

  \node[anchor=east] at (-0.2,0.5) {$\mu_{T,v}$};

  \node[anchor=west] at (2.3,0.5)
    {$=\mathbf{AB}(4)\mathbf{D}(4,5)^2 = \mathbf{A}\mathbf{S}(4)\mathbf{S}(5)^2 = 8$.};
\end{tikzpicture}
\end{center}
In general, for this example, if at least one of the colors in $\{1,2,3,6,7,8\}$ appears in a tree $t$, then $\mu_{T,v}(t)=0$. 
\subsection{Quasi-regular measures} Recall that a measure $\mu$ on a Fra\"iss\'e class $\mathfrak{F}$ with Fra\"iss\'e limit $\Omega$ is quasi-regular (Definition \ref{quasidef}) if there exists an open subgroup $U$ of $\mathrm{Aut}(\Omega)$ such that the corresponding measure $\mu^*$ is regular on the class of transitive $\mathscr{E}$-smooth $U$-sets. We showed that quasi-regularity corresponds to regularity on a relative Fra\"iss\'e class (Corollary \ref{quasicor}) and now consider quasi-regular measures on $\partial \mathfrak{T}_3(n)$.
\begin{prop}
    For each $n \geq 1$, there are no quasi-regular measures on $\partial \mathfrak{T}_3(n)$ that are not regular.
\end{prop}
\begin{proof}
    Let $\mu: \Theta(\partial \mathfrak{T}_3(n)) \to \CC$ be a non-regular measure. By Corollary \ref{quasicor}, it suffices to show that there does not exist $T \in \partial\mathfrak{T}_3(n)$ such that $\mu(T \to X)$ is nonzero for all $X$ which $T$ embeds into. Suppose for the sake of contradiction that such $T$ exists. 
    
    For $n=1$, consider the embedding of $T$ into the following tree $X$
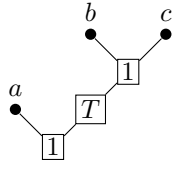
    \begin{center}
    \begin{tikzpicture}[scale=0.5]
    \node (0) at (0,0) [rectangle, draw, inner sep = 1.5pt] {$1$};
    \node (1) at (-1,1) [circle, fill, inner sep = 1.5pt] {};
    \node (2) at (1,1) [rectangle, draw, inner sep = 2pt] {$T$};
    \node (3) at (2,2) [rectangle, draw, inner sep = 1.5pt] {$1$};
    \node (4) at (1,3) [circle, fill, inner sep = 1.5pt] {};
    \node (5) at (3,3) [circle, fill, inner sep = 1.5pt] {};
    \draw (0)--(1);
    \draw (0)--(2);
    \draw (2)--(3);
    \draw (3)--(4);
    \draw (3)--(5);
    \node at (-1,1.5) {$a$};
    \node at (1,3.6) {$b$};
    \node at (3,3.5) {$c$};
    \end{tikzpicture}
    \end{center}
where $a,b \in X \setminus T$ and $c \in X \cap T$. We have \begin{align*}
    \mu(T \to X) &= \mu([X \setminus b, a])\mu([X,b]) \\ &= \mu(C(1,1))\mu(D(1,1)).
\end{align*} From \ref{systema} and \ref{systemb}, we have $$B(1)C(1,1) = B(1)C(1,1) = C(1,1)D(1,1),$$ and since $\mu$ is not regular, at least two of the above terms vanish under $\mu$. Consequently, $\mu(T \to X) = \mu(C(1,1))\mu(D(1,1))=0$, a contradiction.

For $n \geq 2$, suppose the root of $T$ has color $i$ and $j$ is a color distinct from $i$. Take $X$ to be as follows:
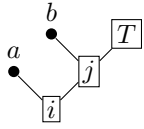
\begin{center}
    \begin{tikzpicture}[scale=0.5]
    \node (0) at (0,0) [rectangle, draw, inner sep = 1.5pt] {$i$};
    \node (2) at (1,1) [rectangle, draw, inner sep = 1.5pt] {$j$};
    \node (3) at (2,2) [rectangle, draw, inner sep = 2pt] {$T$};
    \node (4) at (0,2) [circle, fill, inner sep = 1.5pt] {};
    \node (5) at (-1,1) [circle, fill, inner sep = 1.5pt] {};
    \draw (0)--(2);
    \draw (2)--(3);
    \draw (2)--(4);
    \draw (0)--(5);
    \node at (-1,1.5) {$a$};
    \node at (0,2.6) {$b$}; 
    \end{tikzpicture}
\end{center}
We have \begin{align*}
    \mu(T \to X) &= \mu([X \setminus a,b])\mu([X,a]) \\ &= \mu(C(j,i))\mu(C(i,j)) \\ &= 0
\end{align*} where the last equality follows from \ref{systemc}. We thus obtain the desired contradiction, and $\mu$ is not quasi-regular.
\end{proof}
\section{\texorpdfstring{Fraïssé subclasses of $\partial \mathfrak{T}_3(n)$ and induced regular measures}{Fraïssé subclasses of delta T3(n) and the induced regular measure}} \label{sec5}
In this section, we investigate subclasses of $\partial \mathfrak{T}_3(n)$ that are the support of some measure. Since these subclasses (known as induced subclasses) are Fra\"iss\'e, we begin by classifying all Fra\"iss\'e subclasses of $\partial \mathfrak{T}_3(n)$, providing a recursive formula to count them. Afterwards, we explicitly describe the set of induced subclasses and their unique induced regular measure with a formula. These subclasses, along with their measures, are used to construct semisimple tensor categories in Section \ref{category}. At the end of the section, we illustrate our classification results for the $n=2$ case to concretely demonstrate the values of measures on the subclasses.
\subsection{\texorpdfstring{Classification of Fraïssé subclasses of $\partial \mathfrak{T}_3(n)$}{Classification of Fraïssé subclasses of delta T3(n)}} \label{classification} Given a Fraïssé class $\mathfrak{F}$, the classification and enumeration of its subclasses is not only an interesting combinatorial problem, but also has numerous implications in further analysis of measures in \ref{induced}.

Let $\mathfrak{T}'$ be a Fraïssé subclass of $\partial \mathfrak{T}_3(n)$ and $X_{i,j}$ denote the following tree for colors $i,j \in [n]$: 
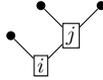
\begin{center}
    \begin{tikzpicture}[scale = 0.8, transform shape]
        \node (5) at (1,0.5) [circle,fill, inner sep=1.5pt] {};

  \node (6) at (1.5,0) [rectangle, draw, inner sep=1.5pt] {$i$};
  \node (7) at (2, 0.5) [rectangle, draw, inner sep = 1.5pt] {$j$};
  \node (8) at (1.5,1) [circle, fill, inner sep=1.5pt] {};
  \node (9) at (2.5,1) [circle, fill, inner sep = 1.5pt] {};

  \draw (5) -- (6);
  \draw (7) -- (6);
  \draw (7) -- (8);
  \draw (7) -- (9); 
    \end{tikzpicture}
\end{center}
\begin{lemma} \label{71}
    If $X_{i,j}, X_{j,k} \in \mathfrak{T}'$ for pairwise distinct $i,j,k$, then $X_{i,k} \in \mathfrak{T}'$.
\end{lemma}
\begin{proof}
    Consider the two following trees $T_1, T_2$ with a common subtree on $\{b,c\}$:
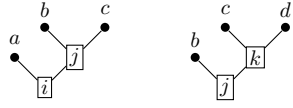
    \begin{center}
\begin{tikzpicture}[scale = 0.8, transform shape]
\node (5) at (1,0.5) [circle,fill, inner sep=1.5pt] {};
\node (6) at (1.5,0) [rectangle, draw, inner sep=1.5pt] {$i$};
  \node (7) at (2, 0.5) [rectangle, draw, inner sep = 1.5pt] {$j$};
  \node (8) at (1.5,1) [circle, fill, inner sep=1.5pt] {};
  \node (9) at (2.5,1) [circle, fill, inner sep = 1.5pt] {};

  \draw (5) -- (6);
  \draw (7) -- (6);
  \draw (7) -- (8);
  \draw (7) -- (9); 

  \node at (1,0.8) {$a$};
  \node at (1.5,1.3) {$b$};
  \node at (2.5,1.3) {$c$};

  \node (5) at (4,0.5) [circle,fill, inner sep=1.5pt] {};

  \node (6) at (4.5,0) [rectangle, draw, inner sep=1.5pt] {$j$};
  \node (7) at (5, 0.5) [rectangle, draw, inner sep = 1.5pt] {$k$};
  \node (8) at (4.5,1) [circle, fill, inner sep=1.5pt] {};
  \node (9) at (5.5,1) [circle, fill, inner sep = 1.5pt] {};

  \draw (5) -- (6);
  \draw (7) -- (6);
  \draw (7) -- (8);
  \draw (7) -- (9); 

  \node at (4,0.8) {$b$};
  \node at (4.5,1.3) {$c$};
  \node at (5.5,1.3) {$d$};
    \end{tikzpicture}
    \end{center}
The only amalgamation of $T_1$ and $T_2$ over $\{b,c\}$ is the following tree $T$ 
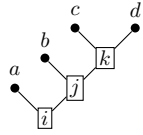
\begin{center}
\begin{tikzpicture}[scale = 0.8, transform shape]
\node (5) at (1,0.5) [circle,fill, inner sep=1.5pt] {};
\node (6) at (1.5,0) [rectangle, draw, inner sep=1.5pt] {$i$};
  \node (7) at (2, 0.5) [rectangle, draw, inner sep = 1.5pt] {$j$};
  \node (8) at (1.5,1) [circle, fill, inner sep=1.5pt] {};
  \node (9) at (2.5,1) [rectangle, draw, inner sep = 1.5pt] {$k$};
  \node (10) at (2,1.5) [circle, fill, inner sep=1.5pt] {};
  \node (11) at (3,1.5) [circle,fill,inner sep = 1.5pt] {};

  \draw (5) -- (6);
  \draw (7) -- (6);
  \draw (7) -- (8);
  \draw (7) -- (9); 
  \draw (9) -- (10);
  \draw (9) -- (11);

  \node at (1,0.8) {$a$};
  \node at (1.5,1.3) {$b$};
  \node at (2,1.8) {$c$};
  \node at (3,1.8) {$d$};
  \end{tikzpicture}
  \end{center}
  which must belong to $\mathfrak{T}'$ by condition \ref{fd}. The induced subtree on $\{a,c,d\}$ in $T$ is isomorphic to $X_{i,k}$, which by condition \ref{fb} implies that $X_{i,k} \in \mathfrak{T}'$. 
\end{proof}
We say that a color $k \in [n]$ \emph{appears} in a subclass $\mathfrak{T}'$ if there exists a tree with node(s) colored $k$ in $\mathfrak{T}'$. Alternatively, $\mathfrak{T}'$ has color $k$ if $k$ appears.
\begin{lemma} \label{72}
    If $i$ and $j$ both appear in $\mathfrak{T}'$, then at least one of $X_{i,j}$ and $X_{j,i}$ is in $\mathfrak{T}'$. 
\end{lemma}
\begin{proof}
    By \ref{fb}, the two following trees belong to $\mathfrak{T}'$: 
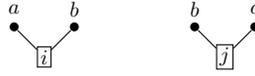
\begin{center}
\begin{tikzpicture}[scale = 0.8, transform shape]
\node (5) at (1,0.5) [circle,fill, inner sep=1.5pt] {};
\node (6) at (1.5,0) [rectangle, draw, inner sep=1.5pt] {$i$};
  \node (7) at (2, 0.5) [circle, fill, inner sep=1.5pt] {};

  \draw (5) -- (6);
  \draw (7) -- (6);

  \node at (1,0.8) {$a$};
  \node at (2,0.8) {$b$};

  \node (5) at (4,0.5) [circle,fill, inner sep=1.5pt] {};
  \node (6) at (4.5,0) [rectangle, draw, inner sep=1.5pt] {$j$};
  \node (7) at (5, 0.5) [circle, fill, inner sep=1.5pt] {};

  \draw (5) -- (6);
  \draw (7) -- (6);

  \node at (4,0.8) {$b$};
  \node at (5,0.8) {$c$};
    \end{tikzpicture}
    \end{center}
The only two amalgamations over the one-point tree on $b$ are
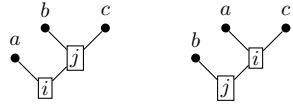
\begin{center}
\begin{tikzpicture}[scale = 0.8, transform shape]
\node (5) at (1,0.5) [circle,fill, inner sep=1.5pt] {};
\node (6) at (1.5,0) [rectangle, draw, inner sep=1.5pt] {$i$};
  \node (7) at (2, 0.5) [rectangle, draw, inner sep=1.5pt] {$j$};
\node (8) at (1.5,1) [circle, fill, inner sep = 1.5pt] {};
\node (9) at (2.5,1) [circle, fill, inner sep = 1.5pt] {};
  \draw (5) -- (6);
  \draw (7) -- (6);
  \draw (7) -- (8);
  \draw (7) -- (9);

  \node at (1.5,1.3) {$b$};
  \node at (2.5,1.3) {$c$};
  \node at (1,0.8) {$a$};

  \node (5) at (4,0.5) [circle,fill, inner sep=1.5pt] {};
\node (6) at (4.5,0) [rectangle, draw, inner sep=1.5pt] {$j$};
  \node (7) at (5, 0.5) [rectangle, draw, inner sep=1.5pt] {$i$};
\node (8) at (4.5,1) [circle, fill, inner sep = 1.5pt] {};
\node (9) at (5.5,1) [circle, fill, inner sep = 1.5pt] {};
  \draw (5) -- (6);
  \draw (7) -- (6);
  \draw (7) -- (8);
  \draw (7) -- (9);

  \node at (4.5,1.3) {$a$};
  \node at (5.5,1.3) {$c$};
  \node at (4,0.8) {$b$};
    \end{tikzpicture}
    \end{center}
at least one of which is in $\mathfrak{T}'$ by \ref{fd}.
\end{proof}
\begin{lemma} \label{73}
    If $X_{i,j}, X_{j,i} \in \mathfrak{T}'$ for distinct $i,j$, then $X_{i,i}, X_{j,j} \in \mathfrak{T}'$.
\end{lemma}
\begin{proof}
    Consider the two trees $T_1$ and $T_2$
    
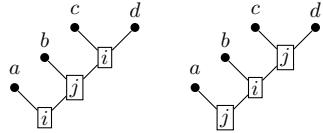
\begin{center}
\begin{tikzpicture}[scale = 0.8, transform shape]
\node (5) at (1,0.5) [circle,fill, inner sep=1.5pt] {};
\node (6) at (1.5,0) [rectangle, draw, inner sep=1.5pt] {$i$};
  \node (7) at (2, 0.5) [rectangle, draw, inner sep=1.5pt] {$j$};
\node (8) at (1.5,1) [circle, fill, inner sep = 1.5pt] {};
\node (9) at (2.5,1) [rectangle, draw, inner sep = 1.5pt] {$i$};
\node (10) at (2,1.5) [circle, fill, inner sep = 1.5pt] {};
\node (11) at (3,1.5) [circle, fill, inner sep = 1.5pt] {};
  \draw (5) -- (6);
  \draw (7) -- (6);
  \draw (7) -- (8);
  \draw (7) -- (9);
  \draw (9) -- (10);
  \draw (9) -- (11);

  \node at (1.5,1.3) {$b$};
  \node at (2,1.8) {$c$};
  \node at (3,1.8) {$d$};
  \node at (1,0.8) {$a$};

  \node (5) at (4,0.5) [circle,fill, inner sep=1.5pt] {};
\node (6) at (4.5,0) [rectangle, draw, inner sep=1.5pt] {$j$};
  \node (7) at (5, 0.5) [rectangle, draw, inner sep=1.5pt] {$i$};
\node (8) at (4.5,1) [circle, fill, inner sep = 1.5pt] {};
\node (9) at (5.5,1) [rectangle, draw, inner sep = 1.5pt] {$j$};
\node (10) at (5,1.5) [circle, fill, inner sep = 1.5pt] {};
\node (11) at (6,1.5) [circle, fill, inner sep = 1.5pt] {};
  \draw (5) -- (6);
  \draw (7) -- (6);
  \draw (7) -- (8);
  \draw (7) -- (9);
  \draw (9) -- (10);
  \draw (9) -- (11);

  \node at (4.5,1.3) {$b$};
  \node at (5,1.8) {$c$};
  \node at (6,1.8) {$d$};
  \node at (4,0.8) {$a$};
    \end{tikzpicture}
    \end{center}
By Proposition \ref{easy}, $T_1$ is the unique amalgamation of $T_1 \setminus a \cong X_{j,i}$ and $T_1 \setminus c \cong X_{i,j}$ over $T_1 \setminus \{a,c\}$, so we must have $T_1 \in \mathfrak{T}'$ by condition \ref{fd}. Similarly, $T_2 \in \mathfrak{T}'$. Finally, notice that the induced subtrees on $a,c,d$ in $T_1$ and $T_2$ are $X_{i,i}$ and $X_{j,j}$ respectively, implying that $X_{i,i}, X_{j,j} \in \mathfrak{T}'$ by \ref{fb}. 
\end{proof}
Let $X_n = \{X_{i,j} : i,j \in [n]\}$ be the set of trees with three leaves in $\partial \mathfrak{T}_3(n)$, and $A$ be a subset of $X_n$. Let $\mathfrak{T}'_A$ denote the corresponding subclass of $\partial \mathfrak{T}_3(n)$ of all structures whose (only) three-element substructures are those in $A$. Clearly, if $A$ is nonempty, then $\mathfrak{T}'_A$ is unique. If $A$ is empty, then $\mathfrak{T}'_A$ is empty or contains structures with at most one or two leaves; we denote the latter classes as $\mathfrak{T}_{\le 1}$ and $\mathfrak{T}_{\le 2}$.
\begin{prop} \label{74}
    The subclass $\mathfrak{T}'_A \subseteq \partial \mathfrak{T}_3(n)$ is Fraïssé if and only if it satisfies the conditions in Lemmas \ref{71}, \ref{72}, \ref{73}.  
\end{prop}
\begin{proof}
    The converse is clear, so we show the forward direction. Conditions \ref{fa} and \ref{fb} are satisfied by construction, while \ref{fc} is inherited from $\partial \mathfrak{T}_3(n)$. It thus suffices to check \ref{fd}. Let $T_1,T_2$ be tree structures in $\mathfrak{T}'_A$ with common substructure $T$. We proceed by induction on $|T_2 \setminus T|$. 
    
    If $|T_2 \setminus T| = 0$, then $T_1$ is the desired amalgamation. Otherwise, by the inductive hypothesis, there exists an amalgamation $T' \in \mathfrak{T}'_A$ of $T_1$ and $T_2 \setminus a$ for any $a \in T_2 \setminus T$. By the amalgamation property of $\partial \mathfrak{T}_3(n)$, there is an amalgamation $T'' \in \partial \mathfrak{T}_3(n)$ such that the restriction of $T''$ to $T_1 \cup T_2 \setminus a$ is isomorphic to $T'$. Now, let $N(a)$ denote the node adjacent to $a$ and $C(a)$ denote its color. For $b \in T_1$, if there is $c \in T_2$ such that $X_{C(a),C(b)}, X_{C(b),C(c)} \in A$ (resp. $X_{C(c),C(b)}, X_{C(b),C(a)} \in A$), then by Lemma \ref{71}, we have $X_{C(a),C(c)} \in A$ (resp. $X_{C(c),C(a)} \in A$), implying that no new three-element structures are introduced in $A$ by $T''$. Otherwise, let $B=\{b_1, \dots, b_m\}$ be the set of $b_i \in T_1$ such that such $c$ does not exist and either the simple path from the root to $b_i$ contains $a$ or the path from the root to $a$ contains $b_i$, since only such $b_i$ may potentially introduce new three-element substructures in $T''$. Without loss of generality, suppose that $X_{C(b_i)},X_{C(b_{i+1})} \in A$ for $i = 1, \dots, n-1$. The induced subgraph on $B \cup \{a\}$ is the following branch:
    
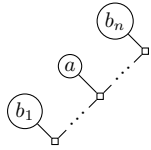
\begin{center}
        \begin{tikzpicture}[scale = 0.8, transform shape]
            \node (0) at (0,0) [rectangle, draw, inner sep=1.5pt] {};
            \node (1) at (-0.5,0.5) [circle, draw, inner sep = 1.5pt] {$b_1$};
            \draw (0)--(1);
            \node (d1) at (9/32, 9/32) {$\cdot$};
            \node (d2) at (6/16,6/16) {$\cdot$};
            \node (d3) at (15/32,15/32) {$\cdot$};
            \node (2) at (3/4,3/4) [rectangle, draw, inner sep = 1.5pt] {};
            \node (3) at (1/4,5/4) [circle, draw, inner sep = 1.5pt] {$a$};
            \draw (2) -- (3);
            \draw (0)--(3/16,3/16);
            \draw (18/32,18/32) -- (2);
            \node (d4) at (33/32,33/32) {$\cdot$};
            \node (d5) at (18/16,18/16) {$\cdot$};
            \node (d6) at (39/32,39/32) {$\cdot$};
            \node (4) at (3/2,3/2) [rectangle, draw, inner sep = 1.5pt] {};
            \node (5) at (1, 2) [circle, draw, inner sep = 1.5pt] {$b_n$};
            \draw (4) -- (5);
            \draw (2) -- (30/32,30/32);
            \draw (42/32,42/32) -- (4);
        \end{tikzpicture}
    \end{center}
Let $j$ be the smallest index such that $X_{C(a),C(b_j)} \in A$. Reposition $a$ along the branch such that $N(a)$ is in between $N(b_{j-1})$ and $N(b_j)$ where $a$ becomes the root if $j=1$. Then, for all $i \geq j$, we have $X_{C(a),C(b_i)} \in A$ by Lemma \ref{71}, so we obtain a desired modification of $T''$ that now belongs to $\mathfrak{T}'_A$. 
\end{proof}
Let $S(X_n)$ denote the set of subsets $A \subseteq X_n$ such that $\mathfrak{T}'_A$ satisfies Lemmas \ref{71}, \ref{72}, \ref{73}. Let $\mathcal{F}_n$ denote the set of Fraïssé subclasses of $\partial \mathfrak{T}_3(n)$ excluding $\mathfrak{T}_{\leq 1}$ and $\mathfrak{T}_{\leq 2}$. The above construction establishes a map $\Phi: S(X_n) \to \mathcal{F}_n$ given by $\Phi(A) = \mathfrak{T}'_A$ (where $\Phi(\varnothing) = \varnothing$). 
\begin{prop} \label{75}
The map $\Phi$ is a surjection.
\end{prop}
\begin{proof}
    Let $\mathfrak{T}'$ be a Fraïssé subclass of $\partial \mathfrak{T}_3(n)$ such that its set of three-element structures is $A$. By the definition of $\mathfrak{T}'_A$, we have $\mathfrak{T}' \subseteq \mathfrak{T}'_A$. On the other hand, let $T$ be a tree structure in $\mathfrak{T}'_A$. We show by induction on $|T|$ that $T \in \mathfrak{T}'$. This is clear for $|T| \leq 3$ by taking induced substructures in $\mathfrak{T}'$, so suppose $|T| \geq 4$. It is easy to see that there must exist leaves $a$ and $b$ such that $a$ and $b$ are separated, i.e., $N(a)$ and $N(b)$ are neither the same nor adjacent. By Proposition \ref{easy}, $T$ is the unique amalgamation of $T \setminus a$ and $T \setminus b$, yet $|T \setminus a| = |T \setminus b| = |T|-1$, so by the inductive hypothesis, $T \setminus a, T \setminus b \in \mathfrak{T}'$, implying that $T \in \mathfrak{T}'$ by \ref{fd}. Thus, $\mathfrak{T}'_A \subseteq \mathfrak{T}'$, so we have $\mathfrak{T}' = \mathfrak{T}'_A$. This means that given any Fra\"iss\'e class in $\mathcal{F}_n$, its preimage under $\Phi$ exists and is determined by its three-element structures.
\end{proof}
\begin{prop} \label{76}
   The map $\Phi$ is injective. 
\end{prop}
\begin{proof}
    This follows from the definition of $\mathfrak{T}'_A$.
\end{proof}
\begin{corollary} \label{class}
    The map $\Phi$ is bijective.
\end{corollary}
As a result, each Fraïssé subclass of $\partial \mathfrak{T}_3(n)$ is determined by its three-element structures, i.e., a subset of $X_n$ satisfying Lemmas \ref{71}, \ref{72}, and \ref{73}. We introduce an additional result to enumerate these subclasses, which is the following.  
\begin{prop} \label{enumerate}
    For each $n \geq 2$, let $f(n)$ denote the number of Fraïssé subclasses of $\partial \mathfrak{T}_3(n)$ with $n$ colors. Then, we have \begin{equation} \label{7.2} f(n) = n f(n-1) + \sum_{m=0}^{n-1} \binom{n}{m} f(m) \end{equation} where $f(0)=1$, $f(1)=2$.  
\end{prop}
\begin{proof}
    Since the empty class is Fraïssé, we have $f(0)=1$. Let $$\varepsilon(i,j) := \mathbf{1}_{X_{i,j} \in \mathfrak{T}'}$$ for some Fra\"iss\'e subclass $\mathfrak{T}' \subseteq \partial \mathfrak{T}_3(n)$. For $n \geq 1$, by Propositions \ref{75} and \ref{76}, $f(n)$ is equal to the number of maps $\varepsilon : [n] \times [n] \to \{0,1\}$ satisfying the following properties for all pairwise distinct $i,j,k \in [n]$:
    \begin{enumerate}
        \item If $\varepsilon(i,j) = \varepsilon(j,k) = 1$ then $\varepsilon(i,k) = 1$.
        \item If $\varepsilon(i,j) = \varepsilon(j,i) = 1$, then $\varepsilon(i,i) = \varepsilon(j,j)=1$.
        \item We have ($\varepsilon(i,j),\varepsilon(j,i)) \neq (0,0)$.
    \end{enumerate}
Property (3) guarantees that all $n$ colors appear in the counted subclass. We say that a nonempty subset $U$ of $[n]$ is \emph{self-connected} if $\varepsilon(i,j) = \varepsilon(j,i) = 1$ for all $i,j \in U$ and \emph{isolated} if $\varepsilon(i,k) = 0$ for all $k \in [n] \setminus U$. We claim that there exists a self-connected and isolated set $U$ for each $\varepsilon$. Note that $U$ is unique if it exists, and Lemma \ref{73} ensures that the case $i=j$ is consistent with the properties of $U$. 

To prove the claim, we proceed by induction on $n$. The result is clear for $n \leq 2$, so suppose $n \geq 3$. Fix a map $\varepsilon$, and consider any self-connected $W \subseteq [n]$. If $\varepsilon(i,k) = 0$ for some $i \in W$ and all $k \in [n] \setminus W$, then $W$ is our desired choice, so suppose there is $k$ such that $\varepsilon(i,k)=1$. We also suppose $W$ is \emph{maximally self-connected}, meaning that $W$ is not properly contained in any larger self-connected set. This implies that $\varepsilon(k,i) = 0$, since otherwise Lemma \ref{71} implies that $W \cup \{k\}$ is self-connected. In fact, by Lemma \ref{71}, we conclude that $(\varepsilon(w,k), \varepsilon(k,w)) = (1,0)$ for all $w \in W$. Let $K \subseteq [n] \setminus W$ denote the set of such $k$. Observe that for all $l \in [n] \setminus (W \cup K)$, we have $(\varepsilon(w,l), \varepsilon(l,w)) = (0,1)$, so it follows that $\varepsilon(l,k)=1$, while $\varepsilon(k,l)=0$ by the maximal self-connectedness of $W$. Now, by the inductive hypothesis, there is a self-connected, isolated set $U \subseteq [n] \setminus W$. It is easy to see that $U \subseteq K$, meaning that $\varepsilon(u,w) = 0$ for all $u \in U$ and $w \in W$, which gives the claim.

The right-hand side of Equation \ref{7.2} is obtained by doing casework on $|U|=n-m$. Note that $|U| > 0$ since all $n$ colors appear in the subclass. If $|U|=1$, suppose $U = \{u_1\}$. We have $\varepsilon(u_1,u_1) \in \{0,1\}$, while there are $f(n-1)$ maps $\varepsilon : [n] \setminus \{u_1\} \times [n] \setminus \{u_1\} \to \{0,1\}$. There are $n$ choices for $u_1$, contributing $$2nf(n-1) = nf(n-1)+\binom{n}{n-1}f(n-1)$$ possible maps. 

If $2 \leq |U| \leq n$, then there are $f(n-|U|)=f(m)$ maps $\varepsilon : [n] \setminus U \times [n] \setminus U \to \{0,1\}$ while Lemma \ref{73} implies that $\varepsilon(u,u)=1$ for all $u \in U$. Since there are $$\binom{n}{|U|} = \binom{n}{n-m} = \binom{n}{m}$$ choices for $U$, we have $$\sum_{m=0}^{n-2} \binom{n}{m}f(m)$$ additional maps, and the proposition follows. 
\end{proof}
\begin{corollary}
    The number of Fraïssé subclasses of $\partial \mathfrak{T}_3(n)$ for $n=1,2,3,4, \dots$ with $n$ colors is $2,9,61,551, \dots$
\end{corollary}
\begin{remark}
    The sequence $(f(n))_{n \geq 0}$ appears in OEIS \cite[A006155]{oeis}. Its exponential generating function is $$F(x) = \frac{1}{2-x-e^x}.$$ 
\end{remark}
\subsection{Induced regular measures} \label{induced}
Given a fixed measure $\mu$ on a Fraïssé class $\mathfrak{F}$, let $\mathfrak{F}'$ be the subclass consisting of all structures $X$ such that $\mu(X) \neq 0$ (where $\mu(X) = \mu(\varnothing \to X)$) and $\mu'$ be the restriction of $\mu$ to $\mathfrak{F}'$. We say that $\mathfrak{F}'$ is the \emph{support} of $\mu$ and $\mu'$ is the \emph{induced regular measure} on $\mathfrak{F}'$. A subclass $\mathfrak{F}' \subseteq \mathfrak{F}$ that is the support of some measure on $\mathfrak{F}$ is called an \emph{induced subclass}. The following proposition justifies the regularity of $\mu'$; we include the proof for familiarity with induced regular measures.
\begin{prop}[{\cite[Proposition 2.18, p.12]{harman2024arborealtensorcategories}}]
    If $\mathfrak{F}'$ is an induced subclass of $\mathfrak{F}$, then it is a Fra\"iss\'e class, and $\mu'$ is a regular measure on $\mathfrak{F}'$.
\end{prop}
\begin{proof}
    The argument follows that of \cite{harman2024arborealtensorcategories}, which we include for completeness. Suppose that $X \in \mathfrak{F}'$ and there is an embedding $i: Y \to X$ (which could be an isomorphism). We have $\mu(X)=\mu(i)\mu(Y)$, so $\mu(X) \neq 0$ implies that $\mu(Y) \neq 0$, or $Y \in \mathfrak{F}'$. Hence $\mathfrak{F}'$ satisfies \ref{fa} and \ref{fb}. 

    Clearly, $\mathfrak{F}'$ satisfies \ref{fc} since it is a subclass of $\mathfrak{F}$, so it suffices to show \ref{fd}. Given structures $X, Y \in \mathfrak{F}'$ with a common substructure $Z$, suppose $Z'_1, \dots, Z'_m$ are all amalgamations of $X$ and $Y$ over $Z$. We have \begin{align} \label{trivial} \mu(Z \to X) = \sum_{k=1}^m \mu(Y \to Z'_i). \end{align} Since the left hand side is nonzero, there must exist some $i$ such that $\mu(Y \to Z'_i)$ is nonzero, so $Z'_i \in \mathfrak{F}'$, as desired.

    Now we show that $\mu'$ defines a measure on $\mathfrak{F}'$. Properties \ref{ma} and \ref{mb} are obvious. Observe that $\mathfrak{F}'$ simply excludes the structures $Z_k'$ where $\mu(Z_k') = \mu(Y \to Z_k') = 0$, which are just the terms on the right-hand side of Equation \ref{trivial} that vanish, so the equation still holds. Thus \ref{mc} is also satisfied.

    Finally, given an embedding $i: Y \to X$ where $X,Y \in \mathfrak{F}'$, we have $\mu'(i) = \mu'(X)\mu'(Y)^{-1}$, which is nonzero, so $\mu'$ is indeed regular on $\mathfrak{F}'$.
\end{proof}
From Theorem \ref{maintheorem}, let $\mu$ be a $\mathbb{Z}\left[\frac 12\right]$-valued measure, i.e., $\mu: \Theta(\partial \mathfrak{T}_3(n)) \to \ZZ \left[\frac 12 \right]$ is a ring homomorphism realizing one of the $(2n+2)^n$ measures on $\partial \mathfrak{T}_3(n)$. Suppose that $\mathfrak{T}' \subseteq \partial \mathfrak{T}_3(n)$ is an induced subclass with induced regular measure $\mu'$. In particular, by Propositions \ref{75} and \ref{76}, $\mathfrak{T}'$ is determined by its three-element structures. We narrow down the list of possible three-element structures of supports with the following result.
\begin{lemma} \label{reduce}
    For all distinct $i,j \in [n]$, at most one of $X_{i,j}$ and $X_{j,i}$ is in $\mathfrak{T}'$.
\end{lemma}
\begin{proof}
    Since $\mu'$ is regular on $\mathfrak{T}'$, if both $X_{i,j}$ and $X_{j,i}$ belong to $\mathfrak{T}'$, then $\mu'(C(i,j))$ and $\mu'(C(j,i))$ are both nonzero. As a result, $\mu'(C(i,j)C(j,i))$ is nonzero, contradicting \ref{systemc}. 
\end{proof}
Let $B$ be the two-leaf tree structure with node colored $k$. Recall that $A = [\varnothing \to \bullet]$ and $B(k) = [\bullet \to B]$ are isomorphism classes of embeddings in $\Theta(\partial \mathfrak{T}_3(n))$.
\begin{lemma} \label{anonzero}
    For each $n \geq 1$, we have $\mu(A) \neq 0$ for all measures $\mu : \Theta(\partial \mathfrak{T}_3(n)) \to \CC$. 
\end{lemma}
\begin{proof}
    Again let $\mathbf{X}$ denote $\mu(X)$. The proof idea is the same as that in Propositions \ref{nonzero} and \ref{biject} by showing that the coefficient matrix of the $\mathbf{S}(i)$ is the adjacency matrix of a DAG with nonzero diagonal terms (so the DAG has self-loops adjoined). Indeed, suppose for the sake of contradiction that $$\mathbf{A} = 1+\sum_{i \in [n]} \mathbf{B}(i) = 0.$$ It follows that \begin{align*} 2\mathbf{S}(i)&=\mathbf{B}(i)+3\mathbf{C}(i,i)+1 +\sum_{p \in [n] \setminus i} \mathbf{B}(p) - \sum_{p \in [n] \setminus i} \mathbf{C}(p,i) \\ &= 3\mathbf{C}(i,i)-\sum_{p \in [n] \setminus i} \mathbf{C}(p,i). \end{align*} Let $\mathbf{M} = (m_{ij})_{i,j=1}^n \in \RR^{(n-1) \times (n-1)}$ and again $\mathbf{s}=(\mathbf{S}(i))_{i=1}^n \in \RR^n$ where \begin{align*}
m_{ij} &=
\begin{cases}
    2-3\delta_{\mathbf{C}(i,i),\mathbf{S}(i)} & \text{if } i=j, \\[6pt]
   \delta_{\mathbf{C}(j,i),\mathbf{S}(j)}     & \text{if } i\neq j.
\end{cases}
\end{align*} We thus obtain the matrix equation $\mathbf{M}\mathbf{s}=\mathbf{0}$. Since $\mathbf{C}(i,j)\mathbf{C}(j,i)=0$ for all distinct $i,j$, $G_{\mathbf{M}}$, the induced graph from $\mathbf{M}$, is directed. To see that $G_{\mathbf{M}}$ has no cycles of length at least two, note that if there was a cycle $a_1 \to \cdots \to a_l \to a_1$ for $l \geq 2$, then $(\mathbf{C}(a_i,a_{i+1}),(\mathbf{C}(a_{i+1},a_i))=(0,\mathbf{S}(a_{i+1}))$ for $i = 1, \dots, l$ where indices are taken mod $l$. Applying \ref{systemf} with $(i,j,k)=(a_1,a_{i+1},a_{i+2})$ for each $i$, we obtain $\mathbf{C}(a_1,a_{i+1})=0$ for all $i$, so $\mathbf{C}(a_1,a_l)=\mathbf{S}(a_1)=0$, a contradiction. Hence $\mathbf{M}$ is the adjacency matrix of a DAG with self-loops adjoined. By Lemma \ref{perm}, there is a permutation matrix $\mathbf{P}$ such that $\mathbf{PMP^{-1}}$ is upper triangular. Consequently, $\det(\mathbf{M}) = \prod_{i \in [n]} m_{ii} = \pm 2^k \neq 0$ for some nonnegative integer $k$, yet $\mathbf{s} \neq \mathbf{0}$ by Proposition \ref{nonzero}, so the matrix equation admits no solutions and the conclusion follows.
\end{proof}
\begin{remark}
    The above lemma ensures that induced subclasses are not automatically the empty Fraïssé class, since if $\mu(A)$ vanishes, then $\mu(X)$ vanishes for all $X$ where $|X| \geq 1$.
\end{remark}
\begin{lemma} \label{bk}
    A color $k$ appears in $\mathfrak{T}'$ if and only if $\mu(B(k)) \neq 0$.
\end{lemma}
\begin{proof}
    Recall that $B$ denotes the two-leaf tree structure with node colored $k$. Notice that $k$ appears in $\mathfrak{T}'$ if and only if $B \in \mathfrak{T}'$, or $\mu(B) \neq 0$. Since $\mu(B) = \mu(\varnothing \to B) = \mu(A)\mu(B(k))$, and $\mu(A)$ is nonzero by Lemma \ref{anonzero}, $\mu(B)$ is nonzero if and only if $\mu(B(k))$ is nonzero, which implies the desired result. 
\end{proof}
Observe that if $\mathfrak{T}'$ is a subclass of $\partial \mathfrak{T}_3(n)$, then it is also a subclass of $\partial \mathfrak{T}_3(m)$ for any $m \geq n$. If the colors $1, \dots, n$ all appear in $\mathfrak{T}'$, then $\partial \mathfrak{T}_3(n)$ can be viewed as the minimal Fraïssé class of node-colored rooted binary trees that contains $\mathfrak{T}'$. The following proposition implies that given an induced (Fra\"iss\'e) subclass $\mathfrak{T}'$ of $\partial \mathfrak{T}_3(n)$ for minimal $n$, the collection of all induced regular measures restricted from $\partial \mathfrak{T}_3(n)$ accounts for all induced regular measures on $\mathfrak{T}'$. In other words, for each $n$, it suffices to consider induced subclasses of $\partial \mathfrak{T}_3(n)$ with $n$ colors when computing induced regular measures.
\begin{prop}
   Let $\mathfrak{T}'$ be a Fra\"iss\'e subclass of $\partial \mathfrak{T}_3(n)$ with $n$ colors and $\mu'$ be an induced regular measure on $\mathfrak{T}'$. If $\mathfrak{T}'$ is the support of $\nu$ where $\nu$ is a measure on $\partial \mathfrak{T}_3(m)$ for any $m \geq n$, then $ \nu|_{\mathfrak{T}'}=\mu'$.   
\end{prop}
\begin{proof}
    Recall that computing a measure is equivalent to solving the linear system encoded by $\mathbf{D}$ in Equation \ref{ds} with appropriate constraints on the entries of $\mathbf{D}$. Since $\mathfrak{T}'$ is the support of $\nu$, the colors $n+1, n+2, \dots, m$ do not appear in $\mathfrak{T}'$. It follows from Lemma \ref{bk} that for all $j \in [m] \setminus [n]$, we have $\nu(B(j))=0$. By \ref{systeme}, we obtain $\nu(B(i)C(j,i))=0$ for all $i \in [n]$, yet $i$ appears in $\mathfrak{T}'$, so $\nu(B(i)) \neq 0$, which means $\nu(C(j,i))=0$ for all $i \in [n]$ and $j \in [m]\setminus[n]$. It follows that $d_{ij}=0$ for all such $i,j$, so the coefficient matrix $\mathbf{D} \in \RR^{m \times m}$ corresponding to $\nu$ is \begin{align*}
         \mathbf{D} &= \begin{pmatrix}
d_{11}  & \dotsc  & d_{1n} & 0 & \dotsc & 0 \\
d_{21} & \dotsc  & d_{2n} & 0 & \dotsc & 0 \\
\vdots  & \ddots  & \vdots  & \vdots & \ddots & \vdots \\
d_{n1} & \dotsc  & d_{nn} & 0 & \dotsc & 0 \\
d_{(n+1)1} & \dotsc & d_{(n+1)n} & d_{(n+1)(n+1)} & \dotsc & d_{(n+1)m} \\ \vdots  & \ddots  & \vdots  & \vdots & \ddots & \vdots \\ d_{m1} & \dotsc & d_{mn} & d_{m(n+1)} & \dotsc & d_{mm}
\end{pmatrix}. 
    \end{align*} Notice that the leading principal matrix of order $n$ of $\mathbf{D}$ encodes precisely the constraints on $\mu'$, while the upper-right off diagonal block is a zero submatrix, so $\nu$ is subject to the exact same constraints on the colors $1, \dots, n$ as $\mu'$ (with potentially more constraints on colors $n+1,\dots, m$), implying that $\nu|_{\mathfrak{T}'} = \mu'$, as desired. 
\end{proof}
By the above proposition, it suffices to consider Fraïssé subclasses of $\partial \mathfrak{T}_3(n)$ where all $n$ colors appear. For each $n \geq 1$, let $\mathcal{F}'_n$ denote the set of Fraïssé subclasses of $\partial \mathfrak{T}_3(n)$ with $n$ colors such that exactly one of $X_{i,j}$ and $X_{j,i}$ is in $\mathfrak{T}'$ for all $\mathfrak{T}' \in \mathcal{F}'_n$.

\begin{prop} \label{714}
    A subclass $\mathfrak{T}'$ with $n$ colors is an induced subclass of $\partial \mathfrak{T}_3(n)$ if and only if $\mathfrak{T}' \in \mathcal{F}'_n$.
\end{prop}
\begin{proof}
    Since exactly one of $X_{i,j}$ and $X_{j,i}$ is in $\mathfrak{T}'$, there is an induced total order on $[n]$ where $i \prec j$ if and only if $X_{i,j} \in \mathfrak{T}'$. Without loss of generality, suppose that $X_{i,j} \in \mathfrak{T}'$ if and only if $i < j$. For all $i,j$, let $$(\mathbf{B}(i),\mathbf{C}(i,j),\mathbf{C}(j,i))= (\mathbf{S}(i),\mathbf{S}(i),0)$$ and $$\mathbf{C}(i,i) = \begin{cases} \mathbf{S}(i) & \text{if } X_{i,i} \in \mathfrak{T}', \\ 0 & \text{otherwise}. \end{cases}$$ The above construction satisfies $\ref{system}$ and defines the unique induced regular measure on $\mathfrak{T}'$ (since $\mathbf{S}(i) \neq 0$) as desired.  

    On the other hand, if $\mathfrak{T}' \not\in \mathcal{F}_n'$, then by Lemma \ref{72}, there exist $i,j$ such that $X_{i,j}$ and $X_{j,i}$ are both in $\mathfrak{T}'$, which by Lemma \ref{reduce} means that $\mathfrak{T}'$ is not the support of any measure.
\end{proof}

\begin{corollary}
    For each $n \geq 1$, there are $2^nn!$ induced subclasses of $\partial \mathfrak{T}_3(n)$ with $n$ colors. 
\end{corollary}
\begin{proof}
    By Proposition \ref{714}, it suffices to count the number of classes in $\mathcal{F}'_n$. There are $n!$ total orderings of the $n$ colors, and for each $i \in [n]$, we can either have $X_{i,i} \in \mathfrak{T}'$ or $X_{i,i} \not \in \mathfrak{T}'$, giving $2^n$ independent choices. Consequently, we obtain $2^nn!$ distinct classes in $\mathcal{F}'_n$, and $2^n$ classes up to color-swapping automorphisms, each admitting an induced regular measure.   
\end{proof}
We now determine a formula for the induced regular measure on each of the induced (Fra\"iss\'e) subclasses with $n$ colors. Without loss of generality, suppose that each subclass is \emph{naturally ordered}; that is, $X_{i,j}$ belongs to the subclass if and only if $i < j$. From now on, let each naturally ordered induced subclass with $n$ colors be denoted as $\partial \mathfrak{T}_3(n)^{\mathrm{ord}}_I$ for some $I \subseteq [n]$ where $i \in I$ if and only if $X_{i,i} \in \partial \mathfrak{T}_3(n)^{\mathrm{ord}}_I$. Note that each $I$ uniquely corresponds to a naturally ordered induced subclass of $\partial \mathfrak{T}_3(n)$ with $n$ colors by Corollary \ref{class}. Let $\mu_n^I$ denote the corresponding induced regular measure. For a tree $T$, let $\ell(T)$ denote the number of leaves and $v_T(i)$ denote the number of nodes of color $i$.
\begin{prop} \label{formula}
For an induced subclass $\partial \mathfrak{T}_3(n)^{\mathrm{ord}}_I$ where $I \subseteq [n]$ and a tree $T \in \partial \mathfrak{T}_3(n)_I^{\mathrm{ord}}$, we have $$\mu_n^I(T)=
\begin{cases}
1 & \text{if }\ell(T)=0,\\[6pt]
\displaystyle\prod_{i=1}^n \Bigl(P_I(i)^{v_T(i)}\,(1+P_I(i))^{\,1+\sum_{j=1}^{i-1} v_T(j)}\Bigr)
& \text{if }\ell(T)\ge 1,
\end{cases}$$where $$P_I(i) = \frac{1}{1-3\cdot \mathbf{1}_I(i)}.$$
\end{prop}
\begin{proof}
Fix a subset $I$, and, for convenience, let $\mu_n^I = \mu$ and $P_I = P$. Recall the coefficient matrix $\mathbf{D}$ from Proposition \ref{biject}, where $d_{ii} = 2-3\delta_{\mathbf{C}(i,i),\mathbf{S}(i)}-\delta_{\mathbf{B}(i),\mathbf{S}(i)}$ and $d_{ij} = \delta_{\mathbf{C}(j,i),\mathbf{S}(j)}-\delta_{\mathbf{B}(j),\mathbf{S}(j)}$ for distinct $i,j$. By the unique construction given in the proof of Proposition \ref{714}, we have $$(\delta_{\mathbf{B}(i),\mathbf{S}(i)}, \delta_{\mathbf{C}(i,j), \mathbf{S}(i)}, \delta_{\mathbf{C}(j,i), \mathbf{S}(j)}, \delta_{\mathbf{C}(i,i),\mathbf{S}(i)}) = (1,1,0,\mathbf{1}_I(i))$$ for all $i,j$ where $i < j$. This means that \begin{align*}
d_{ij} &=
\begin{cases}
   1-3\cdot \mathbf{1}_I(i) & \text{if } i=j, \\[6pt]
 -1      & \text{if } i < j, \\[6pt] 0 & \text{if } i >j.
\end{cases}
\end{align*}
In other words, $\mathbf{D}$ is an upper triangular matrix with diagonal entries equal to $1-3 \cdot \mathbf{1}_I(i)$ and all entries above the diagonal equal to $-1$. The system $\mathbf{Ds} = \mathbf{1}$ thus encodes the linear system $$d_{ii}\mathbf{S}(i)-\sum_{j>i} \mathbf{S}(j) = d_{ii}\mu(S(i)) - \sum_{j > i} \mu(S(j)) = 1$$ over all $i \in [n]$. By inducting downwards with base case $\mu(S(n)) = \frac{1}{d_{nn}}$, we obtain $$\mu(S(i)) = \frac{1}{d_{ii}}\prod_{j=i+1}^n \frac{d_{jj}+1}{d_{jj}} = P(i) \prod_{j=i+1}^n \left(1 + P(j)\right)$$ 
for all $i$. Now recall that $$\mu(A) = \mu([\varnothing \to \bullet]) = 1 + \sum_{i=1}^n \mu(B(i)).$$ Notice that by letting $$Q(i) = \prod_{j=i}^n (1+P(j)),$$ we have $$\mu(S(i)) = Q(i)-Q(i+1),$$ which implies that \begin{align*}
    \mu(A) &= 1 + \sum_{i=1}^n \mu(S(i)) \\ &= Q(1) \\ &= \prod_{j=1}^n (1 + P(j)). 
\end{align*}
From the value of $\mu$ on irreducible marked structures, we now prove the general formula for $\mu$. We induct on $\ell(T)$. 

If $\ell(T) = 0$, then $\mu(T) = \mu(\varnothing \to \varnothing) = 1$, and if $\ell(T)=1$, then indeed $\mu(T) = \mu(A) = \prod_{j=1}^n (1 + P(j))$. If $\ell(T) \geq 2$, then $$\mu(T) = \mu(T \setminus a)\mu([T,a])$$ for some leaf $a \in T$ where $(T,a)$ denotes the embedding $T\setminus a \to T$. Let $c(a) \in [n]$ denote the color of $N(a)$ (the unique node adjacent to $a$). We split cases based on the position of $a$ in $T$.

If $N(a)$ is the root, then either $[T,a] = B(c(a))$ or $[T,a] = C(c(a),j)$ where $j$ is the color of the node closest to the root. In both cases, we have \begin{align*}\mu(T) &= \mu(T \setminus a) \cdot \mu(S(c(a)) \\ &= \mu(T \setminus a) \cdot \left(P(C(a)) \prod_{j=c(a)+1}^n (1+P(j))\right) \\ &= \Bigl(
   P(C(a))^{v_T(c(a))-1}
   \prod_{i \in [n]\setminus c(a)} P(i)^{v_T(i)}
   \prod_{i=1}^{c(a)} (1+P(i))^{1+\sum_{j=1}^{i-1} v_T(j)}
\\ &\qquad\cdot
   \prod_{i=c(a)+1}^{n} (1+P(i))^{\sum_{j=1}^{i-1} v_T(j)}
   \Bigr)
   \Bigl(
   P(C(a)) \prod_{j=c(a)+1}^{n} (1+P(j))
   \Bigr) \\ &= \prod_{i=1}^n \Bigl(P(i)^{v_T(i)}(1+P(i))^{1+\sum_{j=1}^{i-1}v_T(j)}\Bigr), \end{align*} which is the stated expression. 

If $N(a)$ is not the root and adjacent to only one other node, say, of color $j < c(a)$ (as $X_{j,c(a)} \in \partial \mathfrak{T}_3(n)^{\mathrm{ord}}_I)$, then $[T,a]=D(j,c(a))=B(c(a))-C(c(a),j)$ by Corollary \ref{linear}, so \begin{align*}\mu([T,a]) &= \mu(B(c(a))-C(c(a),j)) \\ &= \mu(S(c(a))\\ &= P(C(a)) \prod_{j=c(a)+1}^n (1+P(j)),\end{align*} and the rest of the computation is the same as the case above. 

Finally, if $N(a)$ is adjacent to two nodes, then we do casework on the colors of these two nodes. Here we illustrate the case where they both also have color $c(a)$; the other cases can be checked similarly using the relationships in Corollary \ref{linear}. Indeed, we have \begin{align*}[T,a] &= E(c(a),c(a),c(a)) \\ &= -1-\sum_{p \in [n]} D(c(a),p) \\ &= C(c(a),c(a)) + D(c(a),c(a))-B(c(a),c(a)) \\ &= 2C(c(a),c(a)) - S(c(a)),\end{align*} so again $\mu([T,a]) = \mu(S(c(a))$ as desired, completing the induction.
\end{proof}
\begin{remark}
    If we assume another induced total ordering $\sigma \in S_n$ (so the subclass is not necessarily naturally ordered), then we alternatively have $$\mu_n^I(S(i)) = P_I(i) \prod_{j=\sigma(i)+1}^n (1+P_I(j)).$$
\end{remark}
\begin{remark} We note an interesting occurrence of the lexicographic and wreath products. Suppose $I = [n]$. Let $\Omega_1$ and $\Omega_{\mathrm{ord}}$ denote the Fra\"iss\'e limits of $\partial \mathfrak{T}_3(1)$ and $\partial \mathfrak{T}_3(n)^{\mathrm{ord}}_I$, respectively. If $\Omega_1^{n \circ}$ denotes the $n$-iterated lexicographic product of $\Omega_1$, then one can check that $\Omega_{\mathrm{ord}} \cong \Omega_1^{\circ n}$. Consequently, if $\mathrm{Aut}(\Omega_1)^{\wr n}$ denotes the $n$-iterated wreath product of $\mathrm{Aut}(\Omega_1)$, then $\mathrm{Aut}(\Omega_{\mathrm{ord}}) \cong \mathrm{Aut}(\Omega_1)^{\wr n}$. 
\end{remark}
\subsection{The $n=2$ case} \label{n=2}
We end the section by illustrating our classification results on $\partial \mathfrak{T}_3(2)$ and assigning each of the $(2\cdot 2+2)^2=36$ measures to their corresponding support.

Let $\mathfrak{T}'$ be an infinite Fra\"iss\'e subclass of $\partial \mathfrak{T}_3(2)$ (if $\mathfrak{T}'$ is finite, then it only contains trees with at most four leaves). By Corollary \ref{class}, it suffices to determine the set of three-leaf trees in $\mathfrak{T}'$. Let $X_{1,1}$, $X_{1,2}$, $X_{2,1}$, $X_{2,2}$ denote the following trees, from left to right
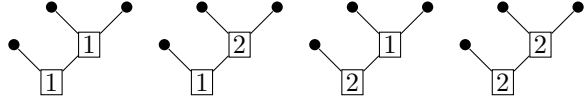
\begin{center}
    \begin{tikzpicture}
        \node (0) at (0,0) [rectangle, draw, inner sep = 1.5pt] {$1$};
        \node (1) at (-1/2,1/2) [circle, fill, inner sep = 1.5pt] {};
        \node (2) at (1/2,1/2) [rectangle, draw, inner sep = 1.5pt] {$1$};
        \node (3) at (0,1) [circle, fill, inner sep =1.5pt] {};
        \node (4) at (1,1) [circle, fill, inner sep = 1.5pt] {};

        \draw (0)--(1);
        \draw (0)--(2);
        \draw (2)--(3);
        \draw (2)--(4);

        \node (0) at (2,0) [rectangle, draw, inner sep = 1.5pt] {$1$};
        \node (1) at (3/2,1/2) [circle, fill, inner sep = 1.5pt] {};
        \node (2) at (5/2,1/2) [rectangle, draw, inner sep = 1.5pt] {$2$};
        \node (3) at (2,1) [circle, fill, inner sep =1.5pt] {};
        \node (4) at (3,1) [circle, fill, inner sep = 1.5pt] {};
\draw (0)--(1);
        \draw (0)--(2);
        \draw (2)--(3);
        \draw (2)--(4);
        \node (0) at (4,0) [rectangle, draw, inner sep = 1.5pt] {$2$};
        \node (1) at (7/2,1/2) [circle, fill, inner sep = 1.5pt] {};
        \node (2) at (9/2,1/2) [rectangle, draw, inner sep = 1.5pt] {$1$};
        \node (3) at (4,1) [circle, fill, inner sep =1.5pt] {};
        \node (4) at (5,1) [circle, fill, inner sep = 1.5pt] {};
\draw (0)--(1);
        \draw (0)--(2);
        \draw (2)--(3);
        \draw (2)--(4);
        \node (0) at (6,0) [rectangle, draw, inner sep = 1.5pt] {$2$};
        \node (1) at (11/2,1/2) [circle, fill, inner sep = 1.5pt] {};
        \node (2) at (13/2,1/2) [rectangle, draw, inner sep = 1.5pt] {$2$};
        \node (3) at (6,1) [circle, fill, inner sep =1.5pt] {};
        \node (4) at (7,1) [circle, fill, inner sep = 1.5pt] {};
    \draw (0)--(1);
        \draw (0)--(2);
        \draw (2)--(3);
        \draw (2)--(4);
    \end{tikzpicture}
\end{center}
and recall that $\varepsilon(i,j) = \mathbf{1}_{X_{i,j} \in \mathfrak{T}'}$. Notice that $\mathfrak{T}'$ is infinite if and only if $\varepsilon(i,i) \neq 0$ for some $i$. Define the Fra\"iss\'e subclasses \[
\begin{aligned}
\partial \mathfrak{T}_3(2)^{\mathrm{nt}\text{-}1}
&=\Bigl\{T\in \partial \mathfrak{T}_3(2):
\text{every node not adjacent to a leaf has color }1\Bigr\},\\
\partial \mathfrak{T}_3(2)^{\mathrm{nr}\text{-}1}
&=\Bigl\{T\in \partial \mathfrak{T}_3(2):
\text{every non-root node has color }1\Bigr\},\\
\partial \mathfrak{T}_3(2)^{\mathrm{nt}\text{-}2}
&=\Bigl\{T\in \partial \mathfrak{T}_3(2):
\text{every node not adjacent to a leaf has color }2\Bigr\},\\
\partial \mathfrak{T}_3(2)^{\mathrm{nr}\text{-}2}
&=\Bigl\{T\in \partial \mathfrak{T}_3(2):
\text{every non-root node has color }2\Bigr\},\\
\partial \mathfrak{T}_3(2)^{\mathrm{ord}} &= \Bigl\{T\in \partial \mathfrak{T}_3(2): \text{no $2$ lies above a $1$ on any root-leaf path}\Bigr\},\\ 
\partial \mathfrak{T}_3(2)^{\mathrm{rev}} &= \Bigl\{T\in \partial \mathfrak{T}_3(2): \text{no $1$ lies above a $2$ on any root-leaf path}\Bigr\}.
\end{aligned}
\]

Using Lemma \ref{73}, we match the existence of three-leaf trees in $\mathfrak{T}'$ and their corresponding Fra\"iss\'e subclasses:
\[
\begin{array}{c c c c | c}
\varepsilon(1,2) & \varepsilon(2,1) & \varepsilon(1,1) & \varepsilon(2,2) & \text{Subclass } \\
\hline
0 & 0 & 1 & 0 & \partial \mathfrak{T}_3(1)\\
0 & 0 & 0 & 1 & \partial \mathfrak{T}_3(1)\\
1 & 0 & 1 & 0 & \partial\mathfrak{T}_3(2)^{\mathrm{nt}\text{-}1}\\
1 & 0 & 0 & 1 & \partial\mathfrak{T}_3(2)^{\mathrm{nr}\text{-}2}\\
1 & 0 & 1 & 1 & \partial \mathfrak{T}_3(2)^{\mathrm{ord}}\\
0 & 1 & 0 & 1 & \partial\mathfrak{T}_3(2)^{\mathrm{nt}\text{-}2}\\
0 & 1 & 1 & 0 & \partial\mathfrak{T}_3(2)^{\mathrm{nr}\text{-}1}\\
0 & 1 & 1 & 1 & \partial\mathfrak{T}_3(2)^{\mathrm{rev}}\\
1 & 1 & 1 & 1 & \partial \mathfrak{T}_3(2)\\
\end{array}
\]

Given an induced regular measure $\mu'$ on $\mathfrak{T}'$, by Lemma \ref{bk},  $\varepsilon(i,j)=1$ for distinct $i,j$ if and only if $\mu'(B(i)), \mu'(B(j))$, and $\mu'(C(i,j))$ are all nonzero. Meanwhile, $\varepsilon(i,i)=1$ if and only if $\mu'(B(i))$ and $\mu'(C(i,i))$ are both nonzero. 

We use the computer algebra system Macaulay 2 \cite{M2} to analyze the polynomial ring $\QQ[B(1), C(1,1), C(1,2), B(2), C(2,1), C(2,2)]$ modulo the ideal $I$ generated by the quadratic equations from Proposition \ref{system} for $n=2$. The command \texttt{dim(R/I)} returns $0$\ and \texttt{degree(R/I)} returns $36$, so the sytem has $36$ complex solutions counted with multiplicity. By calling the \texttt{rationalPoints} function, we obtain $36$ distinct solutions with coordinates that are powers of two up to sign, corresponding to $36$ $\ZZ\left[\frac 12 \right]$-valued measures on $\partial \mathfrak{T}_3(2)$, which verifies our theoretical results from Section \ref{ms}. In Appendix \ref{list}, we list the values of each of the measures and their supports according to the above table.
\begin{remark}
By restricting each of the measures $\mu_1, \mu_2, \dots, \mu_8$ to their support, we recover the induced regular measure $\mu_1^I$ on $\partial \mathfrak{T}_3(1)^{\mathrm{ord}}_I = \partial \mathfrak{T}_3(1)$ where $I = \{1\}$. This is the only regular measure on $\partial \mathfrak{T}_3(n)$ across all $n \geq 1$.    
\end{remark}
\section{Applications to tensor categories} \label{category}
In this section, we construct families of semisimple tensor categories using the induced subclasses of $\partial \mathfrak{T}_3(n)$ and induced regular measures from Section \ref{sec5}. We first recall some background on tensor categories and the Harman--Snowden construction. We then utilize oligomorphic groups to prove a semisimplicity criterion for these categories. Finally, we specialize the construction to the family of induced subclasses $\partial \mathfrak{T}_3(n)^{\mathrm{ord}}_I$, producing infinite families of semisimple tensor categories of superexponential growth.

\subsection{Background on tensor categories} Let $\mathbb{K}$ be a field, and $\mathcal{C}$ be a $\mathbb{K}$-linear symmetric monoidal category (with a bilinear tensor product) with monoidal unit $\mathds{1}$. An object $X$ is said to be \emph{rigid} if there is a dual object $X^*$ equipped with an evaluation morphism $\mathrm{ev}_X:X^* \otimes X \to \mathds{1}$ and coevaluation morphism $\mathrm{coev}_X : \mathds{1} \to X \otimes X^*$ satisfying some properties \cite[\S 2.10]{etingof2017tensor}. A category is rigid if all of its objects are rigid, and \emph{locally finite} if all objects have finite length and all Hom spaces are finite-dimensional.
\begin{defn} \label{tenscat}
    A \emph{tensor category} is a $\mathbb{K}$-linear symmetric monoidal category $\mathcal{C}$ that is abelian, locally finite, rigid, and $\mathrm{End}(\mathds{1})=\mathbb{K}$.
\end{defn}
An important property of tensor categories is their growth, which is defined based on their objects. For an object $X$, let \emph{length} refer to its Jordan--Holder length.
\begin{defn} \label{subexp}
    A tensor category $\mathcal{C}$ is of \emph{subexponential growth} if for every object $X \in \mathcal{C}$, there exists a natural number $N_X$ such that $\mathrm{length}(X^{\otimes n }) \leq (N_X)^n$ for each $n \geq 0$. Otherwise, $\mathcal{C}$ is of \emph{superexponential growth}.
\end{defn}
\subsection{The Harman--Snowden construction} Suppose we have a Fra\"iss\'e class $\mathfrak{F}$ equipped with a $\CC$-valued measure $\mu$. We now recall the Harman--Snowden construction \cite{harman2024oligomorphicgroupstensorcategories}, which produces the $\CC$-linear rigid symmetric monoidal category $\mathbf{Perm}(\mathfrak{F};\mu)$.

For each structure $X \in \mathfrak{F}$, let the formal symbol $\mathrm{Vec}_X$ be the corresponding object. Given two objects $\mathrm{Vec}_X$ and $\mathrm{Vec}_Y$, let  $$\mathrm{Hom}(X,Y)=\mathbb{C}[\mathrm{Amalg}(X,Y)]$$ where $\mathrm{Amalg}(X,Y)$ denotes the set of isomorphism classes of amalgamations of $X$ and $Y$. In other words, $\mathrm{Hom}(X,Y)$ is the $\CC$-module freely generated by isomorphism classes of amalgamations of $X$ and $Y$, of which there are finitely many.    

Now suppose $Y_{1,2}$ is an amalgamation of $X_1, X_2 \in \mathfrak{F}$ and $Y_{2,3}$ is an amalgamation of $X_2,X_3 \in \mathfrak{F}$. Let $\varphi_{1,2}$ and $\varphi_{2,3}$ denote the corresponding morphisms $\mathrm{Vec}_{X_1} \to \mathrm{Vec}_{X_2}$ and $\mathrm{Vec}_{X_2} \to \mathrm{Vec}_{X_3}$. Composition is given by $$\varphi_{2,3} \circ \varphi_{1,2} = \sum_{Y_{1,3}} \sum_{Z} \mu(Y_{1,3} \subset Z) \cdot \varphi_{1,3}$$ where $Y_{1,3}$ ranges over all amalgamations of $X_1$ and $X_3$ and $Z$ ranges over all amalgamations of $X_1$, $X_2$, and $X_3$ such that the following diagram commutes: 
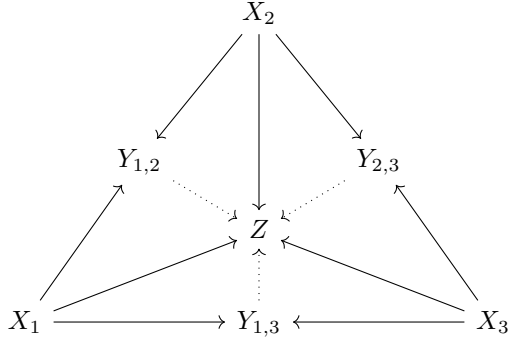
\begin{center}\[\begin{tikzcd}[row sep=1em, column sep = 1em]
	&&&& {X_2} \\
	\\
	\\
	\\
	&& {Y_{1,2}} &&&& {Y_{2,3}} \\
	&&&& Z \\
	\\
	{X_1} &&&& {Y_{1,3}} &&&& {X_3}
	\arrow[from=1-5, to=5-3]
	\arrow[from=1-5, to=5-7]
	\arrow[from=1-5, to=6-5]
	\arrow[dotted, from=5-3, to=6-5]
	\arrow[dotted, from=5-7, to=6-5]
	\arrow[from=8-1, to=5-3]
	\arrow[from=8-1, to=6-5]
	\arrow[from=8-1, to=8-5]
	\arrow[dotted, from=8-5, to=6-5]
	\arrow[from=8-9, to=5-7]
	\arrow[from=8-9, to=6-5]
	\arrow[from=8-9, to=8-5]
\end{tikzcd}\] \end{center}
Here, all the arrows denote embeddings in $\mathfrak{F}$. The above construction bilinearly extends to arbitrary compositions. We define $\mathbf{Perm}(\mathfrak{F}; \mu)$ to be the additive category freely generated by the objects $\{\mathrm{Vec}_X\}$, where morphisms extend uniquely by additivity. There is a tensor product $\mathbf{\otimes}$ on $\mathbf{Perm}(\mathfrak{F}; \mu)$ where $$\mathrm{Vec }_X \, \mathbf{\otimes} \, \mathrm{ Vec}_Y = \bigoplus_{Z \in \mathrm{Amalg}(X,Y)} \mathrm{Vec}_Z,$$ which has the expected tensor properties \cite{harman2024oligomorphicgroupstensorcategories}. 

The unit object is $\mathds{1} = \mathrm{Vec}_{\varnothing}$, and all objects are self-dual with evaluation (which acts on summands) and coevaluation given by projection onto and inclusion of the diagonal amalgamation, respectively. The action of $\otimes$ on morphisms allows rigidity properties to hold; again, see \cite{harman2024arborealtensorcategories} for further details. It follows from the construction that the categorical dimension of $\mathrm{Vec}_X$ is $\mu(X)$, and the trace of an endomorphism equals $\mu(X)$ for the diagonal amalgamation and vanishes otherwise. 

We further recall one construction.
\begin{defn}
    The \emph{Karoubi envelope} $\mathcal{C}^{\mathrm{kar}}$ of a category $\mathcal{C}$ has objects given by $(X,e)$, where $X \in \mathcal{C}$ and $e: X \to X$ is an idempotent. A morphism $(X,e) \to (Y,f)$ is a morphism $m: X \to Y$ in $\mathcal{C}$ such that $m = f \circ m \circ e$, and composition is inherited from $\mathcal{C}$.
\end{defn}
\begin{remark}
    For concreteness, we work over $\CC$ in this paper. All constructions and statements (including those for measures) remain valid over any field $k$ of characteristic $\neq 2$ with identical proofs.
\end{remark}
\subsection{The oligomorphic group perspective} \label{perspective}
For a Fra\"iss\'e class $\mathfrak{F}$, let $\Omega$ be its Fra\"iss\'e limit. From our discussion in Section \ref{s2}, the group $G = \mathrm{Aut}(\Omega)$ with its natural action on $\Omega$ is oligomorphic. There is a relative version of $G$, which is $(G, \mathscr{E})$, the class of transitive $\mathscr{E}$-smooth $G$-sets where $\mathscr{E}$ is the collection of pointwise stabilizers of finite subsets of $\Omega$. Returning to this perspective allows us to establish category-theoretical results applicable to Fra\"iss\'e classes and measures on them.

Recall that given a measure $\mu$ on $\mathfrak{F}$, there is a unique corresponding measure $\mu^*$ on $(G, \mathscr{E})$ such that if $i : X \to Y$ is an embedding in $\mathfrak{F}$ inducing a map of finitary $\mathscr{E}$-smooth $G$ sets $i^* : \Omega^Y \to \Omega^X$, where $\Omega^X$ is the set of embeddings $X \to \Omega$, then $\mu^*(i^*) = \mu(i)$. We similarly define the $\CC$-linear rigid symmetric monoidal category $\mathbf{Perm}(G,\mathscr{E}; \mu^*)$. Its objects are $\mathrm{Vec}_{\mathscr{X}}$ where $\mathscr{X}$ is a finitary $\mathscr{E}$-smooth $G$-set; morphisms $\mathrm{Vec}_{\mathscr{X}} \to \mathrm{Vec}_{\mathscr{Y}}$ are given by $G$-invariant $\mathscr{Y} \times \mathscr{X}$ matrices where composition is given by matrix multiplication relying on $\mu^*$ that follows a theory of integration developed in \cite[\S 7]{harman2024oligomorphicgroupstensorcategories}. The category's tensor product is defined on objects by $\mathrm{Vec}_{\mathscr{X}} \otimes \mathrm{Vec}_{\mathscr{Y}} = \mathrm{Vec}_{\mathscr{X} \times \mathscr{Y}}$. By \cite[Theorem 6.9, p.45]{harman2024oligomorphicgroupstensorcategories}, there is an equivalence of categories $$\mathbf{Perm}(G,\mathscr{E};\mu^*) \cong \mathbf{Perm}(\mathfrak{F}; \mu)$$ where $\mathrm{Vec}_X$ corresponds to $\mathrm{Vec}_{\Omega^X}$, and $\mathrm{Hom}(\mathrm{Vec}_X, \mathrm{Vec}_Y)$ corresponds to \\ $\mathrm{Hom}(\mathrm{Vec}_{\Omega^X}, \mathrm{Vec}_{\Omega^Y})$ via identifying $G$-orbits on $\Omega^Y \times \Omega^X$ with amalgamations of $X$ and $Y$. Consequently, we can translate properties of $(G, \mathscr{E})$ and $\mathfrak{F}$ back and forth, which is especially useful for certain results that are more conveniently considered from the former's perspective.
\subsection{Semi-simplicity criterion}
We now introduce a criterion for semisimple tensor categories that arise from Fra\"iss\'e classes.
\begin{prop} \label{condition} Let $\mu$ be regular measure on a Fraïssé class $\mathfrak{F}$ valued in $\ZZ\left[\frac{1}{N}\right]$ for some integer $N \geq 1$. If every prime divisor of $|\mathrm{Aut}(X)|$ divides $N$ for every structure $X \in \mathfrak{F}$, then the Karoubi envelope of $\mathbf{Perm}(\mathfrak{F};\mu)$, denoted $\mathbf{Rep}(\mathfrak{F};\mu)$, is a semisimple tensor category. 
\end{prop} 
Let $\Omega$ be the Fra\"iss\'e limit of $\mathfrak{F}$, and $G$ be the automorphism group of $\Omega$. We work with the corresponding measure $\mu^*$ on $(G, \mathscr{E})$,  which is also $\ZZ\left[\frac 1N\right]$-valued. Let $\mathscr{E}^+$ be the set of all open subgroups of $G$ that contain some member of $\mathscr{E}$ with finite index.
\begin{lemma} \label{eplus}
    Every $\ZZ\left[\frac 1N\right]$-valued regular measure $\mu^*$ on $(G, \mathscr{E})$ extends uniquely to a $\ZZ\left[\frac 1N\right]$-valued regular measure $\mu^+$ on $(G, \mathscr{E}^+)$. 
\end{lemma}
\begin{proof}
We follow the proofs of \cite[Theorem 2.15]{nekrasov2024upperboundsmeasuresdistal} and \cite[Lemma 4.4]{harman2024arborealtensorcategories}. Let $V \in \mathscr{E^+}$. By definition, there is some finite $A \subset \Omega$ such that if $X$ is the algebraic closure of $A$, then we have the finite-index containments $$G(X) \subset G(A) \subset V \subset G[X]$$ where $G[X]$ is the subgroup of $G$ fixing $X$ as a set. It follows that $$\mu^+(G/V) = [V:G(X)]^{-1} \cdot \mu^*(G/G(X)).$$ Since every prime divisor of $|\mathrm{Aut}(X)| = [G[X]:G(X)]$ divides $N$, so does every prime divisor of $[V:G(X)]$, implying that $\mu^+(G/V)$ is a well-defined $\ZZ\left[\frac 1N\right]$-valued measure. But $\mu^*$ is regular, so $\mu^+(G/V)$ is a unit in $\ZZ\left[\frac 1N\right]$ and $\mu^+$ is also regular, as desired.    
\end{proof}
\begin{remark}
    Note that the previous Lemma shows an isomorphism 
    \[\Theta (G, \mathscr{E})\left[1/N\right]\cong \Theta (G, \mathscr{E}^+)\left[1/N\right],\]
    and therefore, if one considers more general ring-valued measures, we have the equality of corresponding measures as long as $N$ is invertible in the ring.
\end{remark}
\begin{proof}[Proof of Proposition \ref{condition}]
   From Lemma \ref{eplus}, consider the rigid symmetric monoidal category $\mathbf{Perm}(G, \mathscr{E}^+;\mu^+)$ where $\mu^+$ is a regular $\ZZ\left[\frac 1N\right]$-valued measure on $(G,\mathscr{E}^+)$ that restricts to a regular measure $\mu^*$ on $(G,\mathscr{E})$. Since $\ZZ\left[\frac 1N \right]$ can be reduced mod $p$ for all $p \nmid N$, and $\mathscr{E}^+$ is upward closed under finite-index containment, we apply \cite[Theorem 7.11]{harman2024oligomorphicgroupstensorcategories} to obtain that nilpotent endomorphisms in the category have categorical trace $0$. Combined with the regularity of $\mu^+$, by \cite[Proposition 7.19]{harman2024oligomorphicgroupstensorcategories}, endomorphism algebras in $\mathbf{Perm}(G, \mathscr{E}^+;\mu^+)$ are semisimple. It follows that endomorphism algebras in $\mathbf{Perm}(G, \mathscr{E};\mu^*)$, which is a full subcategory of $\mathbf{Perm}(G, \mathscr{E}^+;\mu^+)$, are also semisimple. This implies that the Karoubi envelope of $\mathbf{Perm}(G, \mathscr{E};\mu^*)$ is a semisimple tensor category. The proposition thus follows from our discussion in \ref{perspective}.
\end{proof}
Finally, given a (not necessarily regular) measure $\mu$ with support $\mathfrak{F}'$ and induced regular measure $\mu'$, we can relate their two resulting categories through a process called semi-simplification \cite[\S 2.3]{etingof2019semisimplificationtensorcategories} as follows.
\begin{prop}[{\cite[Corollary 4.6, p.24]{harman2024arborealtensorcategories}}] \label{semi}
If $\mathbf{Rep}(\mathfrak{F}';\mu')$ is semisimple, then nilpotent endormophisms in $\mathbf{Rep}(\mathfrak{F};\mu)$ have categorical trace $0$, and $\mathbf{Rep}(\mathfrak{F}';\mu')$ is the semi-simplification of $\mathbf{Rep}(\mathfrak{F};\mu)$.
\end{prop}
\subsection{Results} \label{results}
Recall the collection of induced subclasses $\partial \mathfrak{T}_3(n)^{\mathrm{ord}}_I$ with $n$ colors each consisting of trees where colors monotonically increase along every root-leaf path and containing $X_{i,i}$ if and only if $i \in I$. There are $2^n$ of these induced subclasses, each equipped with an induced regular measure $\mu_n^I$. Below are the central constructions of our paper, which we have worked towards in the past sections.
\begin{theorem} \label{main}
For each integer $n \geq 1$ and each subset $I \subseteq [n]$, the category $\mathbf{Rep}(\partial \mathfrak{T}_3(n)^{\mathrm{ord}}_I;\mu^I_n)$ is a semisimple tensor category.
\end{theorem}
We will show the above theorem by applying Proposition \ref{condition}. Since $\mu_n^I$ is always regular on its induced subclass, it suffices to check the semisimplicity criterion.
\begin{lemma} \label{poweroftwo}
    For each $n \geq 1$ and $T \in \partial \mathfrak{T}_3(n)$, we have $|\mathrm{Aut}(T)| = 2^k$ for some nonnegative integer $k$.
\end{lemma}
\begin{proof}
    Let $\mathrm{Aut}(X)$ denote the automorphism group of the entire tree (which also includes nodes) preserving edge relations and relation $S$. Note that $\mathrm{Aut}(T)$ is a subgroup of $\mathrm{Aut}(X)$, so it suffices to show that $|\mathrm{Aut}(X)|$ is a power of two.

    We proceed by induction on $|X|$. For the base case $|X|=1$, $X$ is the one-leaf tree where the only automorphism is the identity map, so $|\mathrm{Aut}(X)|=1$. Now, given a tree $T$, suppose its root branches out to subtrees $U$ and $V$:
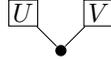
    \begin{center}
    \begin{tikzpicture}
        \node (0) at (0,0) [circle, fill, inner sep = 1.5pt] {};
        \node (1) at (-1/2,1/2) [rectangle, draw, inner sep = 1.5pt] {$U$};
        \node (2) at (1/2,1/2) [rectangle, draw, inner sep = 1.5pt] {$V$};
        \draw (0)--(1);
        \draw (0)--(2);
    \end{tikzpicture}
    \end{center}
If $U \not\cong V$, then $\mathrm{Aut}(X)$ acts on $U$ and $V$ separately, so $\mathrm{Aut}(X) \cong \mathrm{Aut}(U) \times \mathrm{Aut}(V)$. Since $|U|,|V|<|X|$, by the inductive hypothesis, $$|\mathrm{Aut}(X)|=|\mathrm{Aut}(U)||\mathrm{Aut}(V)|$$ is indeed a power of two.

If $U \cong V$, then there is an automorphism $s$ swapping the children of the root, thus also swapping $U$ and $V$. Notice that $\{1,s\} \cong \ZZ/2\ZZ$ permutes the normal subgroup $\mathrm{Aut}(U) \times \mathrm{Aut}(V) \unlhd \mathrm{Aut}(X)$, so we have $$\mathrm{Aut}(T) \cong (\mathrm{Aut}(U) \times \mathrm{Aut}(V)) \rtimes \ZZ/2\ZZ,$$ which implies that $$|\mathrm{Aut}(X)| = 2|\mathrm{Aut}(U)||\mathrm{Aut}(V)|.$$ This is again a power of two by the inductive hypothesis.
\end{proof}
This completes the proof of Theorem \ref{main}. For each subset $I$ of $[n]$, let $\tilde{\mu}^I_n$ denote an extension of $\mu^I_n$ to $\partial \mathfrak{T}_3(n) \supseteq \partial \mathfrak{T}_3(n)^{\mathrm{ord}}_I$.
\begin{corollary}
    Let $n \geq 1$ be an integer. Then nilpotent endomorphisms in $\mathbf{Rep}(\partial \mathfrak{T}_3(n);\tilde{\mu}^I_n)$ have categorical trace $0$, and the semi-simplification of \\$\mathbf{Rep}(\partial \mathfrak{T}_3(n);\tilde{\mu}^I_n)$ is $\mathbf{Rep}(\partial \mathfrak{T}_3(n)^{\mathrm{ord}}_I;\mu^I_n)$.
\end{corollary}
\begin{proof}
    By Theorem \ref{main}, $\mathbf{Rep}(\partial \mathfrak{T}_3(n)^{\mathrm{ord}}_I;\mu^I_n)$ is a semisimple tensor category. The support of $\tilde{\mu}^I_n$ is $\partial \mathfrak{T}_3(n)^{\mathrm{ord}}_I$, whose induced regular measure is $\mu^I_n$, so the corollary follows from Proposition \ref{semi}.
\end{proof}
Finally, we explicitly show that our categories have superexponential growth, which is an essential property of tensor categories arising from measures.
\begin{prop}
For each $n \geq 1$ and each nonempty $I \subseteq [n]$, the tensor category $\mathbf{Rep}(\partial \mathfrak{T}_3(n)^{\mathrm{ord}}_I;\mu^I_n)$ is of superexponential growth.
\end{prop}
\begin{proof}
We resolve the $n=1$ case, from which the result easily follows for larger $n$. For each structure $X$, let $f_m(X)$ be the number of ways to label the elements of $X$ with $\{1, \dots, m\}$ up to an automorphism (where every element has at least one label). For example, if $X$ is the tree structure with two leaves, then $f_m(X)=2^{m-1}-1$ for all $m \geq 2$. Let $T$ be the one-leaf tree structure. By the definition of the tensor product, we find that $$\mathrm{Vec}_T^{\otimes m} = \bigoplus_{|X| \leq m} \mathrm{Vec}_X^{\oplus f_m(X)}$$ where each copy of $\mathrm{Vec}_X$ comes from a distinct choice of embeddings $(i_1, \dots, i_m) : T^m \to X$, which there are precisely $f_m(X)$ of. Now, since $\mathbf{Rep}(\partial \mathfrak{T}_3(1)^{\mathrm{ord}}_I;\mu^I_1)$ is a semisimple tensor category (Theorem \ref{main}), the inequalities $$\mathrm{length}(X^{\otimes m}) \leq \dim \End (X^{\otimes m}) \leq \mathrm{length}(X^{\otimes m})^2$$ hold for all objects $X$ in the category, so we alternatively consider $$\mathrm{End}(\mathrm{Vec}_T^{\otimes m}) = \mathrm{End}\left(\bigoplus_{|X| \leq m} \mathrm{Vec}_X^{\oplus f_m(X)}\right) = \bigoplus_{|X|,|X'| \leq m} \mathrm{Hom}\left(\mathrm{Vec}_X^{\oplus f_m(X)},\mathrm{Vec}_{X'}^{\oplus f_m(X')}\right).$$ In particular,  \begin{align*}
    \dim \mathrm{End}(\mathrm{Vec}_T^{\otimes m}) &\geq \dim \bigoplus_{|X|=|X'| = m} \mathrm{Hom}\left(\mathrm{Vec}_X^{\oplus f_m(X)},\mathrm{Vec}_{X'}^{\oplus f_m(X')}\right) \\ &= \sum_{|X| = |X'| = m} f_m(X)f_m(X')\left|\mathrm{Amalg}(X,X')\right| \\ &\geq \left(\sum_{|X| = m} f_m(X)\right)^2. 
\end{align*} 
Observe that the sum inside parentheses is just the number of labeled $m$-element structures in $\partial \mathfrak{T}_3(1)$, which is equal to $(2m-3)!! = (2m-3)(2m-5) \cdots 3 \cdot 1$ by \cite[A001147]{oeis}. As a result, we obtain \begin{align*}
    \mathrm{length}(X^{\otimes m}) &\geq \sqrt{\dim \mathrm{End}(\mathrm{Vec}_T^{\otimes m})} \\ &\geq \sum_{|X|=m} f_m(X) \\ &= (2m-3)!! \\ &= \exp{(\Theta(m \log m))},
    \end{align*} which is superexponential growth.
\end{proof}
\begin{remark}
    If $I$ is empty, then $X_{i,i} \notin \partial \mathfrak{T}_3(n)_I^{\mathrm{ord}}$ for all $i \in [n]$. As a result, its Fra\"iss\'e limit is finite (with $2^n$ elements), so $\mathbf{Rep}(\partial \mathfrak{T}_3(n)_{\varnothing}^{\mathrm{ord}};\mu_n^{\varnothing})$ is a symmetric fusion category with exponential growth, and thus equivalent to the category of representations of an affine supergroup scheme by Deligne's theorem.
\end{remark}
\newpage
\appendix
\section{\texorpdfstring{Measures on $\partial \mathfrak{T}_3(2)$ and their supports}{Measures on dT3(2) and their induced subclasses}} \label{list}
In the table below, we list the values of each of the $36$ measures on $\partial \mathfrak{T}_3(2)$. In particular, the listed values are for the irreducible marked structures in $\Theta(\partial \mathfrak{T}_3(2))$ that generate the rest of the measure ring. The last column indicates the supports of the measures (\ref{induced}, \ref{n=2}), which is labeled as ``Trivial" if the support is finite.
\begin{table}[ht]
\centering
\begin{tabular}{c|rrrrrr|c}

 & $B(1)$ & $C(1,1)$ & $C(1,2)$ & $B(2)$ & $C(2,1)$ & $C(2,2)$ & Support \\
\hline
$\mu_{1}$  & -0.5  & -0.5  & -0.5  & 0     & 0     & -1    & \multirow{8}{*}{$\partial \mathfrak{T}_3(1)$} \\
$\mu_{2}$  & -0.5  & -0.5  & -0.5  & 0     & 0     & 0     & \\
$\mu_{3}$  & -0.5  & -0.5  & 0     & 0     & 0     & -0.5  & \\
$\mu_{4}$  & -0.5  & -0.5  & 0     & 0     & 0     & 0     & \\
$\mu_{5}$  & 0     & -1    & 0     & -0.5  & -0.5  & -0.5  & \\
$\mu_{6}$  & 0     & -0.5  & 0     & -0.5  & 0     & -0.5  & \\
$\mu_{7}$  & 0     & 0     & 0     & -0.5  & -0.5  & -0.5  & \\
$\mu_{8}$  & 0     & 0     & 0     & -0.5  & 0     & -0.5  & \\
\hline
$\mu_{9}$ & -1    & -1    & -1    & 1     & 0     & 0     & $\partial\mathfrak{T}_3(2)^{\mathrm{nt}\text{-}1}$ \\

$\mu_{10}$ & 0.5   & 0     & 0.5   & -0.5  & 0     & -0.5  & $\partial\mathfrak{T}_3(2)^{\mathrm{nr}\text{-}2}$ \\

$\mu_{11}$ & -0.25 & -0.25 & -0.25 & -0.5  & 0     & -0.5  & $\partial \mathfrak{T}_3(2)^{\mathrm{ord}}$ \\

$\mu_{12}$ & 1     & 0     & 0     & -1    & -1    & -1    & $\partial\mathfrak{T}_3(2)^{\mathrm{nt}\text{-}2}$ \\

$\mu_{13}$ & -0.5  & -0.5  & 0     & 0.5   & 0.5   & 0     & $\partial\mathfrak{T}_3(2)^{\mathrm{nr}\text{-}1}$ \\

$\mu_{14}$  & -0.5  & -0.5  & 0     & -0.25 & -0.25 & -0.25 & $\partial \mathfrak{T}_3(2)^{\mathrm{rev}}$ \\
\hline
$\mu_{15}$ & 0     & -2    & 0     & 0     & -1    & -1    & \multirow{22}{*}{Trivial} \\
$\mu_{16}$ & 0     & -1    & -1    & 0     & 0     & 0     & \\
$\mu_{17}$ & 0     & -1    & 0     & 0     & 0     & -1    & \\
$\mu_{18}$ & 0     & -1    & -1    & 0     & 0     & -2    &  \\
$\mu_{19}$ & 0     & -1    & 0     & 0     & 0     & 0     &  \\
$\mu_{20}$ & 0     & -0.5  & 0     & 0     & 0.5   & 0     &  \\
$\mu_{21}$ & 0     & 0     & 0     & 0     & -1    & -1    &  \\
$\mu_{22}$ & 0     & 0     & 0     & 0     & 0     & -1    &  \\
$\mu_{23}$ & 0     & 0     & 0     & 0     & 0     & 0     &  \\
$\mu_{24}$ & 0     & 0     & 0.5   & 0     & 0     & -0.5  & \\
$\mu_{25}$ & 0     & 0     & 0     & 0     & 0.5   & 0     &  \\
$\mu_{26}$ & 0     & 0     & 0.5   & 0     & 0     & 0     &  \\
$\mu_{27}$ & 0     & -2    & 0     & 1     & 0     & 0     &  \\
$\mu_{28}$ & 0     & -1    & 0     & 1     & 1     & 0     &  \\
$\mu_{29}$ & 0     & 0     & 0     & 1     & 0     & 0     &  \\
$\mu_{30}$ & 0     & 0     & 0     & 1     & 1     & 0     &  \\
$\mu_{31}$ & 1     & 0     & 0     & 0     & 0     & 0     &  \\
$\mu_{32}$ & 1     & 0     & 1     & 0     & 0     & -1    & \\
$\mu_{33}$ & 1     & 0     & 0     & 0     & 0     & -2    &  \\
$\mu_{34}$ & 1     & 0     & 1     & 0     & 0     & 0     &  \\
$\mu_{35}$ & 1     & 0     & 0     & 2     & 2     & 0     &  \\
$\mu_{36}$ & 2     & 0     & 2     & 1     & 0     & 0     &  \\
\bottomrule
\end{tabular}
\end{table}

\bibliographystyle{amsalpha}
\bibliography{biblio}

@misc{harman2024oligomorphicgroupstensorcategories,
  author       = {Nate Harman and Andrew Snowden},
  title        = {Oligomorphic groups and tensor categories},
  year         = {2024},
  howpublished = {arXiv:2204.04526 [math.RT]. Available at \url{https://arxiv.org/abs/2204.04526}}
}

@misc{oeis,
  author       = {{OEIS Foundation Inc.}},
  title        = {The {O}n-{L}ine {E}ncyclopedia of {I}nteger {S}equences},
  year         = {2025},
  howpublished = {Published electronically at \url{https://oeis.org}}
}

@misc{harman2024arborealtensorcategories,
  author       = {Nate Harman and Ilia Nekrasov and Andrew Snowden},
  title        = {Arboreal tensor categories},
  year         = {2024},
  howpublished = {arXiv:2308.06660 [math.RT]. Available at \url{https://arxiv.org/abs/2308.06660}}
}

@misc{snowden2024thirtysevenmeasurespermutations,
  author       = {Andrew Snowden},
  title        = {The thirty-seven measures on permutations},
  year         = {2024},
  howpublished = {arXiv:2404.08775 [math.CO]. Available at \url{https://arxiv.org/abs/2404.08775}}
}

@article{Del02,
  author  = {Pierre Deligne},
  title   = {Cat{\'e}gories tensorielles},
  journal = {Mosc. Math. J.},
  volume  = {2},
  year    = {2002},
  pages   = {227--248},
  doi     = {10.17323/1609-4514-2002-2-2-227-248},
  note    = {Available at \url{http://mi.mathnet.ru/mmj54}}
}

@misc{snowden2023measurescoloredcircle,
  author       = {Andrew Snowden},
  title        = {Measures for the colored circle},
  year         = {2023},
  howpublished = {arXiv:2302.08699 [math.CO]. Available at \url{https://arxiv.org/abs/2302.08699}}
}

@article{Fraisse1953,
  author  = {Fra\"iss\'e, Roland},
  title   = {Sur certaines relations qui g\'en\'eralisent l'ordre des nombres rationnels},
  journal = {C. R. Acad. Sci. Paris},
  volume  = {237},
  year    = {1953},
  pages   = {540--542}
}

@article{treelikeobjects,
  author  = {Cameron, Peter J.},
  title   = {Some treelike objects},
  journal = {Quart. J. Math. Oxford Ser. (2)},
  volume  = {38},
  year    = {1987},
  pages   = {155--183},
  doi     = {10.1093/qmath/38.2.155},
  note    = {Available at \url{https://doi.org/10.1093/qmath/38.2.155}}
}

@misc{M2,
  author       = {Grayson, Daniel R. and Stillman, Michael E.},
  title        = {Macaulay2, a software system for research in algebraic geometry},
  year         = {2026},
  howpublished = {Available at \url{http://www2.macaulay2.com}}
}

@article{Cameron95,
  author  = {Cameron, Peter J.},
  title   = {Counting two-graphs related to trees},
  journal = {Electron. J. Combin.},
  volume  = {2},
  year    = {1995},
  pages   = {Research Paper 4, approx.\ 8},
  doi     = {10.37236/1198},
  note    = {Available at \url{https://doi.org/10.37236/1198}}
}

@misc{Nekrasov2024TensorialMeasures,
  author       = {Nekrasov, Ilia},
  title        = {Tensorial measures on {$\omega$}-categorical structures},
  year         = {2024},
  howpublished = {Deep Blue Documents Repository, University of Michigan. Available at \url{https://deepblue.lib.umich.edu/items/949af992-59ad-4437-82a5-fefa8d6f4585}}
}

@misc{nekrasov2024upperboundsmeasuresdistal,
  author       = {Ilia Nekrasov and Andrew Snowden},
  title        = {Upper bounds for measures on distal classes},
  year         = {2024},
  howpublished = {arXiv:2407.19131 [math.RT]. Available at \url{https://arxiv.org/abs/2407.19131}}
}

@book{etingof2017tensor,
  author    = {Etingof, Pavel and Gelaki, Shlomo and Nikshych, Dmitri and Ostrik, Victor},
  title     = {Tensor categories},
  publisher = {American Mathematical Society},
  year      = {2017}
}

@misc{etingof2019semisimplificationtensorcategories,
  author       = {Pavel Etingof and Victor Ostrik},
  title        = {On semisimplification of tensor categories},
  year         = {2019},
  howpublished = {arXiv:1801.04409 [math.RT]. Available at \url{https://arxiv.org/abs/1801.04409}}
}

@article{Knop_2006,
  author  = {Knop, Friedrich},
  title   = {A construction of semisimple tensor categories},
  journal = {C. R. Math. Acad. Sci. Paris},
  volume  = {343},
  year    = {2006},
  pages   = {15--18},
  doi     = {10.1016/j.crma.2006.05.009},
  note    = {Available at \url{https://doi.org/10.1016/j.crma.2006.05.009}}
}

@article{Knop_2007,
  author  = {Knop, Friedrich},
  title   = {Tensor envelopes of regular categories},
  journal = {Adv. Math.},
  volume  = {214},
  year    = {2007},
  pages   = {571--617},
  doi     = {10.1016/j.aim.2007.03.001},
  note    = {Available at \url{https://doi.org/10.1016/j.aim.2007.03.001}}
}

@misc{snowden2023fastgrowingtensorcategories,
  author       = {Andrew Snowden},
  title        = {Some fast-growing tensor categories},
  year         = {2023},
  howpublished = {arXiv:2305.18230 [math.RT]. Available at \url{https://arxiv.org/abs/2305.18230}}
}

@misc{snowden2024representationtheorysymmetrygroup,
  author       = {Andrew Snowden},
  title        = {On the representation theory of the symmetry group of the {C}antor set},
  year         = {2024},
  howpublished = {arXiv:2308.06648 [math.RT]. Available at \url{https://arxiv.org/abs/2308.06648}}
}

@misc{kriz_quantum_delannoy_categories,
  author       = {Kriz, Sophie},
  title        = {Quantum {D}elannoy categories},
  year         = {2023},
  howpublished = {Available at \url{https://krizsophie.github.io/QuantumDelannoyCategory23111.pdf}}
}

\end{document}